\documentclass[12pt,a4paper]{article}
\usepackage{latexsym,amssymb,amsmath,mathrsfs,amsthm}
\usepackage{sectsty}
\usepackage{upgreek}
\usepackage{color}
\usepackage{marvosym}
\usepackage{wasysym}
\usepackage{bm}
\usepackage[T1]{fontenc}
\usepackage[anchorcolor=blue,%
 bookmarks=true,%
 bookmarksnumbered=true,%
]{hyperref}
\usepackage{pst-all}
\usepackage{graphicx}

\allsectionsfont{\normalsize\sc\center}

\newtheorem{thm}{Theorem}[section]
\newtheorem{prop}[thm]{Proposition}
\newtheorem{cor}[thm]{Corollary}
\newtheorem{lem}[thm]{Lemma}

\newtheorem{defn}[thm]{Definition}
\newtheorem{remark}[thm]{Remark}
\newtheorem{example}[thm]{Example}
\newtheorem{assumption}{Assumption}

\makeatletter \@addtoreset{equation}{section} \makeatother

\renewcommand{\P}{\mathbb{P}}
\newcommand{\E}{\mathbb{E}}

\newcommand{\R}{\mathbb{R}}

\newcommand{\N}{\mathbb{N}}

\newcommand{\e}{\mathrm{e}}
\newcommand{\ep}{\varepsilon}

\renewcommand{\d}{\mathrm{d}}
\newcommand{\m}{\mathfrak{m} }

\newcommand{\s}{\mathfrak{s}}

\newcommand{\DD}{{\Delta\hspace{-0.29cm}\Delta}^{\rm\! HK}}
\newcommand{\HK}{\mathscr{E}\hspace{-0.3cm}\mathscr{E}^{\rm HK}}

\newcommand{\loc}{{\rm loc}}
\newcommand{\1}{{\bf 1}}
\newcommand{\eps}{{\varepsilon}}
\newcommand{\wt}{\widetilde}
\newcommand{\wh}{\widehat}

\renewcommand{\wh}{\widehat}
\renewcommand{\wt}{\widetilde}

\title{\large\bf Riesz transforms for Dirichlet spaces tamed by distributional curvature lower bounds}
\author{
Syota Esaki\footnote{Syota Esaki ({\tt sesaki@fukuoka-u.ac.jp}) 
Department of Applied Mathematics, Fukuoka University,
Fukuoka 814-0180, Japan. 
Supported in part by JSPS Grant-in-Aid for Scientific Research (C) (No. 23K03158) and fund
fund (No.~215001) from the Central Research Institute of Fukuoka University.}\ \ \ \
Zi Jian Xu\footnote{Zi Jian Xu ({\tt a535218668@yahoo.co.jp}) 
Department of Applied Mathematics, Fukuoka University,
Fukuoka 814-0180, Japan. 
}
\ \ and\ \
Kazuhiro Kuwae\footnote{\Letter Kazuhiro Kuwae ({\tt kuwae@fukuoka-u.ac.jp})
Department of Applied Mathematics, Fukuoka University,
Fukuoka 814-0180, Japan. 
Supported in part by JSPS Grant-in-Aid for Scientific Research (S) (No. 22H04942) and fund (No.~215001) from the Central Research Institute of Fukuoka University.}
}

\date{}
\begin{document}
\maketitle

\begin{abstract}
The notion of tamed Dirichlet space was proposed by Erbar, Rigoni, Sturm and Tamanini \cite{ERST} as a Dirichlet space having a weak form of Bakry-\'Emery curvature lower bounds in distribution sense. 
After their work, Braun~\cite{Braun:Tamed2021} established  a vector calculus for it, in  particular, the space of $L^2$-normed $L^{\infty}$-module describing vector fields, 
$1$-forms, Hessian in $L^2$-sense. 
In this framework,    
we establish the Littlewood-Paley-Stein inequality for $1$-forms as an element of $L^p$-cotangent bundles and boundedness of Riesz transforms, which partially solves
the problem raised by Kawabi-Miyokawa~\cite{KawabiMiyokawa}. 
\end{abstract}

{\it Keywords}:  Strongly local Dirichlet space, tamed space, Bakry-\'Emery condition, smooth measures of (extended) Kato class, smooth measures of Dynkin class, 
$L^p$-normed $L^{\infty}$-module, Hilbert module, tensor product of Hilbert module, $L^2$-differential forms, 
$L^2$-vector fields,  
Littlewood-Paley-Stein inequality, Riesz transform, Bochner Laplacian, Hodge-Kodaira Laplacian, 
measure valued Bochner-Wietzenb\"ock formula,
intertwining property, Wiener space, 
path space with Gibbs measure, RCD-space, weighted Riemannian manifold, 
Riemmanian manifold with boundary, configuration space,  

\section{Statement of Main Theorem}\label{sec:StatementMain}

\subsection{Framework}\label{subsec:Frame}
Let $(M,\tau)$ be a topological Lusin space, i.e., a continuous injective image of a Polish 
space, endowed with a $\sigma$-finite Borel measure $\m$ on $M$ with full topological support. 
Let $(\mathscr{E},D(\mathscr{E}))$ be a quasi-regular symmetric strongly local Dirichlet space on $L^2(M;\m)$ 
and $(P_t)_{t\geq0}$ the associated symmetric sub-Markovian strongly continuous semigroup on $L^2(M;\m)$ (see \cite[Chapter IV, Definition~3]{MR} for the quasi-regularity and see \cite[Theorem~5.1(i)]{Kw:func} for the strong locality). 
Then there exists an $\m$-symmetric special standard process ${\bf X}=(\Omega, X_t, \P_x)$ 
associated with $(\mathscr{E},D(\mathscr{E}))$, i.e. for $f\in L^2(M;\m)\cap \mathscr{B}(M)$, 
$P_tf(x)=\E_x[f(X_t)]$ $\m$-a.e.~$x\in M$ (see \cite[Chapter IV, Section~3]{MR}). 
The $\m$-symmetry of ${\bf X}$, $(P_t)_{t\geq0}$ can be extended to a strongly continuous contraction 
semigroup on $L^p(M;\m)$ for $p\in[1,+\infty[$. Denote by $(D(\Delta_p),\Delta_p)$ its generator on $L^p(M;\m)$ defined by 
\begin{align*}
\left\{\begin{array}{cl} D(\Delta_p)&\displaystyle{:=\left\{f\in L^p(M;\m)\mid \exists
\lim_{t\to0}\frac{P_tf-f}{t} \text{ in } L^p(M;\m)\right\}}, \\ \Delta_pf&\displaystyle{:=\lim_{t\to0}\frac{P_tf-f}{t}\text{ for }f\in D(\Delta_p)}.\end{array}\right.
\end{align*}
An increasing sequence $\{F_n\}$ of closed subsets is called an \emph{$\mathscr{E}$-nest} if 
$\bigcup_{n=1}^{\infty}D(\mathscr{E})_{F_n}$ is $\mathscr{E}_1$-dense in $D(\mathscr{E})$, where 
$D(\mathscr{E})_{F_n}:=\{u\in D(\mathscr{E})\mid u=0\;\m\text{-a.e.~on }M\setminus F_n\}$. 
A subset $N$ is said to be \emph{$\mathscr{E}$-exceptional} or \emph{$\mathscr{E}$-polar} if 
there exists an $\mathscr{E}$-nest $\{F_n\}$ such that $N\subset \bigcap_{n=1}^{\infty}(M\setminus F_n)$. 
For two subsets $A,B$ of $M$, we write $A\subset B$ $\mathscr{E}$-q.e. if $A\setminus B$ is $\mathscr{E}$-exceptional, hence we write $A=B$ $\mathscr{E}$-q.e. if $A\setminus B$ $\mathscr{E}$-q.e. and 
$B\setminus A$ $\mathscr{E}$-q.e. hold. 
It is known that any $\mathscr{E}$-exceptional set is $\m$-negligible and for any $\m$-negligible 
$\mathscr{E}$-quasi-open set is $\mathscr{E}$-exceptional, in particular, for an $\mathscr{E}$-quasi continuous function $u$, $u\geq0$ $\m$-a.e. implies $u\geq0$ $\mathscr{E}$-q.e.   
For a statement $P(x)$ with respect to $x\in M$, we say that $P(x)$ holds $\mathscr{E}$-q.e.~$x\in M$ (simply $P$ holds $\mathscr{E}$-q.e.) if the set $\{x\in M\mid P(x)\text{ holds }\}$ is $\mathscr{E}$-exceptional. 
A subset $G$ of $M$ is called an \emph{$\mathscr{E}$-quasi-open set} if there exists an $\mathscr{E}$-nest $\{F_n\}$ such that $G\cap F_n$ is an open set of $F_n$ with respect to the relative topology on 
$F_n$ for each $n\in\mathbb{N}$. A function $u$ on $M$ is said to be \emph{$\mathscr{E}$-quasi continuous} if there exists an $\mathscr{E}$-nest $\{F_n\}$ such that $u|_{F_n}$ is continuous on $F_n$ for each $n\in\mathbb{N}$.
Denote by $(\mathcal{E},D(\mathcal{E})_e)$ the extended Dirichlet space of $(\mathscr{E},D(\mathscr{E}))$ defined by 
\begin{align*}
\left\{\begin{array}{rl}D(\mathcal{E})_e&=\{u\in L^0(M;\m)\mid \exists \{u_n\}\subset D(\mathscr{E}): \mathscr{E}\text{-Cauchy sequence}\\
&\hspace{6cm}  \text{ such  that }\lim_{n\to\infty}u_n=u\;\m\text{-a.e.}\}, \\
\mathscr{E}(u,u)&=\lim_{n\to\infty}\mathscr{E}(u_n,u_n),\end{array}\right.
\end{align*}
(see \cite[p.~40]{FOT} for details on extended Dirichlet space). 
Denote by $\dot{D}(\mathscr{E})_{\loc}$, the space of functions locally in $D(\mathscr{E})$ in the broad sense defined 
by 
\begin{align*}
\dot{D}(\mathscr{E})_{\loc}&=\{u\in L^0(M;\m)\mid \exists \{u_n\}\subset D(\mathscr{E})\text{ and }\exists \{G_n\}:\mathscr{E}\text{-nest of }\mathscr{E}\text{-quasi-open sets}\\ 
&\hspace{8cm}  \text{ such  that }u=u_n\; \m\text{-a.e.~on }G_n\}.
\end{align*}
Here an increasing sequence $\{G_n\}$ of $\mathscr{E}$-quasi open sets is called an 
\emph{$\mathscr{E}$-nest} if $M=\bigcup_{n=1}^{\infty}G_n$ $\mathscr{E}$-q.e.
For $u\in \dot{D}(\mathscr{E})_{\loc}$ and an $\mathscr{E}$-nest $\{G_n\}$ 
of $\mathscr{E}$-quasi-open sets satisfying $u|_{G_n}\in D(\mathscr{E})|_{G_n}$, we write $\{G_n\}\in \Xi(u)$. It is known that $D(\mathscr{E})\subset D(\mathscr{E})_e\subset \dot{D}(\mathscr{E})_{\loc}$ and 
each $u\in \dot{D}(\mathscr{E})_{\loc}$ admits an $\mathscr{E}$-quasi continuous $\m$-version (see \cite{Kw:func}). 

It is known that for $u,v\in D(\mathscr{E})\cap L^{\infty}(M;\m)$ there exists a unique signed finite Borel 
measure $\mu_{\langle u,v\rangle }$ on $M$ such that 
\begin{align*}
2\int_M\tilde{f}\d\mu_{\langle u,v\rangle }=\mathscr{E}(uf,v)+\mathscr{E}(vf,u)-\mathscr{E}(uv,f)\quad \text{ for }\quad u,v\in D(\mathscr{E})\cap L^{\infty}(M;\m).
\end{align*}
Here $\tilde{f}$ denotes the $\mathscr{E}$-quasi-continuous $\m$-version of $f$ (see \cite[Theorem~2.1.3]{FOT}, \cite[Chapter IV, Proposition~3.3(ii)]{MR}). 
We set $\mu_{\langle f\rangle }:=\mu_{\langle f,f\rangle }$ for $f\in D(\mathscr{E})\cap L^{\infty}(M;\m)$. 
Moreover, 
for $f,g\in D(\mathscr{E})$, there exists a signed finite measure $\mu_{\langle f,g\rangle }$ on $M$ such that 
$\mathscr{E}(f,g)=\mu_{\langle f,g\rangle }(M)$, hence $\mathscr{E}(f,f)=\mu_{\langle f\rangle }(M)$. 
We assume $(\mathscr{E},D(\mathscr{E}))$ admits a carr\'e-du-champ $\Gamma$, i.e. 
$\mu_{\langle f\rangle }\ll\m$ for all $f\in D(\mathscr{E})$ and set $\Gamma(f):=\d\mu_{\langle f\rangle }/\d\m$. 
Then $\mu_{\langle f,g\rangle }\ll\m$ for all $f,g\in D(\mathscr{E})$ and $\Gamma(f,g):=
\d\mu_{\langle f,g\rangle }/{\d\m}\in L^1(M;\m)$ is expressed by $\Gamma(f,g)=\frac14(\Gamma(f+g)-\Gamma(f-g))$ for $f,g\in D(\mathscr{E})$. 

We can extend the singed smooth $\mu_{\langle f,g\rangle }$ or carr\'e-du-champ $\Gamma(f,g)$ from $f,g\in D(\mathscr{E})$ to 
$f,g\in \dot{D}(\mathscr{E})_{\loc}$ by polarization (see \cite[Lemmas~5.2 and 5.3]{Kw:func}). In this case, $\mu_{\langle f,g\rangle }$ (resp.~$\Gamma(f,g)$) 
is no longer a finite signed measure (resp.~an $\m$-integrable function). For $u\in \dot{D}(\mathscr{E})_{\loc}$ and 
an $\mathscr{E}$-nest $\{G_n\}$ 
of $\mathscr{E}$-quasi-open sets satisfying $u|_{G_n}\in D(\mathscr{E})|_{G_n}$, then we can define $\mathscr{E}(u,v)$ 
for $v\in\bigcup_{n=1}^{\infty}D(\mathscr{E})_{G_n}$ by 
\begin{align*}
\mathscr{E}(u,v)=\mu_{\langle u,v\rangle }(M).
\end{align*}
Here $D(\mathscr{E})_{G_n}:=\{u\in D(\mathscr{E})\mid \tilde{u}=0 \;\mathscr{E}\text{-q.e.~on }G_n^c\}$.

Fix $q\in\{1,2\}$. Let $\kappa$ be a signed smooth measure with its Jordan-Hahn decomposition $\kappa=\kappa^+-\kappa^-$. 
We assume that $\kappa^+$ is a smooth measure of Dynkin class ($\kappa^+\in S_D({\bf X})$ in short) 
and $2\kappa^-$ is a smooth measure of extended Kato class ($2\kappa^-\in S_{E\!K}({\bf X})$ in short). 
More precisely, $\nu\in S_D({\bf X})$ (resp.~$\nu\in S_{E\!K}({\bf X})$) if and only if $\nu\in S({\bf X})$ and 
$\m\text{-}\sup_{x\in M}\E_x[A_t^{\nu}]<\infty$ for any/some $t>0$ 
(resp.~$\lim_{t\to0}\m\text{-}\sup_{x\in M}\E_x[A_t^{\nu}]<1$) (see \cite{AM:AF}). 
Here $S({\bf X})$ denotes the family of smooth measures with respect to ${\bf X}$ (see \cite[Chapter VI, Definition~2.3]{MR}, \cite[p.~83]{FOT} for the definition of smooth measures) and $\m$-$\sup_{x\in M}f(x)$ denotes the $\m$-essentially supremum for a function $f$ on $M$.  
For $\nu\in  S({\bf X})$ and set $U_{\alpha}\nu(x):=\E_x\left[\int_0^{\infty}e^{-\alpha t}\d A_t^{\nu}\right]$ with its $\m$-essentially supremum $\|U_{\alpha}\nu\|_{\infty}:=\m$-$\sup_{x\in M}U_{\alpha}\nu(x)$, $\|U_{\alpha}\nu\|_{\infty}<\infty$ for some/any $\alpha>0$ 
(resp.~$\lim_{\alpha\to\infty}\|U_{\alpha}\nu\|_{\infty}=0$, $\lim_{\alpha\to\infty}\|U_{\alpha}\nu\|_{\infty}<1$)
if and only if $\nu\in S_D({\bf X})$ (resp.~$\nu\in  S_{E\!K}({\bf X})$, $\nu\in S_K({\bf X})$).  
For $\nu\in S_D({\bf X})$, the following inequality holds
\begin{align}
\int_M\tilde{f}^2\d\nu\leq\|U_{\alpha}\nu\|_{\infty}\mathscr{E}_{\alpha}(f,f),\quad f\in D(\mathscr{E})\label{eq:StollmannVoigt}
\end{align}
which is called the Stollmann-Voigt's inequality. Here $\mathscr{E}_{\alpha}(f,f):=
\mathscr{E}_{\alpha}(f,f)+\alpha\|f\|_{L^2(M;\m)}^2$. 
Then, for $q\in\{1,2\}$, the quadratic form 
\begin{align*}
\mathscr{E}^{q\kappa}(f):=\mathscr{E}(f)+\langle 2\kappa, \tilde{f}^2\rangle 
\end{align*}
with finiteness domain $D(\mathscr{E}^{q\kappa})=D(\mathscr{E})$ is closed, lower semi-bounded, moreover, 
there exists $\alpha_0>0$ and $C>0$ such that 
\begin{align}
C^{-1}\mathscr{E}_1(f,f)\leq \mathscr{E}^{q\kappa}_{\alpha_0}(f,f)\leq C\mathscr{E}_1(f,f)\quad \text{ for all }\quad f\in D(\mathscr{E}^{q\kappa})=D(\mathscr{E})\label{eq:EquivalencePert}
\end{align}
by \eqref{eq:StollmannVoigt} (see \cite[(3.8)]{CFKZ:Pert} and \cite[Assumption of Theorem~1.1]{CFKZ:GenPert}).   
The Feynman-Kac semigroup $(p_t^{q\kappa})_{t\geq0}$ defined by 
\begin{align*}
p_t^{q\kappa}f(x)=\E_x[e^{-qA_t^{\kappa}}f(X_t)],\quad f\in \mathscr{B}_b(M)
\end{align*}
is $\m$-symmetric, i.e. 
\begin{align*}
\int_M p_t^{q\kappa}f(x)g(x)\m(\d x)=\int_M g(x)p_t^{q\kappa}g(x)\m(\d x)\quad\text{ for all }\quad f,g\in \mathscr{B}_+(M)
\end{align*} 
and coincides with the strongly continuous semigroup $(P_t^{q\kappa})_{t\geq0}$ on $L^2(M;\m)$ associated with 
$(\mathscr{E}^{q\kappa}, D(\mathscr{E}^{q\kappa}))$ (see \cite[Theorem~1.1]{CFKZ:GenPert}). 
Here $A_t^{q\kappa}$ is a continuous additive functional (CAF in short) associated with the signed smooth measure $q\kappa$ under Revuz correspondence. 
Under $\kappa^{-}\in S_{K}({\bf X})$ and $p\in[1,+\infty]$, 
$(p_t^{\kappa})_{t\geq0}$ can be extended 
to be a bounded operator on $L^{p}(M;\m)$ denoted by  $P_t^{\kappa}$ such that 
there exist finite constants $C=C(\kappa)>0, C_{\kappa}\geq0$ depending only on $\kappa^-$ such that 
for every $t\geq0$
\begin{align}
\|P_t^{\kappa}\|_{p,p}\leq Ce^{C_{\kappa}t}.\label{eq:KatoContraction}
\end{align}
Here $C=1$ under $\kappa^-=0$. 
$C_{\kappa}\geq0$  
can be taken to be $0$ under $\kappa^-=0$. 

Let $\Delta^{q\kappa}$  be the $L^2$-generator associated with $(\mathscr{E}^{q\kappa},D(\mathscr{E}))$ 
defined by   
\begin{align}
\left\{\begin{array}{ll} D(\Delta^{q\kappa})&:=\{u\in D(\mathscr{E})\mid \text{there exists } w\in L^2(M;\m)\text{ such that }\\
&\hspace{3cm}\mathscr{E}^{q\kappa}(u,v)=-\int_M wv\,\d\m\quad \text{ for any }\quad v\in D(\mathscr{E})\}, 
\\ \Delta^{q\kappa} u&:=w\quad\text{ for } w\in L^2(M;\m)\quad\text{specified as above,} \end{array}\right. \label{eq:generatorL2}
\end{align}
called the {\it Schr\"odinger operator} with potential $q\kappa$.  
Formally, $\Delta^{q\kappa}$ can be understood as \lq\lq$\Delta^{q\kappa}=\Delta-q\kappa$\rq\rq, 
where $\Delta$ is the $L^2$-generator associated with $(\mathscr{E},D(\mathscr{E}))$.

\begin{defn}[$q$-Bakry-\'Emery condition]
{\rm 
Suppose that $q\in\{1,2\}$, $\kappa^+\in S_D({\bf X})$, $2\kappa^-\in S_{E\!K}({\bf X})$ and $N\in[1,+\infty]$. We say that $(M,\mathscr{E},\m)$ or simply $M$ satisfies the $q$-Bakry-\'Emery condition, briefly 
${\sf BE}_q(\kappa,N)$, if for every $f\in  D(\Delta)$ with $\Delta f\in D(\mathscr{E})$ 
and every nonnegative $\phi\in D(\Delta^{q\kappa})$ with 
$\Delta^{q\kappa}\phi\in L^{\infty}(M;\m)$ (we further impose $\phi\in L^{\infty}(M;\m)$ for $q=2$), we have 
\begin{align*}
\frac{1}{q}\int_M \Delta^{q\kappa}\phi \Gamma(f)^{\frac{q}{2}}\d\m-\int_M \phi
\Gamma(f)^{\frac{q-2}{2}}
\Gamma(f,\Delta f)\d\m
\geq \frac{1}{N}\int_M \phi\Gamma(f)^{\frac{q-2}{2}}(\Delta f)^2\d\m.
\end{align*}
The latter term is understood as $0$ if $N=\infty$.
}
\end{defn}

\begin{assumption}\label{asmp:Tamed}
{\rm We assume that $M$ satisfies ${\sf BE}_{2}(\kappa,N)$ condition for a given signed smooth measure  
$\kappa$ with $\kappa^+\in S_D({\bf X})$ and $2\kappa^-\in S_{E\!K}({\bf X})$ and $N\in[1,\infty]$.
}
\end{assumption}
Under Assumption~\ref{asmp:Tamed}, we say that $(M,\mathscr{E},\m)$ or simply $M$ is {\it tamed}. 
In fact, under $\kappa^+\in S_D({\bf X})$ and $2\kappa^-\in S_{E\!K}({\bf X})$, the condition ${\sf BE}_2(\kappa,\infty)$ is {\it equivalent} 
to ${\sf BE}_1(\kappa,\infty)$ (see Lemma~\ref{lem:BakryEmeryEquivalence} below), in particular, the heat flow $(P_t)_{t\geq0}$ 
satisfies 
\begin{align}
\sqrt{\Gamma(P_tf)}\leq P_t^{\kappa}\sqrt{\Gamma(f)}\quad\m\text{-a.e.~for any }f\in D(\mathscr{E})\quad \text{ and }\quad t\geq0\label{eq:gradCont}
\end{align}
(see \cite[Definition~3.3 and Theorem~3.4]{ERST}).
The inequality \eqref{eq:gradCont} plays a crucial role in our paper. 
Note that our condition $\kappa^+\in S_D({\bf X})$, $\kappa^-\in S_{E\!K}({\bf X})$ (resp.~$\kappa^+\in S_D({\bf X})$, $2\kappa^-\in S_{E\!K}({\bf X})$) is stronger than the $1$-moderate (resp.~$2$-moderate) condition treated in \cite{ERST} for the definition of tamed space.  
The $\m$-symmetric Markov process ${\bf X}$ treated in our paper may not be conservative in general. Under Assumption~\ref{asmp:Tamed}, 
a sufficient condition ({\it intrinsic stochastic completeness} called in \cite{ERST}) 
for conservativeness of ${\bf X}$ is discussed in \cite[Section~3.3]{ERST}, in 
particular, under $1\in D(\mathscr{E})$ and Assumption~\ref{asmp:Tamed}, we have the conservativeness of ${\bf X}$.

\bigskip

Our main theorem under Assumption~\ref{asmp:Tamed} is the following:

\begin{thm}\label{thm:main3}
Let $p\in]1,+\infty[$ and $\alpha> C_{\kappa}$. For $f\in L^p(M;\m)\cap L^2(M;\m)$, we define the Riesz operator 
$R_{\alpha}(\Delta)$ by 
\begin{align*}
R_{\alpha}(\Delta)f:=\Gamma((\alpha-\Delta)^{-\frac{1}{2}}f)^{\frac12}.
\end{align*}
Suppose $\kappa^-\in S_K({\bf X})$ and $p\in[2,+\infty[$, or $\kappa_{\ll}^-=-R\m$ with $R\geq0$, $\kappa_{\perp}^-=0$ and $p\in ]1,2]$. 
Then, 
$R_{\alpha}(\Delta)$ can be extended to a (non-linear) bounded operator on $L^p(M;\m)$.   
The operator norm $\|R_{\alpha}(\Delta)\|_{p,p}$ depends only on $\kappa,p$ and $C_{\kappa}$. 
\end{thm}

Fix $p\in ]1,\infty[$. We 
define the $(1,p)$-Sobolev space $(H^{1,p}(M),\|\cdot\|_{H^{1,p}})$ as follows: 
\begin{defn}
{\rm Let $(\mathscr{E},D(\mathscr{E}))$ be a quasi-regular strongly local Dirichlet form on $L^2(M;\m)$ admitting 
carr\'e du champ $\Gamma$. 
Set $\mathscr{D}_{1,p}:=\{u\in L^p(M;\m)\cap D(\mathscr{E})\mid \Gamma(u)^{\frac12}\in L^p(M;\m)\}$.  
For $u\in \mathscr{D}_{1,p}$, we define $\|u\|_{H^{1,p}}^p:=\|u\|_{L^p(M;\m)}^p+\|\Gamma(u)^{\frac12}\|_{L^p(M;\m)}^p$. Then $\|\cdot\|_{H^{1,p}}$ forms a norm on $\mathscr{D}_{1,p}$. 
We now define the Sobolev space $H^{1,p}(M)$ as the $\|\cdot\|_{H^{1,p}}$-completion of $\mathscr{D}_{1,p}$: 
\begin{align*}
\left\{\begin{array}{cl}H^{1,p}(M)&:=\overline{\mathscr{D}_{1,p}}^{\|\cdot\|_{H^{1,p}}},\\
\|u\|_{H^{1,p}}&:=\left(\|u\|_{L^p(M;\m)}^p+\|\Gamma(u)^{\frac12}\|_{L^p(M;\m)}^p\right)^{\frac{1}{p}}\quad\text{ for }\quad u\in H^{1,p}(M).\end{array}\right.
\end{align*}
More precisely, for $u\in H^{1,p}(M)$, we can extend the carr\'e du champ $\Gamma$ to $u$ with $\Gamma(u)^{\frac12}\in L^p(M;\m)$. 
We set $D(\mathscr{E}^{\,p}):=H^{1,p}(M)$ and  $\mathscr{E}^{\,p}(u):=\|\Gamma(u)^{\frac12}\|_{L^p(M;\m)}^p$ for $u\in D(\mathscr{E}^p)$, which is called the $p$-energy of $u\in H^{1,p}(M)$. 
}
\end{defn}
It is proved in \cite[Lemma~3.1]{Kw:SobolevSpace} that $(H^{1,p}(M),\|\cdot\|_{H^{1,p}})$ is a uniformly convex Banach space and consequently it is reflexive. 

\bigskip

The following theorems are essentially shown in the framework of weighted Riemannian manifolds $(M,g,e^{-\phi}v_g)$ (without boundary) with Witten Laplacian $L:=\Delta-\langle \nabla\phi,\cdot\rangle $. For the completeness, we restate them in our framework under Assumption~\ref{asmp:Tamed}. 

\begin{thm}\label{thm:NormEqui}
Fix $\alpha>C_{\kappa}$.  
The following statements are equivalent, and one of them (hence all) holds  under 
$\kappa^-\in S_K({\bf X})$ and $p\in[2,+\infty[$, or under $\kappa_{\ll}^-=-R\m$ with $R\geq0$, $\kappa_{\perp}^-=0$ and $p\in]1,2]$.
\begin{enumerate}
\item[\rm(1)] The Riesz transform $R_{\alpha}(\Delta)$ is bounded in $L^p(M;\m)$ for all $p\in]1,+\infty[$. That is, for every $p\in]1,+\infty[$, there exists a constant $\|R_{\alpha}(\Delta)\|_{p,p}\in]0,+\infty[$ such that 
\begin{align*}
\|R_{\alpha}(\Delta)f\|_{L^p(M;\m)}\leq\|R_{\alpha}(\Delta)\|_{p,p}\|f\|_{L^p(M;\m)},\quad f\in {\rm Test}(M)_{fs},
\end{align*}
where ${\rm Test}(M)_{fs}$ is the family of test functions with finite supports {\rm(}see Definition~\ref{def:TestFunc} below{\rm)}.
\item[\rm(2)] The Sobolev norms $\|\cdot\|_{H^{1,p}}$, 
$\|\cdot\|_{1,p}$ and $\| |\cdot| \|_{1,p}$ are equivalent, where 
\begin{align*}
\|f\|_{1,p}:&=\sqrt{\alpha}\|f\|_{L^p(M;\m)}+\|\Gamma(f)^{\frac12}\|_{L^p(M;\m)},\\
\| |f|\|_{1,p}:&=\|(\alpha-\Delta_p)^{\frac12}f\|_{L^p(M;\m)}
\end{align*}
for $f\in H^{1,p}(M)$.
\item[\rm(3)] The Sobolev space $H^{1,p}(M)$ coincides with the Sobolev space $W^{1,p}(M)$ defined by $W^{1,p}(M):=V_{\alpha}(L^p(M;\m))$. Here $V_{\alpha}$ is the bounded linear operator on $L^p(M;\m)$ defined by $V_{\alpha}f:=\frac{1}{\sqrt{\pi}}\int_0^{\infty}e^{-\alpha t}t^{-\frac12}P_tf \d t=(\alpha-\Delta_p)^{-\frac12}f\in L^p(M;\m)$ satisfying $\sqrt{\alpha}\|V_{\alpha}f\|_{L^p(M;\m)}\leq\|f\|_{L^p(M;\m)}$. 
\item[\rm(4)] The $(1,p)$-capacities $c_{1,p}$ and $C_{1,p}$ are equivalent, where 
\begin{align*}
c_{1,p}(O):&=\inf\{\|f\|_{1,p}^p\mid f\in H^{1,p}(M), f\geq0\;\m\text{-a.e.~on }M, f\geq1\;\m\text{-a.e.~on }O\},\\
C_{1,p}(O):&=\inf\{\|f\|_{1,p}^p\mid f\in W^{1,p}(M), f\geq0\;\m\text{-a.e.~on }M, f\geq1\;\m\text{-a.e.~on }O\},
\end{align*} 
for all open sets $O\subset X$, and for all $A\subset X$,
\begin{align*}
c_{1,p}(A):&=\inf\{c_{1,p}(O)\mid f\in H^{1,p}(M), A\subset O\subset X, O\text{ is open}\},\\
C_{1,p}(A):&=\inf\{C_{1,p}(O) \mid f\in W^{1,p}(M), A\subset O\subset X, O\text{ is open}\}.
\end{align*}
\end{enumerate}
\end{thm}
\begin{proof}[{\bf Proof}] 
The proof is quite similar with the proof of \cite[Theorem~3.1]{Xdli:RieszTrans}. We omit it. 
\end{proof} 
\begin{thm}\label{thm:OrdEqui}
Let $p\in]1,+\infty[$. Suppose that the Riesz transform $R_{\alpha}(\Delta)$ is bounded in $L^p(M;\m)$ for some $\alpha\geq0$. 
Then $R_{\beta}(\Delta)$ is bounded in $L^p(M;\m)$ for all $\beta>\alpha\wedge0$. 
Moreover, there exists a constant $A_p>0$ such that 
\begin{align*}
\|R_{\beta}(\Delta)\|_{p,p}\leq A_p(\sqrt{\alpha\beta^{-1}}\lor 1)\|R_{\alpha}(\Delta)\|_{p,p}.
\end{align*}
\end{thm}
\begin{proof}[{\bf Proof}] 
The proof is quite similar with the proof of \cite[Theorem~3.2]{Xdli:RieszTrans}. We omit it. 
\end{proof} 

\begin{thm}\label{thm:0OrdEqui}
Let $p\in]1,+\infty[$. Suppose that the Riesz transform $R_{\alpha}(\Delta):=\Gamma(\nabla (\alpha-\Delta_p)f)^{\frac12}$ is bounded in $L^p(M;\m)$ for some $\alpha>0$ and the bottom of spectrum of $-\Delta$ is strictly positive, i.e.,
\begin{align*}
\lambda_2(-\Delta):=\inf_{f\in D(\mathscr{E})\setminus\{0\}}\frac{\mathscr{E}(f,f)}{\;\quad\|f\|_{L^2(M;\m)}^2}.
\end{align*}
Then the Riesz $R_{0}(\Delta):=\Gamma(\nabla (-\Delta_p)f)^{\frac12}$ is bounded in $L^p(M;\m)$. 
\end{thm}
\begin{proof}[{\bf Proof}] 
The proof is quite similar with the proof of \cite[Theorem~3.3]{Xdli:RieszTrans}. We omit it. 
\end{proof} 

Under Assumption~\ref{asmp:Tamed}, we have also the following:
\begin{thm}\label{thm:OpeNorm}
Let $p\in]1,+\infty[$ and $\alpha>C_{\kappa}$. Suppose $\kappa^-\in S_K({\bf X})$ under $p\in[2,+\infty[$, or $\kappa_{\ll}=-R\m$, $\kappa_{\perp}=0$ under $p\in]1,2]$. Then we have the following:
\begin{enumerate}
\item[\rm(1)] There exists $A_p>0$ such that for any $t>0$ and $f\in L^p(M;\m)$
\begin{align}
\|\Gamma(P_tf)^{\frac12}\|_{L^p(M;\m)}\leq A_p\|R_{\alpha}(\Delta)\|_{p,p}(\alpha^{1/2}+t^{-1/2})
\|f\|_{L^p(M;\m)}.\label{eq:GradEsti1}
\end{align}
\item[\rm(2)] There exists $A_p>0$ such that for any $\eps\in]0,1[$, $t>0$ and $f\in L^p(M;\m)$
\begin{align}
\|\Gamma(P_tf)^{\frac12}\|_{L^p(M;\m)}\leq A_p
\frac{\|R_{\alpha}(\Delta)\|_{p,p}}{\sqrt{\eps}}\cdot\frac{e^{\alpha \eps t}}{\sqrt{t}}
\|f\|_{L^p(M;\m)}.\label{eq:GradEsti2}
\end{align}
\item[\rm(3)] If the heat semigroup $P_t$ is hypercontractive, i.e., $\|P_t\|_{p,q(t)}=m_{p,q}(t)$, then 
there exists $C_p>0$ such that for any $\eps\in]0,1[$. $t>0$ and $f\in L^p(M;\m)$
\begin{align}
\|\Gamma(P_tf)^{\frac12}\|_{L^p(M;\m)}\leq C_p
\frac{\|R_{\alpha}(\Delta)\|_{p,p}}{\sqrt{\eps}}\cdot\frac{e^{\alpha \eps t}m_{p,q}((1-\eps)t)}{\sqrt{t}}
\|f\|_{L^p(M;\m)}.\label{eq:GradEsti3}
\end{align}
\item[\rm(4)] There exists $A_p>0$ such that, for any $\eps\in]0,1[$, $t>0$ and $f\in L^p(M;\m)$
\begin{align}
\|\Gamma(Q_t^{(\alpha)}f)^{\frac12}\|_{L^p(M;\m)}\leq C_p
\frac{\|R_{\alpha}(\Delta)\|_{p,p}}{\eps}\cdot\frac{e^{-\sqrt{\alpha} \eps t}}{t}
\|f\|_{L^p(M;\m)},\label{eq:GradEsti4}
\end{align} 
where $Q_t^{(\alpha)}f:=e^{-t\sqrt{\alpha-\Delta}}f$ is the Cauchy semigroup. 
\end{enumerate} 
\end{thm}
\begin{proof}[{\bf Proof}] 
The proof is quite similar with the proof of \cite[Theorem~3.4]{Xdli:RieszTrans}. We omit it. 
\end{proof}

\section{Smooth measures of extended Kato class}\label{sec:ExtendeKato}
In this section, we summarize results on smooth measures of extended Kato class. 
Denote by $S({\bf X})$ the family of (non-negative) smooth measures on $(M,\mathscr{B}(M))$ (see \cite[p.~83]{FOT} for the precise 
definition of smooth measures). Then there exists a PCAF $A^{\nu}$ admitting exceptional set satisfying 
\begin{align*}
\E_{h\m}\left[\int_0^{\infty}e^{-\alpha t}f(X_t)\d A_t^{\nu} \right]=\E_{f\nu}\left[\int_0^{\infty}e^{-\alpha t}f(X_t)\d t \right],
\quad f,h\in\mathscr{B}_+(M)
\end{align*}  
(see \cite[(5.14)]{FOT}). 
A smooth measure $\nu$ is said to be of \emph{Dynkin class} (resp.~\emph{Kato class}) if $\|\E_{\cdot}[A_t^{\nu}]\|_{L^{\infty}(M;\m)}<\infty$ for some/all $t>0$ (resp.~$\|\E_{\cdot}[A_t^{\nu}]\|_{L^{\infty}(M;\m)}=0$). 
 A smooth measure $\nu$ is said to be of \emph{extended Kato class} if $\lim_{t\to0}\|\E_x[A_t^{\nu}]\|_{L^{\infty}(M;\m)}<1$. Denote by $S_D({\bf X})$ (resp.~$S_K({\bf X})$, $S_{E\!K}({\bf X})$) the family of 
(non-negative)  smooth measures of Dynkin class (resp.~Kato class, extended Kato class). For $\nu\in S_D({\bf X})$, 
$x\mapsto \E_x[A_t^{\nu}]$ is $\mathscr{E}$-quasi-continuous, hence $\|\E_{\cdot}[A_t^{\nu}]\|_{L^{\infty}(M;\m)}$ actually coincides with a quasi-essentially supremum: 
\begin{align*}
\|\E_{\cdot}[A_t^{\nu}]\|_{L^{\infty}(M;\m)}=q\text{-}\sup_{x\in M}\E_x[A_t^{\nu}]:=\inf_{N:\mathscr{E}\text{-exceptional}}\sup_{x\in N^c}\E_x[A_t^{\nu}].
\end{align*}
For $\nu\in S_{E\!K}({\bf X})$, there exist positive constants $C=C(\nu)$ and $C_{\nu}>0$ such that 
\begin{align}
\|\E_{\cdot}[e^{A_t^{\nu}}]\|_{L^{\infty}(M;\m)}\leq Ce^{C_{\nu}t}.\label{eq:unifKato}
\end{align}
Indeed, we can choose $C(\nu):=(1-\|\E_{\cdot}[A_{t_0}^{\nu}]\|_{L^{\infty}(M;\m)})^{-1}$ and $C_{\nu}:=\frac{1}{t_0}\log C(\nu)$. 
for some sufficiently small $t_0\in]0,1[$ satisfying $\|\E_{\cdot}[A_{t_0}^{\nu}]\|_{L^{\infty}(M;\m)}<1$ (see \cite[Proof of Theorem~2.2]{Sznitzman}). Denote by $P_t^{-\nu}f(x)$ the Feynman-Kac semigroup defined by 
$P_t^{-\nu}f(x):=\E_x[e^{A_t^{\nu}}f(X_t)]$, $f\in \mathscr{B}_+(M)$. \eqref{eq:unifKato} is nothing but 
$\|P_t^{-\nu}\|_{\infty,\infty}\leq Ce^{C_{\nu}t}$, moreover, $\|P_t^{-\nu}\|_{p,p}\leq Ce^{C_{\nu}t}$ for $p\in[2,+\infty]$. 
Here $\|P_t^{-\nu}\|_{p,p}:=\sup\{\|P_t^{-\nu}f\|_{L^p(M;\m)}\mid f\in L^p(M;\m), \|f\|_{L^p(M;\m)}=1\}$ stands for the operator norm from 
$L^p(M;\m)$ to $L^p(M;\m)$. 
For each $\beta\geq1$ satisfying $\beta\nu\in S_{E\!K}({\bf X})$, we can choose
\begin{align}
C_{\beta\nu}\in]0,C_{\nu}]\quad\text{ and }\quad C(\nu)=C(\beta\nu).\label{eq:KatoCoincidence}
\end{align}
Indeed $C_{\nu}$ and $C(\nu)$ can be given by 
\begin{align*}
C_{\nu}=\frac{1}{t_0}\log\left(1-\|\E_{\cdot}[A_{t_0}^{\nu}]\|_{\infty} \right)^{-1},\quad 
C(\nu)=\left(1-\|\E_{\cdot}[A_{t_0}^{\nu}]\|_{\infty} \right)^{-1}
\end{align*}
for some small $t_0\in]0,1]$ satisfying $\|\E_{\cdot}[A_{t_0}^{\nu}]\|_{\infty}<1$ (see \cite[Proof of Theorem~2.2]{Sznitzman}). 
If we choose $s_0\in]0,t_0]$ 
satisfying $\|\E_{\cdot}[A_{s_0}^{\nu}]\|_{\infty}=\|\E_{\cdot}[A_{t_0}^{\beta\nu}]\|_{\infty}<1$, then we see $C_{\beta\nu}\leq C_{\nu}$ and $C(\nu)=C(\beta\nu)$ for such selection of $s_0$ (resp.~$t_0$) to $C_{\nu}$ (resp.~$C_{\beta\kappa}$). 

When $\nu=-R\m$ for a constant $R\in \R$, $C_{\nu}$ (resp.~$C(\nu)$) can be taken to be $R\lor 0$ (resp.~$1$).

The following lemma might be known for experts: 
\begin{lem}\label{lem:ExtendeddKato1}
For $\nu\in S_{E\!K}({\bf X})$ and $\alpha>C_{\nu}$,
\begin{align*}
\left\|\E_{\cdot}\left[\int_0^{\infty}e^{-\alpha t}\d A_t^{\nu} \right]\right\|_{L^{\infty}(M;\m)}< 1.
\end{align*}
\end{lem}
\begin{proof}[{\bf Proof}] 
Take any sufficiently small $t_0\in]0,T[$. 
By use of Markov property, 
\begin{align*}
\E_x\left[\int_0^{\infty}e^{-\alpha t}\d A_t^{\nu} \right]&=\sum_{n=0}^{\infty}\E_x\left[\int_{nt_0}^{(n+1)t_0}e^{-\alpha t}\d A_t^{\nu} \right]\\
&=\sum_{n=0}^{\infty}\E_x\left[\int_0^{t_0}e^{-\alpha(nt_0+t)}\d A_t^{\nu}\circ \theta_{nt_0}\right]\\
&=\sum_{n=0}^{\infty}e^{-\alpha nt_0}\E_x\left[\E_{X_{nt_0}}\left[\int_0^{t_0}e^{-\alpha t}\d A_t^{\nu} \right] \right]\\
&\leq \sum_{n=0}^{\infty}e^{-\alpha nt_0}q\text{-}\sup_{x\in M}\E_x\left[A_{t_0}^{\nu} \right]\\
&=\frac{1}{1-e^{-\alpha t_0}}\|\E_{\cdot}[A_{t_0}^{\nu}]\|_{L^{\infty}(M;\m)}\\
&\leq\frac{1-e^{-C_{\nu}(t_0)t_0}}{1-e^{-\alpha t_0}}.
\end{align*}
Then 
\begin{align*}
\left\|\E_{\cdot}\left[\int_0^{\infty}e^{-\alpha t}\d A_t^{\nu} \right]\right\|_{L^{\infty}(M;\m)}\leq 
\frac{1-e^{-C_{\nu}(t_0)t_0}}{1-e^{-\alpha t_0}}<1
\end{align*}
under $\alpha>C_{\nu}(t_0)$. 
\end{proof} 
\begin{cor}\label{cor:Equivalence}
For $\nu\in S_{E\!K}({\bf X})$ and $\alpha>C_{\nu}$,
\begin{align}
(1-\|U_{\alpha}\nu\|_{\infty})\mathscr{E}(f,f)\leq \mathscr{E}^{-\nu}_{\alpha \|U_{\alpha}\nu\|_{\infty}}(f,f)\leq\mathscr{E}_{\alpha \|U_{\alpha}\nu\|_{\infty}}(f,f)\quad \text{ for }\quad f\in D(\mathscr{E}),\label{eq:Equivalence}
\end{align}
where $U_{\alpha}\nu(x):=\E_x\left[\int_0^{\infty}e^{-\alpha t}\d A_t^{\nu} \right]$. In particular, 
under $\kappa^+\in S_D({\bf X})$ and $\kappa^-\in S_{E\!K}({\bf X})$, for $\alpha>C_{\kappa}$, we have 
\begin{align}
(1-\|U_{\alpha}\kappa^-\|_{\infty})\mathscr{E}(f,f)\leq \mathscr{E}^{\kappa}_{\alpha \|U_{\alpha}\kappa^-\|_{\infty}}(f,f)\leq C\mathscr{E}_1(f,f)\quad \text{ for }\quad f\in D(\mathscr{E}),\label{eq:Equivalence*}
\end{align}
where $C:=\max\left\{(1+\|U_{\alpha}\kappa^+\|_{\infty}),\alpha\|U_{\alpha}|\kappa|\|_{\infty} \right\}$.
\end{cor}
\section{Test functions}

The following lemmas hold. 
\begin{lem}[{\cite[Proposition~6.7]{ERST}}]\label{lem:boundedEst}
Under Assumption~\ref{asmp:Tamed}, ${\sf BE}_2(-\kappa^-,+\infty)$ holds. Moreover, for every $f\in L^2(M;\m)\cap L^{\infty}(M;\m)$ and $t>0$, it holds 
\begin{align}
\Gamma(P_tf)\leq\frac{1}{2t}\|P_t^{-\kappa^-}\|_{\infty,\infty}\cdot\|f\|_{L^{\infty}(M;\m)}^2.\label{eq:BoundedEst}
\end{align}
In particular, for $f\in L^2(M;\m)\cap L^{\infty}(M;\m)$, then $P_tf\in {\rm Test}(M)$. 
\end{lem}

\begin{lem}[{\cite[Theorems~3.4 and 3.6, Proposition~3.7 and Theorem~6.10]{ERST}}]\label{lem:BakryEmeryEquivalence}\quad\\
Under $\kappa^+\in S_D({\bf X})$ and $2\kappa^-\in S_{E\!K}({\bf X})$,  
the condition ${\sf BE}_2(\kappa,\infty)$ is equivalent 
to ${\sf BE}_1(\kappa,\infty)$. In particular, we have \eqref{eq:gradCont}. 
\end{lem}

Under Assumption~\ref{asmp:Tamed}, we now introduce ${\rm Test}(M)$ the set of test functions: 

\begin{defn}\label{def:TestFunc}
{\rm 
Let $(M,\mathscr{E},\m)$ be a tamed space. Let us define the set of {\it test functions} by 
\begin{align*}
{\rm Test}(M):&=\{f\in D(\Delta)\cap L^{\infty}(M;\m)\mid \Gamma(f)\in L^{\infty}(M;\m),\Delta f\in D(\mathscr{E})\},\\
{\rm Test}(M)_{fs}:&=\{f\in {\rm Test}(M)\mid \m({\rm supp}[f])<\infty\}.
\end{align*}
}
\end{defn}
It is easy to see that ${\rm Test}(M)_{fs}\subset \mathscr{D}_{1,p}\subset H^{1,p}(M)$ for any $p\in]1,+\infty[$. 
\begin{lem}[{\cite[Lemma~3.2]{Sav14}}]\label{lem:algebra}
Under Assumption~\ref{asmp:Tamed}, for every $f\in {\rm Test}(M)$, we have 
$\Gamma(f)\in D(\mathscr{E})\cap L^{\infty}(M;\m)$ and there exists $\mu=\mu^+-\mu^-$ with 
$\mu^{\pm}\in D(\mathscr{E})^*$ such that 
\begin{align}
-\mathscr{E}^{2\kappa}(u,\varphi)=\int_M\tilde{\varphi}\,\d \mu\quad\text{ for all }\quad \varphi\in D(\mathscr{E}).
\end{align}
Moreover, ${\rm Test}(M)$ is an algebra, i.e., for $f,g\in {\rm Test}(M)$, $fg\in {\rm Test}(M)$, 
if further {\boldmath$f$}$\in {\rm Test}(M)^n$, then $\Phi(${\boldmath$f$}$)\in {\rm Test}(M)$ for every smooth 
function $\Phi:\R^n\to\R$ with $\Phi(0)=0$.
\end{lem}

\begin{lem}\label{lem:DensenessTestFunc}
${\rm Test}(M)\cap L^p(M;\m)$ is dense in $L^p(M;\m)\cap L^2(M;\m)$ both in $L^p$-norm 
and in $L^2$-norm. In particular, ${\rm Test}(M)$ is dense in $(\mathscr{E},D(\mathscr{E}))$. 
\end{lem}
\begin{proof}[{\bf Proof}] 
Take $f\in L^p(M;\m)\cap L^2(M;\m)$. We may assume $f\in L^{\infty}(M;\m)$, because 
$f$ is $L^p$(and also $L^2$)-approximated by a sequence $\{f^k\}$ of 
$L^p(M;\m)\cap L^2(M;\m)\cap L^{\infty}(M;\m)$-functions defined 
by $f^k:=(-k)\lor f\land k$. If $f\in L^p(M;\m)\cap L^2(M;\m)\cap L^{\infty}(M;\m)$, 
$P_tf\in {\rm Test}(M)$ by Lemma~\ref{lem:boundedEst} and 
$\{P_tf\}\subset {\rm Test}(M)\cap L^p(M;\m)$ converges to 
$f$ in $L^p$ and in $L^2$ as $t\to0$. If $f\in D(\mathscr{E}))$, then $f$ can be approximated by 
$\{P_tf^k\}$ in $(\mathscr{E},D(\mathscr{E}))$. This shows the last statement.
\end{proof} 
\begin{remark}\label{rem:TestFunction}
{\rm 
As proved above, ${\rm Test}(M)$ forms an algebra and dense in $(\mathscr{E},D(\mathscr{E}))$ under 
Assumption~\ref{asmp:Tamed}. However, ${\rm Test}(M)$ is not necessarily a subspace of $C_b(M)$. 
When the tamed space comes from ${\rm RCD}$-space, the Sobolev-to-Lipschitz property of RCD-spaces ensures 
${\rm Test}(M)\subset C_b(M)$. 
}
\end{remark}

\section{Vector space calculus for tamed space}
\subsection{$L^{\infty}$-module}\label{subsec:normedModule}

We need the notion of $L^p$-normed $L^{\infty}(M;\m)$-modules. 
\begin{defn}[$L^{\infty}$-module]\label{df:NormedModule}
{\rm 
Given $p\in[1,+\infty]$, a real Banach space $(\mathscr{M},\|\cdot\|_{\mathscr{M}})$, or simply, $\mathscr{M}$ is called an \emph{$L^p$-normed $L^{\infty}$-module} (over $(M,\m)$) if it satisfies 
\begin{enumerate}
\item[(a)] a bilinear map $\cdot$ : $L^{\infty}(M;\m)\times\mathscr{M}\to\mathscr{M}$ satisfying 
\begin{align*}
(fg)\cdot v&=f\cdot(gv),\\
\1_M\cdot v&=v,
\end{align*} 
\item[(b)] a nonnegatively valued map $|\cdot|_{\m}:\mathscr{M}\to L^p(M;\m)$ such that 
\begin{align*}
|f\cdot v|_{\m}&=|f||v|_{\m}\quad \m\text{-a.e.},\\
\|v\|_{\mathscr{M}}&=\||v|_{\m}\|_{L^p(M;\m)},
\end{align*}
for every $f,g\in L^{\infty}(M;\m)$ and $v\in\mathscr{M}$. If only (a) is satisfied, we call 
$(\mathscr{M},\|\cdot\|_{\mathscr{M}})$ or simply $\mathscr{M}$ an $L^{\infty}(M;\m)$-module. 
\end{enumerate} 
Throughout this paper, we always assume that for every $v\in \mathscr{M}$ its point-wise norm $|v|_{\m}$ is Borel measurable.  
We call $\mathscr{M}$ an \emph{$L^{\infty}$-module} if it is  $L^p$-normed 
for some $p\in[1,+\infty]$. $\mathscr{M}$ is called \emph{separable} 
 if it is a separable Banach space. We call $v\in\mathscr{M}$ is 
($\m$-essentially) bounded if $|v|_{\m}\in L^{\infty}(M;\m)$. $\mathscr{M}$ is called \emph{Hilbert module} if is an $L^2$-normed $L^{\infty}$-module, in this case, the point-wise norm $|\cdot|_{\m}$ satisfies a point-wise $\m$-a.e.~parallelogram identity, hence it induces a \emph{point-wise scalar product} $\langle \cdot,\cdot\rangle_{\m}: \mathscr{M}\times \mathscr{M}\to L^1(M;\m)$ which is $L^{\infty}$-bilinear, $\m$-a.e.~nonnegative definite, local in both components, satisfies the point-wise $\m$-a.e.~Cauchy-Schwarz inequality, and reproduces the Hilbertian scalar product on $\mathscr{M}$ by integration with respect to $\m$.
}
\end{defn}
\begin{defn}[Dual module]\label{df:DualModule}
{\rm Let $\mathscr{M}$ and $\mathscr{N}$ be $L^p$-normed $L^{\infty}$-module. Denote both point-wise norms by $|\cdot|_{\m}$. A map $T:\mathscr{M}\to\mathscr{N}$ is called \emph{module morphism} 
if it is a bounded linear map in the sense of functional analysis and 
\begin{align}
T(f\,v)=f\,T(v)\label{eq:morphism}
\end{align}
for every $v\in\mathscr{M}$ and every $f\in L^{\infty}(M;\m)$. The set of all module morphisms 
is written ${\rm Hom}(\mathscr{M},\mathscr{N})$ and equipped with the usual operator norm $\|\cdot\|_{\mathscr{M},\mathscr{N}}$. We call $\mathscr{M}$ and $\mathscr{N}$ \emph{isomorphic} (as $L^{\infty}$-module) if there exists  $T\in{\rm Hom}(\mathscr{M},\mathscr{N})$ and $S\in {\rm Hom}(\mathscr{N},\mathscr{M})$ such that $T\circ S={\rm Id}_{\mathscr{N}}$ and $S\circ T={\rm Id}_{\mathscr{M}}$. 
Any such $T$ is called \emph{module isomorphism}. If  in addition, such a $T$ is a norm isometry, it is called \emph{module isometric isomorphism}. In fact, by \eqref{eq:morphism} every module isometric isomorphism $T$ preserves pointwise norms $\m$-a.e., i.e. for every $v\in\mathscr{M}$, 
\begin{align*}
|T(v)|_{\m}=|v|_{\m}\quad \m\text{-a.e.}
\end{align*}
The dual module $\mathscr{M}^*$ is defined by 
\begin{align*}
\mathscr{M}^*:={\rm Hom}(\mathscr{M}, L^1(M;\m))
\end{align*}
and will be endowed with the usual operator norm. The point-wise paring between 
$v\in\mathscr{M}$ and $L\in\mathscr{M}^*$ is denoted by $L(v)\in L^1(M;\m)$. If $\mathscr{M}$ is $L^p$-normed, then $\mathscr{M}^*$ is an $L^q$-normed $L^{\infty}$-module, where $p,q\in[1,+\infty]$ with $1/p+1/q=1$ (see \cite[Proposition~1.2.14]{Gigli:NonSmoothDifferentialStr}) with naturally defined multiplications and, and by a slight abuse of notation, point-wise norm is given by 
\begin{align}
|L|_{\m}:=\m\text{\rm-esssup}\{|L(v)|\mid v\in\mathscr{M}, |v|_{\m}\leq1 \;\m\text{-a.e.}\}.\label{eq:DualPointNorm}
\end{align}
By \cite[Corollary~1.2.16]{Gigli:NonSmoothDifferentialStr}, if $p<\infty$,
\begin{align*}
|v|_{\m}=\m\text{\rm-esssup}\{|L(v)|\mid L\in\mathscr{M}, |L|_{\m}\leq1 \;\m\text{-a.e.}\}
\end{align*}
for every $v\in\mathscr{M}$. Moreover, under $p<\infty$, in the sense of functional analysis 
$\mathscr{M}^*$ and the dual Banach space $\mathscr{M}'$ of $\mathscr{M}$ are isometrically isomorphic 
\cite[Proposition~1.2.13]{Gigli:NonSmoothDifferentialStr}. In this case, the natural point-wise paring map $\mathscr{I}:\mathscr{M}\to\mathscr{M}^{**}$, where 
$\mathscr{M}^{**}:={\rm Hom}(\mathscr{M}^*,L^1(M;\m))$, belongs to ${\rm Hom}(\mathscr{M},\mathscr{M}^{**})$ and constitutes a norm isometry is $L^p$-normed for $p\in]1,+\infty[$, this is equivalent to $\mathscr{M}$ being reflexive as Banach space \cite[Corollary~1.2.18]{Gigli:NonSmoothDifferentialStr}, 
while for $p=1$, the implication from \lq\lq reflexive as Banach space\rq\rq\, to \lq\lq reflexive as $L^{\infty}$-module\rq\rq\, still holds \cite[Propositions~1.2.13 and 1.2.17]{Gigli:NonSmoothDifferentialStr}. In particular all Hilbert modules are reflexive in both senses. 
}
\end{defn}

\begin{defn}[$L^0$-module]\label{df:L0Module}
{\rm Fix an $L^{\infty}$-module $\mathscr{M}$. Let $(B_i)_{i\in\mathbb{N}}$ a Borel partition of $M$ such that $\m(B_i)\in]0,+\infty[$ for each $i\in\mathbb{N}$. Denote by $\mathscr{M}^0$ the completion of  
$\mathscr{M}$ with respect to $\mathscr{M}$ the distance ${\sf d}_{\mathscr{M}^0}:\mathscr{M}\times\mathscr{M}\to[0,+\infty[$ defined by 
\begin{align*}
{\sf d}_{\mathscr{M}^0}(v,w):=\sum_{i\in\mathbb{N}}\frac{1}{2^i\m(B_i)}\int_{B_i}(|v-w|_{\m}\land 1)\d\m.
\end{align*}
We call $\mathscr{M}^0$ as the \emph{$L^0$-module} associated with $\mathscr{M}$. The induced topology on $\mathscr{M}^0$ does not depend on  the choice of $(B_i)_{i\in\mathbb{N}}$ (see \cite[p.~31]{Gigli:NonSmoothDifferentialStr}). In addition, scalar and functional multiplication, and the point-wise norm $|\cdot|_{\m}$ extend continuously to $\mathscr{M}^0$, so that all $\m$-a.e.~properties mentioned for $L^{\infty}$-module hold for general elements in $\mathscr{M}^0$ and $L^0(M;\m)$ in place of $\mathscr{M}$ and $L^{\infty}(M;\m)$. The point-wise paring of $\mathscr{M}$ and $\mathscr{M}^*$extends uniquely and continuously to a bilinear map on $\mathscr{M}\times (\mathscr{M}^*)^0$ with values in $L^0(M;\m)$ such that for every $v\in\mathscr{M}^0$ and every $L\in(\mathscr{M}^*)^0$, 
\begin{align*}
|L(v)|\leq |L|_{\m}|v|_{\m}\quad\m\text{-a.e.}
\end{align*} 
}
\end{defn}
We have the following characterization of elements in $(\mathscr{M}^*)^0$:
\begin{prop}[{\cite[Proposition~1.3.2]{Gigli:NonSmoothDifferentialStr}}]\label{prop:L0characterization}
Let $T:\mathscr{M}^0\to L^0(M;\m)$ be a linear map for which there exists $f\in L^0(M;\m)$ such that 
for every $v\in\mathscr{M}$, 
\begin{align*}
|T(v)|\leq f|v|_{\m}\quad\m\text{-a.e.}
\end{align*}
Then there exists a unique $L\in(\mathscr{M}^*)^0$ such that for every $v\in\mathscr{M}$, 
\begin{align*}
L(v)=T(v)\quad\m\text{-a.e.,}
\end{align*}
and we furthermore have
\begin{align*}
|L|_{\m}\leq f\quad\m\text{-a.e.}
\end{align*}
\end{prop}
\subsection{Tensor products}\label{subsec:TensorProducts}
\begin{defn}[Tensor products]\label{df:TensorProducts}
{\rm Let $\mathscr{M}_1$ and $\mathscr{M}_2$ be two Hilbert module. 
By a slight abuse of notation, 
we denote both point-wise scalar products by $\langle \cdot,\cdot\rangle_{\m}$. 
We consider the $L^0$-module $\mathscr{M}_i^0$ induced from $\mathscr{M}_i$, $i=1,2$. 
Let $\mathscr{M}_1^0\odot\mathscr{M}_2^0$ be the \lq\lq algebraic tensor product\rq\rq\,
consisting of all finite linear combinations of formal elements $v\otimes w$, $v\in\mathscr{M}_1^0$ and $w\in\mathscr{M}_2^0$, obtained by factorizing appropriate vector spaces (see \cite[Section~1.5]{Gigli:NonSmoothDifferentialStr}). It naturally comes with a multiplications $\cdot\,: L^0(M;\m)\times \mathscr{M}_1^0\odot\mathscr{M}_2^0 \to \mathscr{M}_1^0\odot\mathscr{M}_2^0$ defined through
\begin{align*}
f(v\otimes w):=(fv)\otimes w=v\otimes (fw)
\end{align*}
and a point-wise scalar product $\langle \cdot\,|\,\cdot\rangle_{\m}\,: (\mathscr{M}_1^0\odot\mathscr{M}_2^0)^2\to L^0(M;\m)$ given by 
\begin{align*}
\langle (v_1\otimes w_1)\,|\,(v_2\otimes w_2)\rangle :=\langle v_1,v_2\rangle_{\m}\langle w_1,w_2\rangle_{\m},
\end{align*}
both extended to $\mathscr{M}_1^0\odot\mathscr{M}_2^0$ by bi-linearity. 
Then $\langle \cdot\,|\,\cdot\rangle_{\m}$ is bilinear, $\m$-a.e.~non-negative definite, symmetric, and local in both components 
(see \cite[Lemma~3.2.19]{GPLecture}).    

The point-wise \emph{Hilbert-Schmidt norm} $|\cdot|_{{\rm HS}}: \mathscr{M}_1^0\odot\mathscr{M}_2^0\to L^0(M;\m)$ is given by 
\begin{align*}
|A|_{{\rm HS}}:=\sqrt{\langle A\,|\,A\rangle_{\m}}.
\end{align*} 
This map satisfies the $\m$-a.e. triangle inequality and is $1$-homogenous with respect to multiplication with $L^0(M;\m)$-functions (see \cite[p.~44]{Gigli:NonSmoothDifferentialStr}). 
Consequently, the map $\|\cdot\|_{\mathscr{M}_1\otimes\mathscr{M}_2}:\mathscr{M}_1^0\odot\,\mathscr{M}_2^0\to[0,+\infty]$ defined through
\begin{align*}
\|A\|_{\mathscr{M}_1\otimes\mathscr{M}_2}:=\||A|_{{\rm HS},\m}\|_{L^2(M;\m)}
\end{align*}
has all properties of a norm except that it might take the value $+\infty$. 

Now we define the \emph{tensor product} $\mathscr{M}_1\otimes\mathscr{M}_2$ the $\|\cdot\|_{\mathscr{M}_1\otimes\mathscr{M}_2}$-completion of the subspace that consists of all $A\in \mathscr{M}_1^0\odot\mathscr{M}_2^0$ such that $\|A\|_{\mathscr{M}_1\otimes\mathscr{M}_2}langle \infty$.  
Denote by $A^{\top}\in \mathscr{M}^{\otimes2}$ the {\it transpose} of $A\in \mathscr{M}^{\otimes2}$ as defined in 
\cite[Section~1.5]{Gigli:NonSmoothDifferentialStr}. For instance, for bounded $v,w\in\mathscr{M}$ we have 
\begin{align}
(v\otimes w)^{\top}=w\otimes v.\label{eq:transpose}
\end{align}
}
\end{defn}

\subsection{Exterior products}\label{subsec:ExteriorProducts} 
Let $\mathscr{M}$ be a Hilbert module and $k\in\N\cup\{0\}$. 
Set $\Lambda^0\!\mathscr{M}^0:=L^0(M;\m)$ and, for $k\geq1$, let $\Lambda^k\!\mathscr{M}^0$ be the \lq\lq exterior product\rq\rq 
constructed by suitable factorizing $(\mathscr{M}^0)^{\odot k}$ (see \cite[Section~1.5]{Gigli:NonSmoothDifferentialStr}). 
The representative of $v_1\odot\cdots\odot v_k, v_1,\cdots, v_k\in\mathscr{M}^0$ in $\Lambda^k\!\mathscr{M}^0$ is written $v_1\wedge \cdots\wedge v_k$. $\Lambda^k\!\mathscr{M}^0$ naturally comes with a multiplication $\cdot :L^0(M;\m)\times\Lambda^k\!\mathscr{M}^0$ via
\begin{align*}
f(v_1\wedge \cdots\wedge v_k):=(fv_1)\wedge \cdots\wedge v_k=\cdots=v_1\wedge \cdots\wedge (fv_k)
\end{align*}
and a pointwise scalar product $\langle \cdot,\cdot\rangle_{\m}:(\Lambda^k\!\mathscr{M}^0)^2\to L^0(M;\m)$ defined by 
\begin{align}
\langle v_1\wedge\cdots \wedge v_k,w_1\wedge \cdots \wedge w_k\rangle_{\m}:={\rm det}\left[\langle v_i,w_j\rangle_{\m} \right]_{i,j\in\{1,2,\cdots,k\}}\label{eq:exteroprinnerproduct}
\end{align}
up to a factor $k!$, both extended to $\Lambda^k\!\mathscr{M}^0$ by (bi-)linearity. Then $\langle \cdot,\cdot\rangle_{\m}$ is bilinear, 
$\m$-a.e. non-negative definite, symmetric, and local in both components.

Given any $k,k'\in\N\cup\{0\}$, the map assigning to $v_1\wedge \cdots \wedge v_k\in\Lambda^k\!\mathscr{M}^0$ and 
$w_1\wedge \cdots \wedge w_{k'}\in \Lambda^{k'}\!\mathscr{M}^0$ the element 
$v_1\wedge \cdots \wedge v_k\wedge w_1\wedge \cdots \wedge w_{k'}\in \Lambda^{k+k'}\!\mathscr{M}^0$ can and will be uniquely extended by bilinearity and continuity to a bilinear map 
$\bigwedge:\Lambda^k\!\mathscr{M}^0\times\Lambda^{k'}\!\mathscr{M}^0$ termed \emph{wedge product} (see \cite[p.~47]{GPLecture}).
If $k=0$ or $k'=0$, it simply corresponds to multiplication of elements of $\Lambda^{k'}\!\mathscr{M}^0$ or 
$\Lambda^{k}\!\mathscr{M}^0$, respectively, with function in $L^0(M;\m)$ according to \eqref{eq:exteroprinnerproduct}.  

By a slight abuse of notation, define the map $|\cdot|_{\m}:\Lambda^k\!\mathscr{M}^0\to L^0(M;\m)$ by 
\begin{align*}
|\omega|_{\m}:=\sqrt{\langle \omega,\omega\rangle_{\m}}.
\end{align*}
It obeys the $\m$-a.e. triangle inequality and is  homogenous with respect to multiplication with $L^0(M;\m)$-functions (see \cite[p.~47]{GPLecture}). 

It follows that the map $\|\cdot\|_{\Lambda^k\!\mathscr{M}}:=\| |\cdot|_{\m}\|_{L^2(M;\m)}$
has all properties of a norm except that $\|w\|_{\Lambda^k\!\mathscr{M}}$ might be infinite. 
\begin{defn}
{\rm The ($k$-fold) exterior product $\Lambda^k\!\mathscr{M}$ is defined as the completion with respect to 
$\|\cdot\|_{\Lambda^k\!\mathscr{M}}$ of the subspace consisting of all $\omega\in\Lambda^k\!\mathscr{M}^0$ such that $\|\omega\|_{\Lambda^k\!\mathscr{M}}$.
}
\end{defn}
The space $\Lambda^k\!\mathscr{M}$ naturally becomes a Hilbert module and, if $\mathscr{M}$ is separable, is separable as wel (see \cite[p.~47]{GPLecture}).

\subsection{Cotangent module}\label{subsec:CotangentModule}
Let $(\mathscr{E},D(\mathscr{E}))$ be a quasi-regular strongly local Dirichlet form on $L^2(M;\m)$. 
We define the \emph{cotangent module} $L^2(T^*\!M)$, i.e., the space of differential 
$1$-forms that are $L^2$-integrable in a certain \lq\lq universal\rq\rq\, sense. 
\begin{defn}[Pre-cotangent module]\label{df:premodule}
{\rm We define the \emph{pre-cotangent module} 
${\rm Pcm}$ by 
\begin{align*}
{\rm Pcm}:&=\Biggl\{\left.(f_i,A_i)_{i\in\mathbb{N}}\;\Biggr|\, (A_i)_{i\in\mathbb{N}}\text{ Borel partition of }X,
 \right.\\
&\hspace{3cm}
\left. (f_i)_{i\in\mathbb{N}}\subset D(\mathscr{E})_e,\; \sum_{i\in\mathbb{N}}\int_{A_i}\Gamma(f_i)\d\m<\infty \right\}
\end{align*}
Moreover, we define a relation $\sim$ on ${\rm Pcm}$ by $(f_i,A_i)_{i\in\mathbb{N}}\sim(g_j,B_j)_{j\in\mathbb{N}}$ if and only if $\int_{A_i\cap B_j}\Gamma(f_i-g_j)\d\m=0$ for every $i,j\in\mathbb{N}$. The relation, in fact forms an equivalence relation by \cite[\S2.1]{Braun:Tamed2021}. 
The equivalence class of an element $(f_i,A_i)_{i\in\mathbb{N}}\in{\rm Pcm}$ with respect to $\sim$ is denoted by $[f_i,A_i]$. The space ${\rm Pcm}/\!\!\sim$ of equivalence classes becomes a vector space via the well-defined operations 
\begin{align}
[f_i,A_i]+[g_j,B_j]:=[f_i+g_j,A_i\cap B_j],\qquad 
\lambda[f_i,A_i]:=[\lambda f_i,A_i]\label{eq:vectorSpaceequivalece}
\end{align}  
for every $[f_i,A_i],[g_j,B_j]\in {\rm Pcm}/\!\!\sim$ and $\lambda\in\R$.

Now we define the space ${\rm SF}(M;\m)\subset L^{\infty}(M;\m)$ of simple functions, i.e., each element 
$h\in {\rm SF}(M;\m)$ attains only a finite number values. For $[f_i,A_i]\in {\rm Pcm}/\!\!\sim$ 
and $h=\sum_{j=1}^{\ell}a_j\1_{B_j}\in {\rm SF}(M;\m)$ with a Borel partition $(B_j)$ of $M$, we define the product 
$h[f_i,A_i]\in {\rm Pcm}/\!\!\sim$ as 
\begin{align}
h[f_i,A_i]:=[a_jf_i,A_i\cap B_j],\label{eq:multiplicationRule1}
\end{align}
where we set $B_j:=\emptyset$ and $a_j:=0$ for every $j>{\ell}$. 
It is readily verified that this definition is well-posed and that the resulting multiplication is a bilinear map from ${\rm SF}(M;\m)\times{\rm Pcm}/\!\!\sim$ into ${\rm Pcm}/\!\!\sim$ such that for every 
$[f_i,A_i]\in {\rm Pcm}/\!\!\sim$ and every $h,k\in {\rm SF}(M;\m)$
\begin{align}
(hk)[f_i,A_i]=h(k[f_i,A_i]),\qquad 
\1[f_i,A_i]=[f_i,A_i].\label{eq:multiplicationRule2}
\end{align} 
Moreover, the map $\|\cdot\|_{L^2(T^*\!M)}:{\rm Pcm}/\!\!\sim\, \to[0,+\infty[$ given by 
\begin{align*}
\| [f_i,A_i]\|_{L^2(T^*\!M)}^2:=\sum_{i\in\mathbb{N}}\int_{A_i}\Gamma(f_i)\d\m<\infty
\end{align*}
constitutes a norm on ${\rm Pcm}/\!\!\sim$. 
}
\end{defn}

\begin{defn}[Cotangent module]\label{df:CotangentModule}
{\rm We define the Banach space $(L^2(T^*\!M),\|\cdot\|_{L^2(T^*\!M)})$ as the completion of $({\rm Pcm}/\!\!\sim, \|\cdot\|_{L^2(T^*\!M)})$. The pair $(L^2(T^*\!M),\|\cdot\|_{L^2(T^*\!M)})$  or simply 
$L^2(T^*\!M)$ is called \emph{cotangent module}, and the elements of $L^2(T^*\!M)$ are called \emph{cotangent vector fields} or (\emph{differential}) \emph{$1$-forms}. 
}
\end{defn}
The following are shown in \cite[Lemma~2.3 and Theorem~2.4]{Braun:Tamed2021}:

\begin{lem}[{\cite[Lemma~2.3]{Braun:Tamed2021}}]\label{lem:ModulePropoerty}
The map from ${\rm SF}(M)\times {\rm Pcm}/\!\!\sim$ into ${\rm Pcm}/\!\!\sim$ defined in \eqref{eq:multiplicationRule1} 
extends continuously and unniquely to a bilinear map from $L^{\infty}(M;\m)\times L^2(T^*\!M)$ into $L^2(T^*\!M)$ satisfying, for every $f,g\in L^{\infty}(M;\m)$ and every $\omega\in L^2(T^*\!M)$,
\begin{align*}
(fg)\omega=f(g\omega),\quad \1_M\omega=\omega,\quad \|f\omega\|_{L^2(T^*\!M)}\leq \|f\|_{L^{\infty}(M;\m)}\cdot\|\omega\|_{L^2(T^*\!M)}.
\end{align*}
\end{lem}
\begin{thm}[{Module property, \cite[Theorem~2.4]{Braun:Tamed2021}}]\label{thm:ModuleProperty}
The cotangent module $L^2(T^*\!M)$ is an $L^2$-normed $L^{\infty}$-module over $M$ with respect to $\m$ whose point-wise norm $|\cdot|_{\m}$ satisfies, for every $[f_i,A_i]\in{\rm Pcm}/\!\!\sim$,
\begin{align}
|[f_i,A_i]|_{\m}=\sum_{i\in\mathbb{N}}\1_{A_i}\Gamma(f_i)^{\frac12}\quad\m\text{-a.e.}\label{eq:ModuleProperty}
\end{align}
In particular, $L^2(T^*\!M)$ is a Hilbert module with respect to $\m$. 
\end{thm}
\begin{defn}[$L^2$-differential]\label{df:differential}
{\rm The $L^2$-differential $\d f$ of any function $f\in D(\mathscr{E})_e$ is defined by 
\begin{align*}
{\d f}:=[f,X]\in L^2(T^*\!M),
\end{align*}
where $[f,X]\in {\rm Pcm}/\!\!\sim\;\subset L^2(T^*\!M)$ 
is the representative of the sequence $(f_i,A_i)_{i\in\mathbb{N}}$ given by 
$f_i:=f$, $A_1:=X$, $f_i:=0$ and $A_i:=\emptyset$ for every $i\geq2$. 

\bigskip

As usual, we call a $1$-form $\omega\in  L^2(T^*\!M)$ \emph{exact} if, for some $f\in D(\mathscr{E})_e$,
\begin{align*}
\omega=\d f.
\end{align*}
The $L^2$-differential $\d$ is a linear operator on $D(\mathscr{E})_e$. By \eqref{eq:ModuleProperty}, the $L^{\infty}$-module structure induced by $\m$ according to Theorem~\ref{thm:ModuleProperty},
\begin{align*}
|\d f|_{\m}=\Gamma(f)^{\frac12}\quad\m\text{-a.e.}
\end{align*}
holds for every $f\in D(\mathscr{E})_e$. 
}
\end{defn}
\subsection{Tangent module}\label{subsec:TangentModule}
In this section, let $(M,{\sf d},\m)$ be a metric measure space and assume the 
infinitesimally Hilbertian condition for it. In particular, we are given a strongly local Dirichlet form 
$(\mathscr{E},D(\mathscr{E}))$ on $L^2(M;\m)$. 
We define the notion of \emph{tangent module}. 
\begin{defn}[Tangent module]\label{df:TangentModule}
{\rm The tangent module $(L^2(TM),\|\cdot\|_{L^2(TM)})$ or simply $L^2(TM)$ is 
\begin{align*}
L^2(TM):=L^2(T^*\!M)^*
\end{align*}
and it is endowed with the norm $\|\cdot\|_{L^2(TM)}$ induced by \eqref{eq:DualPointNorm}. 
The elements of $L^2(TM)$ will be called \emph{vector fields}.  
}
\end{defn}
 As in Subsection~\ref{subsec:normedModule}, the point-wise pairing between $\omega\in L^2(T^*\!M)$ and $X\in L^2(TM)$ is denoted by $\omega(X)\in L^1(M;\m)$, and, by a slight abuse of notation, 
 $|X|\in L^2(M;\m)$ denotes the point-wise norm of $M$. By \cite[Lemma~2.7 and Proposition~1.24]{Braun:Tamed2021}, 
 $L^2(TM)$ is a separable Hilbert module. Furthermore, in terms of the point-wise scalar product $\langle \cdot,\cdot\rangle $ on $L^2(T^*\!M)$ and $L^2(TM)$, respectively, \cite[Proposition~1.24]{Braun:Tamed2021} allows us to define the \emph{(Riesz) musical isomorphisms} $\sharp: L^2(T^*\!M)\to L^2(TM)$ 
 and $\flat:=\sharp^{-1}$ defined by 
 \begin{align}
 \langle \omega^{\sharp},X\rangle :=\omega(X)=:\langle X^{\flat},\omega\rangle \quad \m\text{-a.e.}\label{eq:MusicalMap}
 \end{align}
\begin{defn}[$L^2$-gradient]\label{df:Gradient}
{\rm The \emph{$L^2$-gradient} $\nabla f$ of a function $f\in D(\mathscr{E})_e$ is defined by 
\begin{align*}
\nabla f:=(\d f)^{\sharp}.
\end{align*}
Observe from \eqref{eq:MusicalMap} that $f\in D(\mathscr{E})_e$, is characterized as the unique element 
$X\in L^2(TM)$ which satisfies 
\begin{align*}
\d f(X)=|\d f|^2=|X|^2\quad\m\text{-a.e.}
\end{align*}
}
\end{defn} 

\subsection{Divergences}\label{subsec:divergence}

\begin{defn}[$L^2$-divergence]\label{df:L2Divergence}
{\rm We define the space $D({\rm div})\subset L^2(TM)$ by 
\begin{align*}
D({\rm div}):&=\Bigl\{X\in  L^2(TM)\,\Bigl|\, \text{ there exists a function }f\in L^2(M;\m)\\ 
&\hspace{2cm}\text{ such  that for every }h\in D(\mathscr{E}),\quad -\int_M h\,f\,\d\m=\int_M\d h(X)\d\m\Bigr\}.
\end{align*} 
Such $f\in L^2(M;\m)$ is unique, and it is called the \emph{{\rm(}$L^2$-{\rm)}divergence of} $M$ and denoted by ${\rm div}\,X$.
}
\end{defn} 
The uniqueness of ${\rm div}\,X$ comes from the denseness of $D(\mathscr{E})$ in $L^2(M;\m)$.     
Note that the map ${\rm div}:D({\rm div})\to L^2(M;\m)$ is linear, hence $D({\rm div})$ is a vector space. By definition of $D(\Delta)$, $\nabla D(\Delta)\subset D({\rm div})$ and 
\begin{align}
{\rm div}\nabla f=\Delta f\quad \m\text{-a.e.}\label{eq:Divergence}
\end{align}
for every $f\in D(\Delta)$. Moreover, using the Leibniz rule in \cite[Theorem~4.4]{GPLecture},
one can verify that for every $X\in D({\rm div})$ and $f\in D(\mathscr{E})_e\cap L^{\infty}(M;\m)$ with 
$|\d f|\in L^{\infty}(M;\m)$, we have $fX\in D({\rm div})$ and 
\begin{align}
{\rm div}(fX)=f\,{\rm div}(X)+\d f(X)\quad\m\text{-a.e.}\label{eq:DivergenceFormula}
\end{align}

\begin{defn}[Measure-valued divergence]\label{def:MeasureDivergence}
{\rm We define the space $D({\bf div})\subset L^2(TM)$ by 
\begin{align}
D({\bf div}):&=\Bigl\{X\in  L^2(TM)\,\Bigl|\, \text{ there exists a $\sigma$-finite signed Radon smooth measure }\nu\notag\\
&\hspace{1cm}\text{ such  that for every }h\in D(\mathscr{E})_c\cap L^{\infty}(M;\m),\quad -\int_M h\,\d\nu=\int_M\d h(X)\d\m\Bigr\}.\label{eq:MeasureDivergence}
\end{align}
Such a $\sigma$-finite signed Radon smooth measure $\nu$ is unique, and it is called the \emph{measure-valued divergence} of $X$ and denoted by ${\bf div}\,X$.
Moreover, we define the subclass $D_f({\bf div})$ of $D({\bf div})$ as follows: 
\begin{align}
D_f({\bf div}):&=\Bigl\{X\in  L^2(TM)\,\Bigl|\, \text{ there exists a finite signed smooth measure }\nu\notag\\
&\hspace{0.5cm}\text{ such  that for every }h\in D(\mathscr{E})\cap L^{\infty}(M;\m),\quad -\int_M h\,\d\nu=\int_M\d h(X)\d\m\Bigr\}.\label{eq:MeasureDivergenceFinite}
\end{align}
Such a finite signed smooth measure $\nu$ is unique, and it is also called the \emph{measure-valued divergence} of $X$ and denoted by ${\bf div}\,X$. 
}
\end{defn}
The following Leibniz rule is an improvement of \cite[Lemma~3.13]{Braun:Tamed2021} and is useful. 
\begin{lem}\label{lem:LeibnizRule}
For every $X\in D({\bf div})$ {\rm(}resp.~$X\in D_f({\bf div})${\rm)} and every $f\in D(\mathscr{E})_e\cap L^{\infty}(M;\m)$, we have 
$fX\in D({\bf div})$ {\rm(}resp.~$fX\in D_f({\bf div})${\rm)} with 
\begin{align*}
{\bf div}(fX)=\tilde{f}{\bf div}\,X+\d f(X)\m.
\end{align*} 
\end{lem}
\begin{proof}[{\bf Proof}] 
We only show the proof for the case $X\in D({\bf div})$. The proof for the case $X\in D_f({\bf div})$ is similar. 
It is easy to see $fX\in L^2(TM)$. Moreover $(\d f)(X)=\langle \nabla f,X\rangle \in L^1(M;\m)$, since $X\in D({\bf div})$ and $f\in D(\mathscr{E})_e\cap L^{\infty}(M;\m)$. Then $(\d f)(X)\m$ is a finite signed smooth measure, hence 
a $\sigma$-finite signed Radon smooth measure.  Then the measure $\nu$ defined by 
\begin{align*}
\nu:=\tilde{f}{\bf div}\,X+\d f(X)\m
\end{align*} 
is also a signed Radon smooth measure, because $\tilde{f}$ is bounded. For $h\in D(\mathscr{E})_c\cap L^{\infty}(M;\m)$, 
\begin{align*}
-\int_M\tilde{h}\,\d\nu&=-\int_M\tilde{h}\tilde{f}\,\d {\bf div}\,X-\int_M h \,\d f(X)\d\m\\
&=\int_M\d (hf)(X)\d\m-\int_M h\d f(X)\d\m\\
&=\int_M(\d(hf)-h\d f)(X)\d\m\\
&=\int_Mf \d h(X)\d\m=\int_M\d h(fX)\,\d\m,
\end{align*}
which implies $\nu={\bf div}\,fX$. Here we use the Leibniz rule (see \cite[Proposition~2.11]{Braun:Tamed2021}): 
For every $f,g\in D(\mathscr{E})_e\cap L^{\infty}(M;\m)$, 
\begin{align*}
\d(fg)=\tilde{f}\,\d g+\tilde{g}\,\d f.
\end{align*}
\end{proof} 

\subsection{Measure-valued Laplacian}\label{subsec:MeasureLaplacian}

\begin{defn}[{Measure-valued Schr\"odinger operator}]\label{def:Schr\"odinger operator}
{\rm We define $D({\text{\boldmath$\Delta$}}^{2\kappa})$ to consist of all $u\in D(\mathscr{E})$ for which there exists a 
$\sigma$-finite Radon signed smooth measure $\nu$ such that for $D(\mathscr{E})\subset L^1(M;|\nu|)$ and 
for every $h\in D(\mathscr{E})$, we have \begin{align*}
-\int_M\tilde{h}\d\nu=\mathscr{E}^{2\kappa}(h,u).
\end{align*}
In case of existence, $\nu$ is unique, denoted by ${\text{\boldmath$\Delta$}}^{2\kappa}u$ and shall be called the 
\emph{measure-valued Schr\"odinger operator} with potential $2\kappa$. 
We define 
\begin{align*}
D_f({\text{\boldmath$\Delta$}}^{2\kappa}):=\{u\in \text{$D(\text{\boldmath$\Delta$}^{2\kappa})$}\mid  
{\text{\boldmath$\Delta$}}^{2\kappa}u \text{ is a finite signed smooth measure }\}.
\end{align*}
} 
\end{defn}
\begin{defn}[Measure-valued Laplacian]\label{def:MeasreValuedLaplacian}
{\rm We define $D({\text{\boldmath$\Delta$}})$ to consist of all $u\in D(\mathscr{E})$ for which there exists a 
$\sigma$-finite Radon signed smooth measure $\nu$ such that  
for every $h\in D(\mathscr{E})_c\cap L^{\infty}(M;\m)$, we have 
\begin{align*}
-\int_M\tilde{h}\d\nu=\mathscr{E}(h,u).
\end{align*}
In case of existence, $\nu$ is unique, denoted by ${\text{\boldmath$\Delta$}}u$ and shall be called the 
\emph{measure-valued Laplacian}. In another word, 
$D({\text{\boldmath$\Delta$}})$ consists of all $u\in D(\mathscr{E})$ for which $\nabla u\in D({\bf div})$, and $ {\text{\boldmath$\Delta$}}u:={\bf div}\nabla u$. The subspace 
$D_f({\text{\boldmath$\Delta$}})(\subset \text{$D(\text{\boldmath$\Delta$})$})$ 
such that 
the  ${\text{\boldmath$\Delta$}}u$ forms a finite signed smooth measure for $u\in D_f({\text{\boldmath$\Delta$}})$  
can be 
similarly defined by replacing the test functions $h\in D(\mathscr{E})_c\cap L^{\infty}(M;\m)$ with 
$h\in D(\mathscr{E})\cap L^{\infty}(M;\m)$. In particular, $D_f({\text{\boldmath$\Delta$}})$ consists of all 
$u\in D(\mathscr{E})$ for which $\nabla u\in D_f({\bf div})$. 
For $u\in D({\text{\boldmath$\Delta$}})$, 
denote by ${\text{\boldmath$\Delta$}}_{\ll}u$ (resp.~${\text{\boldmath$\Delta$}}_{\perp}u$) the absolutely continuous (resp.~singular) part of ${\text{\boldmath$\Delta$}}u$ with respect to $\m$ and set 
\begin{align*}
\Delta_{\ll}u:=\frac{\d {\text{\boldmath$\Delta$}}_{\ll}u}{\d\m}.
\end{align*}
}
\end{defn}

\begin{lem}[{cf. \cite[Lemma~8.12]{Braun:Tamed2021}}]\label{lem:Inclusion}
The signed Borel measure $\tilde{u}^2\kappa$ for $u\in D(\mathscr{E})$ has finite total variation. 
For $u\in D(\mathscr{E})$ with $u^2\in D({\text{\boldmath$\Delta$}}^{2\kappa})$ {\rm(}resp.~$u^2\in D_f({\text{\boldmath$\Delta$}}^{2\kappa})${\rm)}, we have $u^2\in D({\text{\boldmath$\Delta$}})$ {\rm(}resp.~$u^2\in D_f({\text{\boldmath$\Delta$}})${\rm)} with 
\begin{align*}
{\text{\boldmath$\Delta$}}^{2\kappa}u^2={\text{\boldmath$\Delta$}}u^2-2\tilde{u}^2\kappa.
\end{align*}
\end{lem}
\begin{proof}[{\bf Proof}] 
Since $|\kappa|\in S_D({\bf X})$, we see that $\tilde{u}^2|\kappa|$ and $\tilde{u}^2\kappa^{\pm}$ are finite smooth measures in view of Stollmann-Voigt's inequality \eqref{eq:StollmannVoigt}.
Since $(M,\tau)$ is a Lusin topological space, any finite measure is a Radon smooth measure. Thus, $\tilde{u}^2\kappa$ is a 
 signed Radon smooth measure. 
For $u\in D(\mathscr{E})$ with  $u^2\in D({\text{\boldmath$\Delta$}}^{2\kappa})$ (resp.~$u^2\in D_f({\text{\boldmath$\Delta$}}^{2\kappa})$), if we set $\nu:=
{\text{\boldmath$\Delta$}}^{2\kappa}u^2+2\tilde{u}^2\kappa$, then $\nu$ 
is a signed Radon (resp.~finite) smooth measure, and for $h\in D(\mathscr{E})_c\cap L^{\infty}(M;\m)$ (resp.~$h\in D(\mathscr{E})\cap L^{\infty}(M;\m)$), we have
\begin{align*}
-\int_M\tilde{h}\,\d\nu&=-\int_M\tilde{h}\,\d {\text{\boldmath$\Delta$}}^{2\kappa}u^2-2\int_M\tilde{h}\tilde{u}^2\d\kappa\\
&=\mathscr{E}^{2\kappa}(h,u^2)-\langle 2\kappa,\tilde{h}\tilde{u}^2\rangle \\
&=\mathscr{E}(h,u^2),
\end{align*}
hence $u^2\in D({\text{\boldmath$\Delta$}})$ (resp.~$u^2\in D_f({\text{\boldmath$\Delta$}})$) and $\nu={\text{\boldmath$\Delta$}}u^2$. 
\end{proof} 

\begin{lem}\label{lem:Composition}
Take $f\in D({\text{\boldmath$\Delta$}})$. Let $\Phi\in C^2(\R)$ satisfy $\Phi(0)=0$, $\Phi',\Phi''\in C_b(\R)$. 
Then $\Phi(f)\in D({\text{\boldmath$\Delta$}})$ with 
\begin{align}
{\text{\boldmath$\Delta$}}\Phi(f)&=\Phi''(f)|\nabla f|^2\m+\Phi'(f){\text{\boldmath$\Delta$}}f, \label{eq:Composition1}\\
\Delta_{\ll}\Phi(f)&=\Phi''(f)|\nabla f|^2+\Phi'(f)\Delta_{\ll}f.  \label{eq:Composition2}
\end{align}
If further $f\in D_f({\text{\boldmath$\Delta$}})$, then $\Phi(f)\in D_f({\text{\boldmath$\Delta$}})$ with \eqref{eq:Composition1} and \eqref{eq:Composition2}. 
\end{lem}
\begin{proof}[{\bf Proof}] 
It is easy to see $\Phi(f)\in D(\mathscr{E})$ by $\Phi(0)=0$. Take $h\in D(\mathscr{E})_c\cap L^{\infty}(M;\m)$. Then, by way of the Leibniz rule for measure-valued divergence (see Lemma~\ref{lem:LeibnizRule}),  
\begin{align*}
-\mathscr{E}(h,\Phi(f))&=-\int_M\langle \nabla h,\nabla\Phi(f)\rangle \d\m\\
&=-\int_M\langle \nabla h,\Phi'(f)\nabla f\rangle \d\m\\
&=\int_M\tilde{h}\,\d\,{\bf div}(\Phi'(f)\nabla f)\\
&=\int_M\tilde{h}\,\d\left(\Phi'(f){\bf div}(\nabla f)+ \d\Phi'(f)(\nabla f) \right)\quad (\text{Lemma~\ref{lem:LeibnizRule}})\\
&=\int_M \tilde{h}\,\d \left(\Phi'(f){\text{\boldmath$\Delta$}}f+\Phi''(f)|\nabla f|^2\m \right),
\end{align*}
which implies $\Phi(f)\in D({\text{\boldmath$\Delta$}})$ and ${\text{\boldmath$\Delta$}}\Phi(f)=\Phi'(f){\text{\boldmath$\Delta$}}f+\Phi''(f)|\nabla f|^2 \m$. The proof of the latter assertion is similar.  
\end{proof} 

\begin{defn}\label{def:measureLaplacian}
{\rm 
We define $\dot{D}({\text{\boldmath$\Delta$}})_{\loc}$ to consist of all $u\in \dot{D}(\mathscr{E})_{\loc}$ for which there exists a $\sigma$-finite signed Radon smooth measure $\nu$ such that for every $h\in \bigcup_{k=1}^{\infty}D(\mathscr{E})_{G_k}\cap L^{\infty}(M;\m)$ with $\{G_k\}\in\Xi(u)$, we have 
\begin{align*}
-\int_M\tilde{h}\,\d\nu=\mathscr{E}(h,u). 
\end{align*}
In case of existence, $\nu$ is unique, denoted by ${\text{\boldmath$\Delta$}}u$ and shall be called the 
\emph{measure-valued Laplacian}.
We define the subspace $\dot{D}_f({\text{\boldmath$\Delta$}})_{\loc}$ of $\dot{D}({\text{\boldmath$\Delta$}})_{\loc}$ by 
\begin{align*}
\dot{D}_f({\text{\boldmath$\Delta$}})_{\loc}:=\{u\in \dot{D}({\text{\boldmath$\Delta$}})_{\loc}\mid {\text{\boldmath$\Delta$}}u\text{ is a finite signed smooth measure}\}. 
\end{align*}
It is easy to see $D({\text{\boldmath$\Delta$}})\subset \dot{D}({\text{\boldmath$\Delta$}})_{\loc}$, 
$D_f({\text{\boldmath$\Delta$}})\subset \dot{D}_f({\text{\boldmath$\Delta$}})_{\loc}$ and 
$1\in \dot{D}_f({\text{\boldmath$\Delta$}})_{\loc}$ with ${\text{\boldmath$\Delta$}}1=0$. 
}
\end{defn}
\begin{cor}\label{cor:Composition}
Take $\eps>0$ and $p\in[1,2]$. If $u\in D({\text{\boldmath$\Delta$}})_+$ {\rm(}resp.~$u\in D_f({\text{\boldmath$\Delta$}})_+${\rm)}, then $(u+\eps^2)^{\frac{p}{2}}\in \dot{D}({\text{\boldmath$\Delta$}})_{\loc}$ {\rm(}resp.~$(u+\eps^2)^{\frac{p}{2}}\in \dot{D}_f({\text{\boldmath$\Delta$}})_{\loc}${\rm)}
and 
\begin{align}
{\text{\boldmath$\Delta$}}(u+\eps^2)^{\frac{p}{2}}&=\frac{p}{2}\left(\frac{p}{2}-1 \right)\left(u+\eps^2\right)^{\frac{p}{2}-2}|\nabla u|^2\m+\frac{p}{2}(\tilde{u}+\eps^2)^{\frac{p}{2}-1}{\text{\boldmath$\Delta$}}u,\label{eq:Composition3}\\
\Delta_{\ll}(u+\eps^2)^{\frac{p}{2}}&=\frac{p}{2}\left(\frac{p}{2}-1 \right)\left(u+\eps^2\right)^{\frac{p}{2}-2}|\nabla u|^2+\frac{p}{2}(\tilde{u}+\eps^2)^{\frac{p}{2}-1}\Delta_{\ll}u\quad\m\text{-a.e.}\label{eq:Composition4}
\end{align}
\end{cor}
\begin{proof}[{\bf Proof}] 
Take $\Psi(r):=(r+\eps^2)^{\frac{p}{2}}-\eps^p$ for $r\geq0$.  $\Psi$ can be extended to a $C^2$-function on $\R$ with $\Psi(0)=0$, $\Psi',\Psi''\in C_b(\R)$. Then we have $\Psi(u)\in D({\text{\boldmath$\Delta$}})$ and 
\begin{align*}
{\text{\boldmath$\Delta$}}\Psi(u)=\frac{p}{2}\left(\frac{p}{2}-1 \right)\left(u+\eps^2\right)^{\frac{p}{2}-2}|\nabla u|^2\m+\frac{p}{2}(\tilde{u}+\eps^2)^{\frac{p}{2}-1}{\text{\boldmath$\Delta$}}u.
\end{align*}
In particular, $(u+\eps^2)^{\frac{p}{2}}=\Psi(u)+\eps^p\in \dot{D}({\text{\boldmath$\Delta$}})_{\loc}$ and 
${\text{\boldmath$\Delta$}}(u+\eps^2)^{\frac{p}{2}}={\text{\boldmath$\Delta$}}(\Psi(u)+\eps^p)={\text{\boldmath$\Delta$}}\Psi(u)$.
\end{proof} 

\subsection{The Lebesgue spaces $L^p(T^*\!M)$ and $L^p(TM)$}\label{subsec:LebesgueSpace}
Let $L^0(T^*\!M)$ and $L^0(TM)$ be the $L^0$-modules as in Definition~\ref{df:L0Module} associated to $L^2(T^*\!M)$ and $L^2(TM)$, i.e.,
\begin{align*}
L^0(T^*\!M):=L^2(T^*\!M)^0,\quad  L^0(TM):=L^2(TM)^0.
\end{align*}
The characterization of Cauchy sequences in these spaces grants that 
the point-wise norms $|\cdot|_{\m}:L^2(T^*\!M)\to L^2(M;\m)$ and $|\cdot|_{\m}:L^2(TM)\to L^2(M;\m)$ as well as the (Riesz) musical isomorphisms $\flat:L^2(TM)\to L^2(T^*\!M)$ and $\sharp:L^2(T^*\!M)\to L^2(TM)$ uniquely extend to (non-relabeled) continuous map $\flat:L^0(TM)\to L^0(T^*\!M)$ and $\sharp:L^0(T^*\!M)\to L^0(TM)$. 

For $p\in[1,+\infty]$, let $L^p(T^*\!M)$ and $L^p(TM)$ be the Banach spaces consisting of all $\omega\in 
L^0(T^*\!M)$ and $X\in L^0(TM)$ such that $|\omega|\in L^p(M;\m)$ and $|X|\in L^p(M;\m)$, respectively, endowed with the norms
\begin{align*}
\|\omega\|_{L^p(T^*\!M)}:=\||\omega|_{\m}\|_{L^p(M;\m)},\qquad
\|X\|_{L^p(TM)}:=\||X|_{\m}\|_{L^p(M;\m)}.
\end{align*}
Since by \cite[Lemma~2.7]{Braun:Tamed2021}, $L^2(T^*\!M)$ is separable, and so is $L^2(TM)$ by \cite[Proposition~1.24]{Braun:Tamed2021}. Then one easily derives that if $p\ne+\infty$, the spaces 
$L^p(T^*\!M)$ and $L^p(TM)$ are separable as well. Since $L^2(T^*\!M)$ and $L^2(TM)$ are reflexive as Hilbert spaces, by the discussion in Definition~\ref{df:DualModule} it follows that $L^p(T^*\!M)$ and $L^p(TM)$ are reflexive for every $p\in]1,+\infty[$. For $q\in[1,+\infty[$ such that $1/p+1/q=1$ in the sense of $L^{\infty}$-modules we have the duality
\begin{align*}
L^p(T^*\!M)^*=L^q(TM).
\end{align*}
The point-wise norms $|\cdot|_{\m}:L^p(T^*\!M)\to L^p(M;\m)$ and $|\cdot|_{\m}:L^p(TM)\to L^p(M;\m)$ as well as the (Riesz) musical isomorphisms $\flat:L^p(TM)\to L^p(T^*\!M)$ and $\sharp:L^p(T^*\!M)\to L^p(TM)$ are defined by 
\begin{align}
{}_{L^p(TM)}\langle \omega^{\sharp},X\rangle_{L^q(TM)}:=\omega(X)=:{}_{L^q(T^*\!M)}\langle X^{\flat},\omega\rangle_{L^p(T^*\!M)}\quad \m\text{-a.e.}\label{eq:couplingP}
\end{align}
for $X\in L^q(TM)$ and $\omega\in L^p(T^*\!M)$,  
and we write \eqref{eq:couplingP} by 
\begin{align}
\langle\omega^{\sharp},X\rangle :=\omega(X)=:\langle X^{\flat},\omega\rangle \quad \m\text{-a.e.}\label{eq:couplingP*}
\end{align}
for simplicity.

\subsection{Test and regular objects}\label{subsec:TestandRegularObjects}
\begin{defn}[${\rm Test}(TM)$ and ${\rm Reg}(TM)$]\label{def:TestVecFields}
{\rm We define the subclass of $L^2(TM)$ consisting of \emph{test vector fields} or \emph{regular vector fields}, respectively: 
\begin{align*}
{\rm Test}(TM)&=\left\{\left.\sum_{i=1}^ng_i\nabla f_i\,\right|\, n\in\N,f_i,g_i\in {\rm Test}(M) \right\},\\
{\rm Reg}(TM)&=\left\{\left.\sum_{i=1}^ng_i\nabla f_i\,\right|\, n\in\N,f_i\in {\rm Test}(M),g_i\in {\rm Test}(M)\cup\R\1_M \right\}.
\end{align*}

}
\end{defn}

\begin{defn}[${\rm Test}(T^*\!M)$ and ${\rm Reg}(T^*\!M)$]\label{def:Test1Forms}
{\rm We define the subclass of $L^2(T^*\!M)$ consisting of \emph{test $1$-forms} or \emph{regular $1$-forms}, respectively:
\begin{align*}
{\rm Test}(T^*\!M)&=\left\{\left.\sum_{i=1}^ng_i\d f_i\,\right|\, n\in\N,f_i,g_i\in {\rm Test}(M) \right\},\\
{\rm Reg}(T^*\!M)&=\left\{\left.\sum_{i=1}^ng_i\d f_i\,\right|\, n\in\N,f_i\in {\rm Test}(M),g_i\in {\rm Test}(M)\cup\R\1_M \right\}.
\end{align*}
${\rm Test}(TM)$ (resp.~${\rm Test}(T^*\!M)$) is a dense subspace of $L^p(TM)$ for $p\in[1,+\infty[$ (resp.~$L^p(T^*\!M)$) (see \cite{Kw:HessSchraderUhlenbrock}), hence 
${\rm Reg}(TM)\cap L^p(TM)$ (resp.~${\rm Reg}(T^*\!M)\cap L^p(T^*\!M)$) is dense in $L^p(TM)$ (resp.~$L^p(T^*\!M)$). 
From this, we have that $L^2(TM)\cap L^p(TM)$ (resp.~$L^2(T^*\!M)\cap L^p(T^*\!M)$) is dense in $L^p(TM)$ 
(resp.~$L^p(T^*\!M)$) for $p\in[1,+\infty[$.
}
\end{defn}
\subsection{Lebesgue spaces on tensor products}\label{subsec:LebesgueSpaceTensorProdulct}
Denote the two-fold tensor products of $L^2(T^*\!M)$ and $L^2(TM)$, respectively, in the sense of Definition~\ref{df:TensorProducts} by 
\begin{align*}
L^2((T^*)^{\otimes2}X):=L^2(T^*\!M)\otimes L^2(T^*\!M),\quad 
L^2((T)^{\otimes2}X):=L^2(TM)\otimes L^2(TM).
\end{align*}
By the discussion from Subsection~\ref{subsec:TensorProducts}, 
Theorem~\ref{thm:ModuleProperty} and 
\cite[Proposition~1.24]{Braun:Tamed2021}, both are separable modules. They are point-wise isometrically module isomorphic: the respective pairing is initially defined by 
\begin{align*}
(\omega_1\otimes\omega_2)(X_1\otimes X_2):=\omega_1(X_1)\omega_2(X_2)\quad\m\text{-a.e.}
\end{align*}
for $\omega_1,\omega_2\in L^2(T^*\!M)\cap L^{\infty}(T^*\!M)$ and $X_1,X_2\in L^2(TM)\cap L^{\infty}(TM)$, 
and is extended by linearity and continuity to $L^2((T^*)^{\otimes2}X)$ and $L^2((T)^{\otimes2}X)$, respectively. By a slight abuse of notation, this pairing, with \cite[Proposition~1.24]{Braun:Tamed2021}, induces the (Riesz) musical isomorphisms $\flat:L^2((T)^{\otimes2}X)\to L^2((T^*)^{\otimes2}X)$ and $\sharp:=\flat^{-1}$ given by 
\begin{align}
\langle A^{\sharp}\,|\, T\rangle_{\m}:=A(T)=: \langle A\,|\,T^{\flat}\rangle_{\m}\quad\m\text{-a.e.}\label{eq:musical2}
\end{align}
and  
write $|A|_{\rm HS}:=\sqrt{\langle A\,|\,A\rangle_{\m}}$ and $|T|_{\rm HS}:=\sqrt{\langle T\,|\,T\rangle_{\m}}$ for $A\in L^2((T^*)^{\otimes2}X)$ and $T\in L^2((T)^{\otimes2}X)$. 

We let $L^p((T^*)^{\otimes2}X)$ and $L^p(T^{\otimes2}X)$, $p\in\{0\}\cup[1,+\infty]$, can be defined similarly as in 
Subsection~\ref{subsec:LebesgueSpace}. For $p\in[1,+\infty]$, these spaces naturally become Banach space which, if $p<\infty$, are separable. The following coincidence also holds 
\begin{align*}
L^p((T^*)^{\otimes2}X)=L^p(T^*\!M)\otimes L^p(T^*\!M),\quad 
L^q((T)^{\otimes2}X):=L^q(TM)\otimes L^q(TM)
\end{align*}
and  the (Riesz) musical isomorphisms $\flat:L^q((T)^{\otimes2}X)\to L^p((T^*)^{\otimes2}X)$ and $\sharp:=\flat^{-1}$ can be defined by 
\eqref{eq:musical2} 
for $A\in L^p((T^*)^{\otimes2}X)$ and $T\in L^q((T)^{\otimes2}X)$. 

\bigskip

We define the $L^p$-dense sets, $p\in[1,+\infty]$, intended strongly if $p<\infty$ and weakly* if $p=\infty$, reminiscent of 
Subsection~\ref{subsec:TensorProducts}. 
\begin{align*}
{\rm Test}((T^*)^{\otimes2}X):={\rm Test}(T^*\!M)^{\odot2},\\
{\rm Test}(T^{\otimes2}X):={\rm Test}(TM)^{\odot2},\\
{\rm Reg}((T^*)^{\otimes2}X):={\rm Reg}(T^*\!M)^{\odot2},\\
{\rm Reg}(T^{\otimes2}X):={\rm Reg}(TM)^{\odot2}.
\end{align*}
But the denseness in $L^p((T^*)^{\otimes2}X)$ (resp.~$L^p(TM)^{\otimes2}$) 
for ${\rm Reg}((T^*)^{\otimes2}X)$ (resp.~${\rm Reg}(T^{\otimes2}X)$) should be understood as 
for ${\rm Reg}((T^*)^{\otimes2}X)\cap L^p((T^*)^{\otimes2}X)$ (resp.~${\rm Reg}(T^{\otimes2}X)\cap L^p(T^{\otimes2}X)$).

\subsection{Lebesgue spaces on exterior products}\label{subsec:LebesgueSpExtPro}
Given any $k\in\N\cup\{0\}$, we set 
\begin{align*}
L^2(\Lambda^kT^*\!M):&=\Lambda^k L^2(T^*\!M),\\
L^2(\Lambda^kTM):&=\Lambda^k L^2(TM),
\end{align*}
where the exterior products are defined in Subsection~\ref{subsec:ExteriorProducts}. For $k\in\{0,1\}$, we see 
\begin{align*}
L^2(\Lambda^1T^*\!M)&=L^2(T^*\!M),\\
L^2(\Lambda^1TM)&=L^2(TM),\\
L^2(\Lambda^0T^*\!M)&=L^2(\Lambda^0TM)=L^2(M;\m).
\end{align*}
By Subsection~\ref{subsec:ExteriorProducts}, these are naturally Hilbert modules. As in Subsection~\ref{subsec:LebesgueSpaceTensorProdulct}, 
$L^2(\Lambda^kT^*\!M)$ and $L^2(\Lambda^kTM)$ are pointwise isometrically module isomorphic. 
For brevity, the induced pointwise pairing between $\omega\in L^2(\Lambda^kT^*\!M)$ and $X_1\land \cdots\land X_k\in 
L^2(\Lambda^k TM)$ with $X_1,\cdots, X_k\in L^2(TM)\cap L^{\infty}(TM)$, is written by 
\begin{align*}
\omega(X_1,\cdots, X_k):=\omega(X_1\land\cdots \land X_k). 
\end{align*}
We let $L^p(\Lambda^kT^*\!M)$ and $L^p(\Lambda^k TM)$, $p\in\{0\}\cup[1,+\infty]$, be as in Subsection~\ref{subsec:LebesgueSpace}. 
For $p\in[1,+\infty]$, these spaces are Banach and, if $p<\infty$, additionally separable.  
We define the formal $k$-th exterior products, $k\in\N\cup\{0\}$, of the classes from Subsection~\ref{subsec:TestandRegularObjects} as follows: 
\begin{align*}
{\rm Test}(\Lambda^kT^*\!M):&=\left\{\left. \sum_{i=1}^n f_i^0\d f_i^1\land\cdots\land \d f_i^k  \;\right|\; 
n\in\N, f_i^j\in{\rm Test}(M)\text{ for }0\leq j\leq k \right\},\\
{\rm Test}(\Lambda^kTM):&=\left\{\left. \sum_{i=1}^n f_i^0\nabla f_i^1\land\cdots\land \nabla f_i^k  \;\right|\; 
n\in\N, f_i^j\in{\rm Test}(M)\text{ for }0\leq j\leq k \right\},\\
{\rm Reg}(\Lambda^kT^*\!M):&=\left\{\left. \sum_{i=1}^n f_i^0\d f_i^1\land\cdots\land \d f_i^k  \;\right|\; 
n\in\N, f_i^j\in{\rm Test}(M)\text{ for }1\leq j\leq k,\right. \\
&\hspace{9cm} f_i^0\in{\rm Test}(M)\cup\R\1_M \Biggr\},\\
{\rm Reg}(\Lambda^kTM):&=\left\{\left. \sum_{i=1}^n f_i^0\d f_i^1\land\cdots\land \d f_i^k  \;\right|\; 
n\in\N, f_i^j\in{\rm Test}(M)\text{ for }1\leq j\leq k, \right.\\ 
&\hspace{9cm}
f_i^0\in{\rm Test}(M)\cup\R\1_M \Biggr\}. 
\end{align*}
We employ the evident interpretations for $k=1$, while the respective spaces for $k=0$ are identified with those spaces to which their generic elements's zeroth order terms belong to. These classes are dense in their respective $L^p$-spaces, 
$p\in[1,+\infty]$, intended strongly if $p<\infty$ and weakly* if $p=\infty$.  
But the denseness in $L^p(\Lambda^kT^*\!M)$ (resp.~$L^p(\Lambda^kTM)$) for ${\rm Reg}(\Lambda^kT^*\!M)$ (resp.~${\rm Reg}(\Lambda^kTM)$) should be modified as noted before.

\section{Covariant derivative}

\subsection{The Sobolev spaces $W^{1,2}(TM)$ and $H^{1,2}(TM)$}\label{subsec:W12SobolevTX}
We construct the tangent module $L^2(TM)$ in Subsection~\ref{subsec:TangentModule}. Now we define the $(1,2)$-Sobolev space 
$W^{1,2}(TM)$ based on the notion of $L^2$-Hessian. 
\begin{defn}[$(1,2)$-Sobolev space $W^{1,2}(TM)$]\label{def:SobolevSpaceW12TX}
{\rm The space $W^{1,2}(TM)$ is defined to consist of all $X\in L^2(TM)$ for which there exists a $T\in L^2(T^{\otimes2}X)$
such that for every $g_1,g_2,h\in {\rm Test}(M)$,
\begin{align*}
\int_Mh\langle T\,|\,&\nabla g_1\otimes\nabla g_2\rangle \d\m\\
&=-\int_M\langle X,\nabla g_2\rangle {\rm div}(h\nabla g_1)\d\m-\int_M\,h\,{\rm Hess}\,g_2(X,\nabla g_1)\d\m.
\end{align*}
Here ${\rm Hess}\,g_2\in L^2((T^*)^{\otimes2}X)$ is the \emph{Hessian} defined for $g_2\in {\rm Test}(M)$ (see \cite[Definition~5.2]{Braun:Tamed2021}).  
In case of existence, the element $T$ is unique, denoted by $\nabla X$ and called the \emph{covariant derivative} of $X$. 
}
\end{defn} 

Arguing as in \cite[after Definition~5.2]{Braun:Tamed2021}, the uniqueness statement in Definition~\ref{subsec:W12SobolevTX} is derived. 
In particular, $W^{1,2}(TM)$ constitutes a vector space and the covariant derivative $\nabla$ is a linear operator on it. 
The space $W^{1,2}(TM)$ is endowed with the norm $\|\cdot\|_{W^{1,2}(M)}$ given by
\begin{align*}
\|X\|_{W^{1,2}(TM)}^2:=\|X\|_{L^2(TM)}^2+\|\nabla X\|_{L^2(T^{\otimes2}X)}^2.
\end{align*}
We also define the \emph{covariant functional} $\mathscr{E}_{\rm cov}:L^2(TM)\to [0,+\infty[$ by
\begin{align}
\mathscr{E}_{\rm cov}(X):=\left\{\begin{array}{cc}\displaystyle{\int_M|\nabla X|_{\rm HS}^2\d\m} & \text{ if }X\in W^{1,2}(TM), \\ \infty & \text{ otherwise. }\end{array}\right.\label{eq:covariantfunctional}
\end{align}
It is proved in \cite[Theorem~6.3]{Braun:Tamed2021} that 
$(W^{1,2}(TM),\|\cdot\|_{W^{1,2}(TM)})$ is a separable Hilbert space, $\nabla$ is a closed operator, ${\rm Reg}(TM)\subset W^{1,2}(TM)$, $\|\cdot\|_{W^{1,2}(TM)}$-denseness of $W^{1,2}(TM)$ in $L^2(TM)$, and the lower semi continuity of $\mathscr{E}_{\rm cov}: L^2(TM)\to[0,+\infty[$. 
 
\begin{defn}[$(1,2)$-Sobolev space $H^{1,2}(TM)$]\label{def:SobolevSpaceH12TX}
{\rm We define the space $H^{1,2}(TM)\subset W^{1,2}(TM)$ as the $\|\cdot\|_{W^{12,}(TM)}$-closure of ${\rm Reg}(TM)$:
\begin{align*}
H^{1,2}(TM):=\overline{{\rm Reg}(TM)}^{\|\cdot\|_{W^{1,2}(TM)}}.
\end{align*}
$H^{1,2}(TM)$ is in general a strict subset of $W^{1,2}(TM)$. 
}
\end{defn} 
 
The following lemma is a version of what is known as \emph{Kato's inequality} (for the Bochner Laplacian) 
in the smooth case (see \cite[Chapter~2]{HSU}, \cite[Lemma~3.5]{GP}).  
\begin{lem}[{Kato's inequality, \cite[Lemma~6.12]{Braun:Tamed2021}}]\label{lem:KatoIneq}
For every $X\in H^{1,2}(TM)$, $|X|\in D(\mathscr{E})$ and
\begin{align*}
|\nabla |X||\leq|\nabla X|_{\rm HS}\quad\m\text{-a.e.}
\end{align*}
In particular, if $X\in H^{1,2}(TM)\cap L^{\infty}(M;\m)$, then $|X|^2\in D(\mathscr{E})$.
\end{lem}

\section{Exterior derivative}\label{sec:ExteriorDerivative}
Throughout this section, we fix $k\in\N\cup\{0\}$.  
\subsection{The Sobolev space $D(\d^k)$ and $D(\d_*^k)$}\label{subsec:ExteriorDerivSobolev}

Given $\omega\in L^0(\Lambda^k T^*\!M)$ and $X_0,\cdots, X_k,Y\in L^0(TM)$, we shall use the standard abbreviations: for $1\leq i<j\leq k$ 
\begin{align*}
\omega(\widehat{X}_i):&=\omega(X_0,\cdots,\widehat{X}_i,\cdots, X_k),\\
:&=\omega(X_0\land\cdots\land X_{i-1}\land X_{i+1}\land\cdots\land X_k),\\
\omega(Y,\widehat{X}_i,\widehat{Y}_j):&=\omega(Y,X_0,\cdots,\widehat{X}_i,\cdots,\widehat{X}_j,\cdots, X_k),\\
:&=\omega(Y\land X_0\land\cdots\land X_{i-1}\land X_{i+1}\land\cdots\land Y_{j-1}\land Y_{j+1}\land\cdots\land X_k).
\end{align*}

\begin{defn}[Sobolev space $D(\d^k)$]
{\rm We define $D(\d^k)$ to consist of all $\omega\in L^2(\Lambda^kT^*\!M)$ for which there exists $\eta\in L^2(\Lambda^{k+1}T^*\!M)$ such that for every $X_0,\cdots, X_k\in {\rm Test}(M)$, 
\begin{align*}
\int_M\eta(X_0,\cdots, X_k)\d\m&=\int_M\sum_{i=0}^k(-1)^{i+1}\omega(\widehat{X}_i){\rm div}\,X_i\d\m\\
&\hspace{2cm}+\int_M\sum_{i=0}^k\sum_{j=i+1}^k(-1)^{i+j}\omega([X_i,X_j],\widehat{X}_i,\widehat{X}_j)\d\m.
\end{align*}
In case of existence, the element $\eta$ is unique, denoted by $\d\omega$ and called the \emph{exterior derivative} 
(or \emph{exterior differential}) of $\omega$. 
}
\end{defn}

The uniqueness follows by density of ${\rm Test}(\Lambda^{k+1}T^*\!M)$ in $L^2(\Lambda^{k+1}T^*\!M)$ as discussed in 
Section~\ref{subsec:LebesgueSpace}. It is then clear that  $D(\d^k)$ is a real vector space and that $\d$ is 
a linear operator on it.

We always endow  $D(\d^k)$ with the norm $\|\cdot\|_{D(\d^k)}$ given by 
\begin{align*}
\|\omega\|_{D(\d^k)}^2:=\|\omega\|_{L^2(\Lambda^k T^*\!M)}^2+\|\d\omega\|_{L^2(\Lambda^k T^*\!M)}^2.
\end{align*}

We introduce the functional $\mathscr{E}_{\d}: L^2(\Lambda^kT^*\!M)\to[0,+\infty]$ with
\begin{align*}
\mathscr{E}_{\d}(\omega):=\left\{\begin{array}{cc}\displaystyle{\int_M|\d\omega|^2\d\m} & \text{ if }\omega\in D(\d^k) \\ \infty & \text{ otherwise.} \end{array}\right.
\end{align*} 
We do not make explicit the dependency of $\mathscr{E}_{\d}$ on the degree $k$. It will always be clear 
from the context which one is intended. It is proved in \cite[Theorem~7.5]{Braun:Tamed2021} that 
$(D(\d^k),\|\cdot\|_{D(\d^k)})$  is a separable Hilbert space, the exterior differential $\d$ is a closed operator, 
${\rm Reg}(\Lambda^kT^*\!M)\subset D(\d^k)$, $D(\d^k)$ is dense in $L^2 (\Lambda^kT^*\!M)$, and the functional 
$\mathscr{E}_{\d}:L^2(\Lambda^kT^*\!M)\to[0,+\infty]$ is lower semi continuous. 

\begin{defn}[The space $D_{\rm reg}(\d^k)$]
{\rm We define the space $D_{\rm reg}(\d^k)\subset D(\d^k)$ by the closure of ${\rm Reg}(\Lambda^kT^*\!M)$ with respect to the norm $\|\cdot\|_{D(\d^k)}$:
\begin{align*}
D_{\rm reg}(\d^k):=\overline{{\rm Reg}(\Lambda^kT^*\!M)}^{\|\cdot\|_{D(\d^k)}}.
\end{align*}
It is proved in \cite[Theorem~7.5]{Braun:Tamed2021} that for every $\omega\in D_{\rm reg}(\d^k)$, we have 
$\d\omega\in D_{\rm reg}(\d^{k+1})$ with $\d(\d\omega)=0$. 
}
\end{defn}

\begin{defn}[The space $D(\d_*^k)$]\label{def:adjointextrioderivative}
{\rm Given any $k\in\N$, the space  $D(\d_*^k)$ is defined to consist of all $\omega\in L^2(\Lambda^kT^*\!M)$ for which there exists $\rho\in L^2(\Lambda^{k-1}T^*\!M)$ such that for every $\eta\in {\rm Test}(\Lambda^{k-1}T^*\!M)$, we have 
\begin{align*}
\int_M\langle \rho,\eta\rangle \,\d\m=\int_M\langle \omega,\d\eta\rangle \,\d\m. 
\end{align*}
If it exists, $\rho$ is unique, denoted by $\d_*\omega$ and called the \emph{codifferential} of $\omega$. We simply define 
$D(\d_*^0):=L^0(M;\m)$ and $\d_*:=0$ on this space. 
}
\end{defn}
By the density of ${\rm Test}(\Lambda^{k-1}T^*\!M)$ in $L^2(\Lambda^{k-1}T^*\!M)$, the uniqueness statement is indeed true. 
Furthermore, $\d_*$ is a closed operator, i.e., the image of the assignment ${\rm Id}\times\d_*:D(\d_*^k)\to L^2(\Lambda^kT^*\!M)\times L^2(\Lambda^{k-1}T^*\!M)$ is closed in $L^2(\Lambda^kT^*\!M)\times L^2(\Lambda^{k-1}T^*\!M)$. 

By way of \cite[Theorem~7.5 and Lemma~7.15]{Braun:Tamed2021}, we have the following: 
For $f\in D(\mathscr{E})_e\cap L^{\infty}(M;\m)$ and $\omega\in H^{1,2}(T^*\!M)$ satisfying $\d f\in L^{\infty}(M;\m)$ or $\omega\in L^{\infty}(T^*\!M)$, 
$f\omega\in H^{1,2}(T^*\!M)$ with
\begin{align}
\d(f\omega)&=f\d\omega+\d f\land \omega,\\
\d_*(f\omega)&=f\d_*\omega-\langle \d f,\omega\rangle .
\end{align}

\subsection{The Sobolev spaces $W^{1,p}(\Lambda^kT^*\!M)$ and $H^{1,p}(\Lambda^{k+1}T^*\!M)$ 
}\label{subsec:SobolevExteriorDual}

\begin{defn}[The space $W^{1,2}(\Lambda^kT^*\!M)$]
{\rm We define the space $W^{1,2}(\Lambda^kT^*\!M)$ by 
\begin{align*}
W^{1,2}(\Lambda^kT^*\!M):=D(\d^k)\cap D(\d_*^k).
\end{align*}
By \cite[Theorem~7.5 and Lemma~7.15]{Braun:Tamed2021}, we already know that $W^{1,2}(\Lambda^kT^*\!M)$ is a dense subspace of 
$L^2(\Lambda^kT^*\!M)$. 

We endow $W^{1,2}(\Lambda^kT^*\!M)$ with the norm $\|\cdot\|_{W^{1,2}(\Lambda^kT^*\!M)}$ given by 
\begin{align*}
\|\omega\|_{W^{1,2}(\Lambda^kT^*\!M)}^2:=\|\omega\|_{L^2(\Lambda^kT^*\!M)}^2+\|\d \omega\|_{L^2(\Lambda^{k+1}T^*\!M)}^2+
\|\d_*\omega\|_{L^2(\Lambda^{k-1}T^*\!M)}^2
\end{align*}
and we define the \emph{contravariant} functional $\mathscr{E}_{\rm con}: L^2(\Lambda^kT^*\!M)\to[0,+\infty]$ by 
\begin{align*}
\mathscr{E}_{\rm con}(\omega):=\left\{\begin{array}{cc}\displaystyle{\int_M\left[|\d\omega|^2+|\d_*\omega|^2 \right]\d\m} & \text{ if }\omega\in W^{1,2}(\Lambda^kT^*\!M),\\\infty & \text{otherwise.} \end{array}\right.
\end{align*}
}
\end{defn}
Arguing as for \cite[Theorems~5.3, 6.3 and 7.5]{Braun:Tamed2021}, $W^{1,2}(\Lambda^kT^*\!M)$ becomes a separable Hilbert space with respect to $\|\cdot\|_{W^{1,2}(\Lambda^kT^*\!M)}$. moreover, the functional $\mathscr{E}_{\rm con}:L^2(\Lambda^kT^*\!M)\to[0,+\infty]$ is clearly lower semi continuous. 

Again by \cite[Theorem~7.5 and Lemma~7.15]{Braun:Tamed2021}, we have ${\rm Reg}(\Lambda^kT^*\!M)\subset W^{1,2}(\Lambda^kT^*\!M)$, so that the following definition makes sense. 

\begin{defn}[The space $H^{1,2}(\Lambda^kT^*\!M)$]
{\rm The space $H^{1,2}(\Lambda^kT^*\!M)\subset W^{1,2}(\Lambda^kT^*\!M)$ is defined by 
the closure of ${\rm Reg}(\Lambda^kT^*\!M)$ with respect to $\|\cdot\|_{W^{1,2}(\Lambda^kT^*\!M)}$: 
\begin{align*}
H^{1,2}(\Lambda^kT^*\!M):=\overline{{\rm Reg}(\Lambda^kT^*\!M)}^{\|\cdot\|_{W^{1,2}(\Lambda^kT^*\!M)}}.
\end{align*}

}
\end{defn}

\subsection{Hodge-Kodaira Laplacian}\label{subsec:HodgeKodairaLaplacian}
We now define the Hosdge-Kodaira Laplacian in L$^2$-sense: 
\begin{defn}[$L^2$-Hodge-Kodaira Laplacian $\DD_k$]
{\rm The space $D(\DD_k)$ is defined to consist of all $\omega\in H^{1,2}(\Lambda^kT^*\!M)$ for which there exists $\alpha\in L^2(\Lambda^kT^*\!M)$ such that for every $\eta\in  H^{1,2}(\Lambda^kT^*\!M)$,
\begin{align*}
\int_M\langle \alpha,\eta\rangle \d\m=-\int_M\left[\langle \d\omega,\d\eta\rangle +\langle \d_*\omega,\d_*\eta\rangle  \right]\d\m.
\end{align*}
In case of existence, the element $\alpha$ is unique, denoted by $\DD_k\omega$ and called the 
\emph{Hodge Laplacian}, \emph{Hodge-Kodaira Laplacian} or \emph{Hodge-de~\!\!Rham Laplacian} of $\omega$. 
Formally $\DD_k\omega$ can  be written \lq\lq$\DD_k\omega=-(\d\d_*+\d_*\d)\omega$\rq\rq. 
}
\end{defn}

For the most important case $k=1$, we write $\DD$ instead of $\DD_1$. We see $\DD_0=\Delta$ the usual $L^2$-generator associated to the given quasi-regular strongly local Dirichlet form $(\mathscr{E},D(\mathscr{E}))$. Moreover, the Hodge-Kodaira Laplacian $\DD_k$ is a closed operator. 

\subsection{Heat flow of $1$-forms associated with $\DD$}\label{subsec:HFHosgekodaira}
We define 
the functional $\widetilde{\mathscr{E}}_{\rm con}:L^2(T^*\!M)\to [0,+\infty]$ with 
\begin{align*}
\widetilde{\mathscr{E}}_{\rm con}(\omega):=\left\{\begin{array}{cc}\displaystyle{\int_M\left[|\d\omega|^2+|\d_*\omega|^2 \right]\d\m} & \text{ if }\omega\in H^{1,2}(T^*\!M), \\ \infty & \text{ otherwise.}\end{array}\right.
\end{align*}
We write $\HK$ instead of $\widetilde{\mathscr{E}}_{\rm con}$. Let $(P_t^{\rm HK})_{t\geq0}$ be the heat semigroup of bounded linear and self-adjoint operator on $L^2(T^*\!M)$ formally written by 
\begin{align*}
\text{\lq\lq $P_t^{\rm HK}:=e^{t\,\DD}$\rq\rq}.
\end{align*}
The following are important: 
\begin{lem}[{\cite[Lemma~8.32, Corollaries~8.35 and 8.36]{Braun:Tamed2021}}]
We have the following:
\begin{enumerate}
\item[\rm(1)]  For every $f\in D(\mathscr{E})$ and every $t>0$, $\d P_tf\in D(\DD)$ and 
\begin{align}
P_t^{\rm HK}\d f=\d P_tf.\label{eq:intertwining1}
\end{align}
\item[\rm(2)] If $\omega\in D(\d_*)$ and $t>0$, then $P_t^{\rm HK}\omega\in D(\d_*)$ and 
\begin{align}
\d_*P_t^{\rm HK}\omega=P_t\d_*\omega.\label{eq:intertwining2}
\end{align}
\item[\rm(3)] $\inf\sigma(-\Delta^{\kappa})\leq\inf \sigma(-\DD)$.
\end{enumerate}
\end{lem}
The formulas \eqref{eq:intertwining1} and \eqref{eq:intertwining2} are called \emph{intertwining properties}, which play 
a crucial role to prove the $L^p$-boundedness of Riesz operator.

The following theorem is known as \emph{Hess-Schrader-Uhlenbrock inequality}. 
This is proved in \cite[Theorem~8.41]{Braun:Tamed2021} under an extra assumption \cite[Assumption~8.37]{Braun:Tamed2021} in the framework of tamed Dirichlet space (Assumption~\ref{asmp:Tamed}). In \cite{Kw:HessSchraderUhlenbrock}, we improve 
\cite[Theorem~8.41]{Braun:Tamed2021} without this assumption. 
Theorem~\ref{thm:HessShcraderUhlenbrock} also plays a crucial role to prove the $L^p$-boundedness of Riesz operator.  

\begin{thm}[{Hess-Schrader-Uhlenbrock inequality, \cite[Theorem~1.2]{Kw:HessSchraderUhlenbrock}}]\label{thm:HessShcraderUhlenbrock}
We have the following: 
\begin{enumerate}
\item[\rm(1)] For every $\omega\in L^2(T^*\!M)$ and $\alpha>C_{\kappa}$, 
\begin{align}
|R_{\alpha}^{\rm HK}\omega|\leq R_{\alpha}^{\kappa}|\omega|\quad\m\text{-a.e.}\label{eq:ResolventHSU}
\end{align}
\item[\rm(2)] For every $\omega\in L^2(T^*\!M)$ and every $t\geq0$, 
\begin{align}
|P_t^{\rm HK}\omega|\leq P_t^{\kappa}|\omega|\quad\m\text{-a.e.}\label{eq:HSU}
\end{align}
\end{enumerate}
\end{thm}

Theorem~\ref{thm:HessShcraderUhlenbrock}(ii) is proved in 
\cite[Theorem~A]{Braun:HeatFlow} in the framework of RCD$(K,\infty)$-space for $K\in\R$, and also 
proved in 
\cite[Theorem~8.41]{Braun:Tamed2021} in the present framework with an extra assumption \cite[Assumption~8.37]{Braun:Tamed2021}.

\begin{cor}[{$C_0$-property of $(P_t^{\rm HK})_{t\geq0}$ on $L^p(M;\m)$, \cite[Corollary~1.3]{Kw:HessSchraderUhlenbrock}}]\label{cor:HessShcraderUhlenbrock}
 Suppose $p\in[2,+\infty]$, or $\kappa^-\in S_K({\bf X})$ and $p\in[1,+\infty]$.  
Then the heat flow $(P_t^{\rm HK})_{t\geq0}$ can be extended to a semigroup on $L^p(T^*\!M)$ and 
for each $t>0$
\begin{align}
\|P_t^{\rm HK}\omega\|_{L^p(T^*\!M)}\leq 
C(\kappa)e^{C_{\kappa}t}\|\omega\|_{L^p(T^*\!M)},\quad \omega\in L^p(T^*\!M).\label{eq:LpContHKHF}
\end{align}
Moreover, if $\kappa^-\in S_K({\bf X})$ and $p\in[1,+\infty[$, then  
$(P_t^{\rm HK})_{t\geq0}$ is strongly continuous on $L^p(T^*\!M)$, i.e., 
$(P_t^{\rm HK})_{t\geq0}$ is a $C_0$-semigroup on $L^p(T^*\!M)$, and further  
$(P_t^{\rm HK})_{t\geq0}$ is weakly* continuous on $L^{\infty}(T^*\!M)$.\end{cor}

\section{Ricci curvature measures}\label{sec:RicciCurvMeasure}
To state the notion of Ricci curvature measure, we need preliminary preparation.

The following are shown in \cite[Lemmas~8.1 and 8.2]{Braun:Tamed2021}:
\begin{lem}[{\cite[Lemma~8.1]{Braun:Tamed2021}}]
For every $g\in {\rm Test}(M)\cup\R\1_M$ and every $f\in{\rm Test}(M)$. we have $g\,\d f\in D(\DD)$ with 
\begin{align*}
\DD(g\,\d f)=g\,\d \Delta f+\Delta g\,\d f+2{\rm Hess}\,f(\nabla g,\cdot), 
\end{align*}
with the usual interpretations $\nabla \1_M=0$ and $\Delta \1_M=0$. More generally, for every $X\in {\rm Reg}(TM)$ and every $h\in {\rm Test}(M)\cup\R\1_M$, we have $hX^{\flat}\in D(\DD)$ with  
\begin{align*}
\DD(hX^{\flat})=h\DD X^{\flat}-\DD hX^{\flat}+2(\nabla_{\nabla h}X)^{\flat}.
\end{align*}
\end{lem}

\begin{lem}[{\cite[Lemma~8.2]{Braun:Tamed2021}}]\label{lem:Lemma8.2}
For every $X\in {\rm Reg}(TM)$, we have $|X|^2\in D({\text{\boldmath$\Delta$}}^{2\kappa})$ with 
\begin{align*}
{\text{\boldmath$\Delta$}}^{2\kappa}\frac{|X|^2}{2}\geq\left[|\nabla X|_{\rm HS}^2+\langle X,(\DD X^{\flat})^{\sharp}\rangle  \right]\m.
\end{align*}
\end{lem} 

\begin{defn}[The space $H^{1,2}_{\sharp}(TM)$]
{\rm We define the space $H^{1,2}_{\sharp}(TM)$ as the image of $H^{1,2}(T^*\!M)$ under the map $\sharp$, endowed with the norm 
\begin{align*}
\|X\|_{H^{1,2}_{\sharp}(TM)}:=\|X^{\flat}\|_{H^{1,2}(T^*\!M)}.
\end{align*}
}
\end{defn}
The following is proved in \cite[Lemma~8.8]{Braun:Tamed2021}:
\begin{lem}[{\cite[Lemma~8.8]{Braun:Tamed2021}}]\label{lem:Lemma8.8}
$H^{1,2}_{\sharp}(TM)$ is a subspace of $H^{1,2}(TM)$. The aforementioned natural inclusion is continuous. Additionally, for every $X\in H^{1,2}_{\sharp}(TM)$, 
\begin{align*}
\mathscr{E}_{\rm cov}(X)\leq\mathscr{E}_{\rm con}(X^{\flat})-\langle \kappa, |X|^2\rangle.
\end{align*}
\end{lem}
Then, we can state the main result of \cite{Braun:Tamed2021}: 
\begin{thm}[{{\cite[Theorem~8.9]{Braun:Tamed2021}}}]
\label{thm:Theorem8.9}
There exists a unique continuous mapping ${\bf Ric}^{\kappa}: H^{1,2}_{\sharp}(TM)^2\to \mathfrak{M}_{\rm f}^{\pm}(M)_{\mathscr{E}}$ satisfying the identity
\begin{align}
{\bf Ric}^{\kappa}(X,Y)={\text{\boldmath$\Delta$}}^{2\kappa}\frac{\langle X,Y\rangle }{2}-\left[\frac12\langle X,(\DD Y^{\flat})^{\sharp}\rangle 
+\frac12\langle Y,(\DD X^{\flat})^{\sharp}\rangle +\langle \nabla X\,|\,\nabla Y\rangle 
 \right]\m\label{eq:kappaRicciMeasure}
\end{align}
for every $X,Y\in {\rm Reg}(TM)$. Here $\mathfrak{M}_{\rm f}^{\pm}(M)_{\mathscr{E}}$ denotes the family of finite signed smooth measures. 
The map ${\bf Ric}^{\kappa}$ is symmetric and $\R$-linear. Furthermore, for every $X,Y\in H^{1,2}_{\sharp}(TM)$, it obeys
\begin{align}
{\bf Ric}^{\kappa}(X,X)&\geq0,\label{eq:NonKappaRicci}\\
{\bf Ric}^{\kappa}(X,Y)(M)&=\int_M\left[\langle \d X^{\flat},\d Y^{\flat}\rangle +\langle \d_*X^{\flat},\d_*Y^{\flat}\rangle  \right]\d\m\notag\\
 &\hspace{2cm}- \int_M\langle \nabla X\,|\,\nabla Y\rangle \d\m-\langle \kappa,\widetilde{\langle X,Y\rangle }\rangle ,\label{eq:RicciKappaMeasureTotal}\\
 \|{\bf Ric}^{\kappa}(X,Y)\|_{\rm TV}^2&\leq\left[\mathscr{E}_{\rm con}(X^{\flat})-\langle \kappa,|X|^2\rangle  \right]\left[\mathscr{E}_{\rm con}(Y^{\flat})-\langle \kappa,|Y|^2\rangle  \right].\label{eq:RicciKappaCont}
\end{align}
\end{thm}

Keeping in mind Theorem~\ref{thm:Theorem8.9}, we define the map 
${\bf Ric}:H^{1,2}_{\sharp}(TM)^2\to \mathfrak{M}_{\rm f}^{\pm}(M)_{\mathscr{E}}$ by 
\begin{align}
{\bf Ric}(X,Y):={\bf Ric}^{\kappa}(X,Y)+\langle \kappa,\widetilde{\langle X,Y\rangle }\rangle . \label{eq:RicciMeasure}
\end{align}
This map is well-defined, symmetric and $\R$-bilinear. ${\bf Ric}$ is also jointly continuous, which follows from polarization, \cite[Corollary~6.14]{Braun:Tamed2021} and Lemma~\ref{lem:Lemma8.8}. Lastly, by Theorem~\ref{thm:Theorem8.9}, Lemma~\ref{lem:Lemma8.2} and \cite[Lemma~8.12]{Braun:Tamed2021}, for every $X,Y\in{\rm Reg}(TM)$, we have a similar expression 
with \eqref{eq:kappaRicciMeasure}: 
\begin{align}
{\bf Ric}(X,Y)={\text{\boldmath$\Delta$}}\frac{\langle X,Y\rangle }{2}-\left[\frac12\langle X,(\DD Y^{\flat})^{\sharp}\rangle 
+\frac12\langle Y,(\DD X^{\flat})^{\sharp}\rangle +\langle \nabla X\,|\,\nabla Y\rangle  \right]\m\label{eq:RicciMeasure}
\end{align}
Denote by ${\bf Ric}_{\ll}^{\kappa}(X,Y)$ (resp.~${\bf Ric}_{\perp}^{\kappa}(X,Y)$) the absolutely continuous (resp.~singular) part of ${\bf Ric}^{\kappa}(X,Y)$ for $X,Y\in H^{1,2}_{\sharp}(TM)$ with respect to $\m$. Then we see ${\bf Ric}_{\ll}^{\kappa}(X,X)\geq0$ and 
${\bf Ric}_{\perp}^{\kappa}(X,X)\geq0$ for $X\in H^{1,2}_{\sharp}(TM)$, which yields 
\begin{align}
\left\{\begin{array}{rl}{\bf Ric}_{\ll}(X,X)&={\bf Ric}_{\ll}^{\kappa}(X,X)+\wt{|X|^2}\kappa_{\ll}\geq\wt{|X|^2}\kappa_{\ll}, \\
{\bf Ric}_{\perp}(X,X)&={\bf Ric}_{\perp}^{\kappa}(X,X)+\wt{|X|^2}\kappa_{\perp}\geq\wt{|X|^2}\kappa_{\perp}\end{array}\right.
\label{eq:RicciLower}
\end{align}
for $X\in  H^{1,2}_{\sharp}(TM)$. 
Here $\kappa_{\ll}$ (resp.~$\kappa_{\perp}$) is the absolutely continuous (resp.~singular) part of $\kappa$ with respect to $\m$. 
We define the Ricci tensor $\mathfrak{Ric}(X,Y)$ by 
\begin{align}
\mathfrak{Ric}(X,Y):=\frac{\d{\bf Ric}_{\ll}(X,Y)}{\d\m}\quad\text{ for }\quad X,Y\in H^{1,2}_{\sharp}(TM).\label{eq:RicciTensor}
\end{align}
In particular, by \eqref{eq:RicciLower}
\begin{align}
\mathfrak{Ric}(X,X)\geq |X|^2\frac{\d\kappa_{\ll}}{\d\m}\quad\text{ for }\quad X\in H^{1,2}_{\sharp}(TM).\label{eq:RicciTensor*}
\end{align}

\section{Littlewood-Paley $G$-functions}\label{sec:LittewoodPaley}
Recall the heat flow $(P_t^{\rm HK})_{t\geq0}$ of Hodge-Kodaira Laplacian $\DD$ acting on $L^2(T^*\!M)$. 
By Corollary~\ref{cor:HessShcraderUhlenbrock}, $(P_t^{\rm HK})_{t\geq0}$ can be 
regarded as a bounded semigroup on $L^p(T^*\!M)$. 
Let $(Q_t^{(\alpha),{\rm HK}})_{t\geq0}$ be the $\alpha$-order Cauchy semigroup acting on $L^p(T^*\!M)$ defined by 
\begin{align*}
Q_t^{(\alpha),{\rm HK}}\theta:=\int_0^{\infty}e^{-\alpha s}P_s^{\rm HK}\theta\lambda_t(\d s),\quad \theta\in L^p(T^*\!M),
\end{align*}
where $\lambda_t$ with $t\geq0$ is the probability measure on $[0,+\infty[$, whose Laplace transform is given by 
\begin{align*}
\int_0^{\infty}e^{-\gamma s}\lambda_t(\d s)=e^{-\sqrt{\gamma}t}, \quad \gamma\geq0.
\end{align*}
For $t>0$, $\lambda_t$ has the following expression
\begin{align*}
\lambda_t(\d s)=\frac{t}{2\sqrt{\pi}}e^{-t^2/4s}s^{-3/2}\d s.
\end{align*}
For $\theta\in L^2(T^*\!M)$, $Q_t^{(\alpha),{\rm HK}}\theta\in D(\DD)$ and 
\begin{align*}
\|\DD Q_t^{(\alpha),{\rm HK}}\theta\|_{L^2(T^*\!M)}\leq\left(\frac{1}{\sqrt{2}}\int_0^{\infty}e^{-\alpha s}\frac{1}{s}\lambda_t(\d s) \right)\|\theta\|_{L^2T^*\!M)}<\infty.
\end{align*}
$Q_t^{(\alpha),{\rm HK}}\theta$ can be formally written by 
\begin{align*}
Q_t^{(\alpha),{\rm HK}}\theta(x)= \text{\lq\lq $e^{-\sqrt{\alpha-\DD}\,t}\theta(x)$\rq\rq.}
\end{align*}
By Corollary~\ref{cor:HessShcraderUhlenbrock} and $\kappa^-\in S_K({\bf X})$, we see that
\begin{align}
\|Q_t^{(\alpha),{\rm HK}}\theta\|_{L^p(T^*\!M)}\leq Ce^{-\sqrt{\alpha-C_{\kappa}} t}\|\theta\|_{L^p(T^*\!M)}\label{eq:QtContraction}
\end{align}
holds for any $p\in[1,+\infty]$ 
and $(Q_t^{(\alpha),{\rm HK}})_{t\geq0}$ is a strongly continuous semigroup on $L^p(T^*\!M)$ for $p\in[1,+\infty[$ and $\alpha\geq C_{\kappa}$ and 
weakly* continuous on $L^{\infty}(T^*\!M)$ in the same way of the proof of  
\cite[Corollary~1.3]{Kw:HessSchraderUhlenbrock}. 
The infinitesimal generator of $(Q_t^{(\alpha),{\rm HK}})_{t\geq0}$ on $L^p(T^*\!M)$ is denoted by 
$-\sqrt{\alpha-\DD_p}$, where $\DD_p$ is the infinitesimal generator of $(P_t^{\rm HK})_{t\geq0}$ on $L^p(T^*\!M)$. 
 We may omit the subscript $p$ for simplicity. 

For $\theta\in L^2(T^*\!M)\cap L^p(T^*\!M)$ and $\alpha\geq C_{\kappa}$, we define the Littlewood-Paley's $G$-functions by 
\begin{align*}
{g_{\stackrel{\rightarrow}{\theta}}}(x,t)&:=\left|\frac{\partial}{\partial t}(Q_t^{(\alpha),\rm HK}\theta)(x) \right|,\qquad\qquad\quad {G_{\stackrel{\rightarrow}{\theta}}}(x):=\left(\int_0^{\infty}t{g_{\stackrel{\rightarrow}{\theta}}}(x,t)^2\d t  \right)^{1/2},\\
{g_{\theta}^{\uparrow}}
(x,t)&:=\left|\nabla(Q_t^{(\alpha),\rm HK}\theta)^{\sharp}(x)\right|_{\rm HS},\qquad\quad\quad{G_{\theta}^{\uparrow}}(x):=\left(\int_0^{\infty}t{g_{\theta}^{\uparrow}}(x,t)^2\d t  \right)^{1/2},\\
g_{\theta}(x,t)&:=\sqrt{{g_{\stackrel{\rightarrow}{\theta}}}(x,t)^2+
{g_{\theta}^{\uparrow}}
(x,t)^2}, \qquad\quad G_{\theta}(x):=\left(\int_0^{\infty}tg_{\theta}(x,t)^2\d t \right)^{1/2}.
\end{align*}

\begin{lem}[{Maximal ergodic inequality, \cite[Theorem~3.3]{ShigekawaText}}]\label{lem:MaxErgo}
Suppose $p\in]1,+\infty[$. For $f\in L^p(M;\m)$, 
\begin{align}
\left\|\sup_{t\geq0}P_t|f| \right\|_{L^p(M;\m)}\leq \frac{p}{p-1}\|f\|_{L^p(M;m)}.\label{eq:MaxErgo}
\end{align}
\end{lem}
\begin{prop}\label{prop:MaxErgo}
Suppose that $\kappa^-\in S_K({\bf X})$ and $p\in]1,+\infty[$. Then there exists $C_p>0$ such that for any $\alpha\geq C_{\kappa}$,
\begin{align*}
\int_M\sup_{t\geq0}\left| Q_t^{(\alpha),{\rm HK}}\theta(x)\right|^p\m(\d x)\leq C_p\|\theta\|_{L^p(T^*\!M)}^p,\quad \theta\in L^p(T^*\!M). 
\end{align*}
\end{prop}
\begin{proof}[{\bf Proof}] 
By Theorem~\ref{thm:HessShcraderUhlenbrock}, 
\begin{align*}
\left|  Q_t^{(\alpha),{\rm HK}}\theta(x)\right|&=\left|\int_0^{\infty}e^{-\alpha s}P_s^{\rm HK}\theta(x)\lambda_t(\d s) \right|\\
&\leq \int_0^{\infty}e^{-\alpha s}\left| P_s^{\rm HK}\theta(x)\right|\lambda_t(\d s)\\
&\leq \int_0^{\infty}e^{-\alpha s} P_s^{\kappa}|\theta|(x)\lambda_t(\d s)\\
&=\int_0^{\infty}\E_x\left[e^{-\alpha s-A_s^{\kappa}}|\theta|(X_s) \right]\lambda_t(\d s).
\end{align*}
Take any $\beta>p/(p-1)>1$ and $\beta':=\beta/(\beta-1)\in]1,p[$. Then
\begin{align*}
\sup_{t\geq0}\left|  Q_t^{(\alpha),{\rm HK}}\theta(x)\right|&\leq \sup_{s\geq0}\E_x\left[e^{-\alpha s-A_s^{\kappa}}|\theta|(X_s) \right]\\
&\leq \sup_{s\geq0}\E_x\left[e^{-\beta(\alpha s+A_s^{\kappa})} \right]^{\frac{1}{\beta}}\E_x\left[|\theta|^{\beta'}(x)\right]^{\frac{1}{\beta'}}\\
&\leq \sup_{s\geq0,y\in M}\left(P_s^{\beta(\alpha\m+\kappa)}1(y) \right)^{\frac{1}{\beta}}\cdot\sup_{s\geq0}\left( P_s|\theta|^{\beta'}(x)\right)^{\frac{1}{\beta'}}\\
&\leq \sup_{s\geq0,y\in M}\left(P_s^{\alpha\m-\beta\kappa^-}1(y) \right)^{\frac{1}{\beta}}\cdot\sup_{s\geq0}\left( P_s|\theta|^{\beta'}(x)\right)^{\frac{1}{\beta'}}\\
&\hspace{-0.2cm}\stackrel{\eqref{eq:KatoCoincidence}}{\leq} \sup_{s\geq0}\left( C e^{-\alpha s+C_{\kappa}s}\right)^{\frac{1}{\beta}}\cdot\sup_{s\geq0}\left( P_s|\theta|^{\beta'}(x)\right)^{\frac{1}{\beta'}}\\
&\leq C^{\frac{1}{\beta}}\sup_{s\geq0}\left( P_s|\theta|^{\beta'}(x)\right)^{\frac{1}{\beta'}}, \quad (\alpha\geq C_{\kappa}).
\end{align*}
Applying Lemma~\ref{lem:MaxErgo}, we obtain
\begin{align*}
\int_M\sup_{t\geq0}\left|  Q_t^{(\alpha),{\rm HK}}\theta(x)\right|^p\m(\d x)&=C^{\frac{p}{\beta}}\int_M\left|
\sup_{s\geq0}P_s|\theta|^{\beta'}(x)
 \right|^{\frac{p}{\beta'}}\m(\d x)\\
 &\leq C^{\frac{p}{\beta}}\left(\frac{\frac{p}{\beta'}}{\frac{p}{\beta'}-1} \right)^{\frac{p}{\beta'}}
 \int_M|\theta(x)|^{\beta'\frac{p}{\beta'}}\m(\d x)\\
 &=C^{\frac{p}{\beta}}\left(\frac{p}{p-\beta'} \right)^{\frac{p}{\beta'}}\|\theta\|_{L^p(T^*\!M)}^p\\
 &=C^{\frac{p}{\beta}}\left(\frac{p(\beta-1)}{(p-1)\beta-p} \right)^{\frac{p(\beta-1)}{\beta}}\|\theta\|_{L^p(T^*\!M)}^p.
\end{align*}
\end{proof} 

Now we set $\eta=\eta(x,t):=Q_t^{(\alpha),{\rm HK}}\theta(x)$.  
Let $\overline{\nabla}:=\left(\nabla, \frac{\partial}{\partial t}\right)$ be the covariant derivative over $X\times \R$ and 
$\nabla $ the covariant derivative acting on $W^{1,2}(TM)$: 
\begin{align*}
|\overline{\nabla}\eta(x,t)|^2:= \left|\nabla (Q_t^{(\alpha),{\rm HK}}\theta)^{\sharp}(x)\right|_{\rm HS}^2+\left|\frac{\partial}{\partial t}Q_t^{(\alpha),{\rm HK}}\theta(x) \right|^2=|\nabla\eta^{\sharp}(x,t)|_{\rm HS}^2+\left|\frac{\partial}{\partial t}\eta(x,t) \right|^2.
\end{align*}

\section{Auxiliary Brownian motion}\label{sec:BrawnianMotion}
Let $(B_t,\P_{\stackrel{\rightarrow}{a}})$ be one-dimensional Brownian motion 
starting at $a\in \R$ with the generator $\frac{\partial^2}{\partial a^2}$. We set 
$\widehat{M}:=M\times\R$, $\hat{x}:=(x,a)\in \widehat{M}$, 
$\widehat{X}_t:=(X_t,B_t)$, $t\geq0$, $\widehat{\m}:=\m\otimes m$ 
 and $\P_{\hat{x}}:=\P_x\otimes \P_{\stackrel{\rightarrow}{a}}$. 
Then $\widehat{\bf X}:=(\widehat{X}_t,\P_{\hat{x}})$ is an $\widehat{\m}$-symmetric diffusion process on 
$\widehat{M}$ with the (formal) generator $\Delta+\frac{\partial^2}{\partial a^2}$, where $m$ is one-dimensional Lebesgue measure. Let $\tau:=\inf\{t>0\mid B_t=0\}$ be the first hitting time of $B_t$ to $\{0\}$.  
We denote the semigroup on $L^p(\widehat{M};\widehat{\m})$ associated with the diffusion process 
$(\widehat{X}_t)_{t\geq0}$ by $(\widehat{P}_t)_{t\geq0}$ and its generator by $\widehat{\Delta}_p$. We also denote the Dirichlet form on $L^2(\widehat{M};\widehat{\m})$ associated with $\widehat{\Delta}_2$ by 
$(\widehat{\mathscr{E}}, D(\widehat{\mathscr{E}}))$. That is, 
\begin{align*}
D(\widehat{\mathscr{E}}):&=\left\{u\in L^2(\widehat{M};\widehat{\m})\;\left|\; \lim_{t\to0}\frac{1}{t}(u-\widehat{P}_tu,u)_{L^2(\widehat{M};\widehat{\m})}<+\infty\right.\right\},\\
\widehat{\mathscr{E}}(u,v):&=\lim_{t\to0}\frac{1}{t}(u-\widehat{P}_tu,v)_{L^2(\widehat{M};\widehat{\m})}\qquad\text{ for }\quad u,v\in D(\widehat{\mathscr{E}}).
\end{align*} 
Since $(\widehat{\mathscr{E}}, D(\widehat{\mathscr{E}}))$ is associated to the 
$\widehat{\m}$-symmetric Borel right process $\widehat{\bf X}$, it is quasi-regular by 
Fitzsimmons~\cite{Fitzsimmons}. 
We need the following lemma:
\begin{lem}\label{lem:Capacity}
We have the following:
\begin{enumerate}
\item[\rm(1)] Let $N$ be an $\wh{\mathscr{E}}$-exceptional set. Then, for $m$-a.e.~$a\in \R$, $N_a:=\{x\in M\mid (x,a)\in N\}$ is an $\mathscr{E}$-exceptional set. 
\item[\rm(2)] Let $(x,a)\mapsto u(x,a)$ be an $\wh{\mathscr{E}}$-quasi continuous function. Then, for $m$-a.e.~$a\in \R$, 
$x\mapsto u(x,a)$ is an $\mathscr{E}$-quasi continuous function. 
\end{enumerate}
\end{lem}
\begin{proof}[{\bf Proof}]
(1) is proved in $\widehat{\rm O}$kura~\cite[Theorem~4.1(4)]{Okura}. We prove (2). Since the $\mathscr{E}$-quasi continuity is equivalent to the $\mathscr{E}$-q.e.~fine continuity by 
\cite[Theorem~4.2.2 and Theorem~4.6.1]{FOT}, it suffices to prove that the fine continuity of $(x,a)\mapsto u(x,a)$ with respect to 
$\wh{\bf X}$ implies the fine continuity of $x\mapsto u(x,a)$ for each $a\in\R$ with respect to ${\bf X}$. 
This can be confirmed by that for each $a\in \R$  
\begin{align*}
t\mapsto u(\wh{X}_t)=u(X_t,a)\; \text{ is right continuous }\; \P_{(x,a)}\text{-a.s.~equivalently }\; \P_x\text{-a.s.~for }\; x\in M
\end{align*}
in view of \cite[Theorem~A.2.7]{FOT}. 
\end{proof}

Throughout this subsection, we assume $\kappa^+\in S_D({\bf X})$ and $\kappa^-\in S_{E\!K}({\bf X})$.  
We define $\widehat{\kappa}^{\,\pm}:=\kappa^{\pm}\otimes m$ and $\widehat{\kappa}:=
\widehat{\,\kappa}^+-\widehat{\,\kappa}^-$. 
The associated positive continuous additive functional (PCAF in short) $\widehat{A}_t^{\;\widehat{\,\kappa}^{\pm}}$ 
in Revuz correspondence under $\wh{\bf X}$
is given by $\widehat{A}_t^{\;\widehat{\kappa}^{\pm}}=A_t^{\kappa^{\pm}}$, in particular, $\widehat{\kappa}^{\,+}$ (resp.~$\widehat{\kappa}^{\,-}$) 
is a smooth measure of Dynkin (resp.~extended Kato) class  with respect to $\widehat{\bf X}$, i.e. 
$\widehat{\kappa}^{\,+}\in S_D(\widehat{\bf X})$ (resp.~$\widehat{\kappa}^{\,-}\in S_{E\!K}(\widehat{\bf X})$). 
Then we can define the following quadratic form 
$(\widehat{\mathscr{E}}^{\;\widehat{\kappa}}, D(\widehat{\mathscr{E}}^{\;\widehat{\kappa}}))$ 
on $L^2(\widehat{M};\widehat{\m})$: 
\begin{align}
D(\widehat{\mathscr{E}}^{\;\widehat{\kappa}}):=D(\widehat{\mathscr{E}}),\qquad
\widehat{\mathscr{E}}^{\;\widehat{\kappa}}(u,v):=\widehat{\mathscr{E}}(u,v)+\langle \widehat{\kappa},\tilde{u}\tilde{v}\rangle \quad\text{ for }\quad u,v\in D(\widehat{\mathscr{E}}^{\;\widehat{\kappa}}).\label{eq:quadratic}
\end{align}
Here $\tilde{u}$ denotes the $\widehat{\mathscr{E}}$-quasi-continuous $\widehat{\m}$-version of $u$ with respect to 
$(\widehat{\mathscr{E}}, D(\widehat{\mathscr{E}}))$. 
The strongly continuous semigroup $(\widehat{P}_t^{\;\widehat{\kappa}})_{t\geq0}$ on $L^2(\widehat{M};\widehat{\m})$ associated with the quadratic form 
$(\widehat{\mathscr{E}}^{\;\widehat{\kappa}}, D(\widehat{\mathscr{E}}^{\;\widehat{\kappa}}))$ 
is given by 
\begin{align}
\widehat{P}_t^{\widehat{\;\kappa}}u(\hat{x})=
\E_{\hat{x}}
[e^{-A_t^{\kappa}}u(\widehat{X}_t)]
\quad \text{ for }\quad u\in L^2(\widehat{M};\widehat{\m})\cap \mathscr{B}(\widehat{M}).\label{eq:FeynmanKac}
\end{align}

We denote by $\widehat{\mathscr{C}}:={\rm Test}
(M)\odot C_c^{\infty}(\R)$ the totality of all finite linear combinations of $f\otimes\varphi$, $f\in {\rm Test}(M)$, $\varphi\in C_c^{\infty}(\R)$, where 
$(f\otimes \varphi)(y):=f(x)\varphi(a)$. Meanwhile, the space $L^2(M;\m)\otimes L^2(\R)$ and 
$D(\mathscr{E})\otimes H^{1,2}(\R)$ are usual tensor products of Hilbert spaces, where $H^{1,2}(\R)$ is the Sobolev space which consists of all functions $\varphi\in L^2(\R)$ such that the weak derivative $\dot{\varphi}$ exists and belongs to $L^2(\R)$. Then we have 

\begin{lem}[Lemma~3.1 in \cite{EJK}]\label{lem:Core}
$\widehat{\mathscr{C}}$ is dense in $D(\widehat{\mathscr{E}}
)$. Moreover, 
for $u,v\in D(\mathscr{E}
)\otimes H^{1,2}(\R)$,  we have
\begin{align}
\widehat{\mathscr{E}}(u,v)&=\int_{\R}\mathscr{E}(u(\cdot,a),v(\cdot,a))m(\d a)+
\int_M\m(\d x)\int_{\R}\frac{\partial u}{\partial a}(x,a)\frac{\partial v}{\partial a}(x,a)m(\d a).\label{eq:Identity}
\end{align}
\end{lem}

The following lemma is proved in \cite[Proposition~3.9]{ShigekawaText}: 

\begin{lem}[{\cite[Proposition~3.9]{ShigekawaText}}]\label{lem:ShigekawaProp3.9}
Let $h$ be a non-negative measurable function on $[0,+\infty[$. Then 
\begin{align}
\E_{\stackrel{\rightarrow}{a}}
\left[\int_0^{\tau}h(B_s)\d s \right]=\int_0^{\infty}(a\land t)h(t)\d t.\label{eq:ShigekawaProp3.9}
\end{align}
\end{lem}
Based on Lemma~\ref{lem:ShigekawaProp3.9}, we can obtain the following generalization of \cite[Proposition~6.1]{Shigekawa1}

\begin{prop}[{\cite[Proposition~6.1]{Shigekawa1}}]\label{prop:ShigekawaProp6.1}
Let $\nu$ be a non-negative smooth measure on $M$ associated with a PCAF $A^{\nu}$ in Revuz correspondence 
under ${\bf X}$
and $j$ be a non-negative 
measurable function on $X\times[0,+\infty[$. Denote by $\wh{A}^{\,\wh{\nu}}$ the PCAF of $\wh{\bf X}$ associated with the smooth measure $\wh{\nu}:=\nu\otimes m$ with respect to $\wh{\bf X}$. 
Then for all~$a\in]0,+\infty[$ and Borel measurable non-negative function $j$ on $X\times\R$
\begin{align}
\E_{\m\otimes\delta_a}\left[\int_0^{\tau}j(\wh{X}_s)\d \wh{A}_s^{\,\wh{\nu}} \right]=\int_M\int_0^{\infty}(a\land t)j(x,t)\nu(\d x)\d t.\label{eq:ShigekawaProp3.10}
\end{align}  
\end{prop}
Let ${\rm Reg}(T^*\!M) \odot C_c^{\infty}(\R)$ be the algebraic tensor product of  ${\rm Reg}(T^*\!M)$ and $C_c^{\infty}(\R)$, i.e., any element of ${\rm Reg}(T^*\!M) \odot C_c^{\infty}(\R)$ forms a finite linear combination of $\theta\otimes\varphi$ with 
$\theta\in  {\rm Reg}(T^*\!M)$ and $\varphi\in C_c^{\infty}(\R)$, 
 which is a dense subspace of the tensor product $L^2(T^*\!M)\otimes L^2(\R)$ of Hilbert spaces $L^2(T^*\!M)$ and $L^2(\R)$.  
We now define the quadratic form $\wh{\mathscr{E}}^{\,\rm HK}$ on ${\rm Reg}(T^*\!M) \odot C_c^{\infty}(\R)$ by 
\begin{align*}
\wh{\mathscr{E}}^{\rm HK}(\theta,\eta):=\int_{\R}\mathscr{E}^{\rm HK}(\theta(\cdot,t),\omega(\cdot,t))\d t
+\int_{\R}\langle \dot{\theta}(\cdot,t),\dot{\eta}(\cdot,t)\rangle_{L^2(T^*\!M)}\d t,\quad \theta,\eta\in {\rm Reg}(T^*\!M) \odot C_c^{\infty}(\R).
\end{align*}

\begin{lem}\label{lem:Closable}
The quadratic form $(\wh{\mathscr{E}}^{\,\rm HK}, {\rm Reg}(T^*\!M) \odot C_c^{\infty}(\R))$ is closable on $L^2(T^*\!M)\otimes L^2(\R)$. Denote by $(\wh{\mathscr{E}}^{\,\rm HK}, D(\wh{\mathscr{E}}^{\,\rm HK}))$ its closure on $L^2(T^*\!M)\otimes L^2(\R)$. 
Then, $(\wh{\mathscr{E}}^{\,\rm HK}, D(\wh{\mathscr{E}}^{\,\rm HK}))$ is a closed bilinear form on $L^2(T^*\!M)\otimes L^2(\R)$.
\end{lem}
\begin{proof}[{\bf Proof}] 
For $\theta\in {\rm Reg}(T^*\!M) \odot C_c^{\infty}(\R)$, it is easy to see $\wh{\Delta\hspace{-0.29cm}\Delta}^{\rm HK}\theta:=\left(\frac{\partial^2}{\partial t^2}+\DD \right)\theta\in L^2(T^*\!M)\otimes L^2(\R)$. Then 
\begin{align*}
\wh{\mathscr{E}}^{\,\rm HK}(\theta,\eta)=\left(-\wh{\Delta\hspace{-0.29cm}\Delta}^{\rm HK}\theta,\eta \right)_{L^2(T^*\!M)\otimes L^2(\R)}\quad\text{ for }\quad\theta,\eta\in {\rm Reg}(T^*\!M) \odot C_c^{\infty}(\R).
\end{align*}
Then the closability of $(\wh{\mathscr{E}}^{\,\rm HK}, D(\wh{\mathscr{E}}^{\,\rm HK}))$ on $L^2(T^*\!M)\otimes L^2(\R)$ follows 
\cite[Chapter I, Proposition~3.3]{MR}. 
\end{proof} 
\begin{remark}\label{rem:NormDifference}
{\rm The norm $\|\cdot\|_{H^{1,2}(T^*\!M)\otimes H^{1,2}(\R)}$ of the tensor product $H^{1,2}(T^*\!M)\otimes H^{1,2}(\R)$ of Hilbert spaces $H^{1,2}(T^*\!M)$ and $H^{1,2}(\R)$ does not coincide with the norm $\|\cdot\|_{\wh{\mathscr{E}}_1^{\,\rm HK}}$ derived from 
$(\wh{\mathscr{E}}^{\,\rm HK}, D(\wh{\mathscr{E}}^{\,\rm HK}))$ on $L^2(T^*\!M)\otimes L^2(\R)$. Indeed, for $\theta\in {\rm Reg}(T^*\!M)$ and $\varphi\in  C_c^{\infty}(\R)$.
\begin{align*}
\|\theta\otimes\varphi\|_{H^{1,2}(T^*\!M)\otimes H^{1,2}(\R)}^2&=\|\theta\|_{H^{1,2}(T^*\!M)}^2\|\varphi\|_{H^{1,2}(\R)}^2\\
&=\left(\mathscr{E}^{\rm HK}(\theta,\theta)+\|\theta\|_{L^2(T^*\!M)}^2 \right)\left(\|\dot{\varphi}\|_{L^2(\R)}^2+\|\varphi\|_{L^2(\R)}^2\right)\\
\|\theta\otimes\varphi\|_{\wh{\mathscr{E}}_1^{\,\rm HK}}^2&=
\wh{\mathscr{E}}^{\,\rm HK}_1(\theta\otimes\varphi,\theta\otimes\varphi)\\
&=
\mathscr{E}^{\rm HK}(\theta,\theta)\|\varphi\|_{L^2(\R)}^2+\|\theta\|_{L^2(T^*\!M)}^2\|\dot{\varphi}\|_{L^2(\R)}^2+\|\theta\|_{L^2(T^*\!M)}^2\|\varphi\|_{L^2(\R)}^2.
\end{align*}
}\end{remark}
\begin{lem}\label{lem:TEnsorProduct}
We have the inclusion $H^{1,2}(T^*\!M)\odot H^{1,2}(\R)\subset D(\wh{\mathscr{E}}^{\rm HK})$. The map 
$H^{1,2}(T^*\!M)\times H^{1,2}(\R)\to H^{1,2}(T^*\!M)\odot H^{1,2}(\R)$ is continuous.
\end{lem}
\begin{proof}[{\bf Proof}] 
It suffices to show that $\theta\otimes\varphi\in D(\wh{\mathscr{E}}^{\rm HK})$ for $\theta\in H^{1,2}(T^*\!M)$ and $\varphi H^{1,2}(\R)$. Take an approximating sequence $\{\theta_n\}\subset {\rm Reg}(T^*\!M)$ (resp.~$\{\varphi_n\}\subset C_c^{\infty}(\R)$) in $H^{1,2}(T^*\!M)$ (resp.~in $H^{1,2}(\R)$) to $\theta\in H^{1,2}(T^*\!M)$ (resp.~$\varphi\in H^{1,2}(\R)$).
Then it is easy to see that $\{\theta_n\otimes\varphi_n\}\subset H^{1,2}(T^*\!M)\odot H^{1,2}(\R)$  is an 
$\wh{\mathscr{E}}^{\,\rm HK}_1$-Cauchy sequence. This yields the conclusion and the following identity: for $\theta,\eta\in H^{1,2}(T^*\!M)$ and $\varphi,\psi\in H^{1,2}(\R)$, 
\begin{align*}
\wh{\mathscr{E}}^{\,\rm HK}(\theta\otimes\varphi,\eta\otimes\psi)=\mathscr{E}^{\rm HK}(\theta,\eta)(\varphi,\psi)_{L^2(\R)}+
\langle \theta,\eta\rangle_{L^2(T^*\!M)}(\dot{\varphi},\dot{\psi})_{L^2(\R)},
\end{align*}
which implies the desired continuity.
\end{proof} 
\begin{lem}\label{lem:CharacterizationDomain}
Suppose $\eta\in L^2(T^*\!M)\otimes L^2(\R)$. Then the following are equivalent to each other:
\begin{enumerate}
\item[\rm (1)] $\eta\in D(\wh{\mathscr{E}}^{\,\rm HK})$. 
\item[\rm (2)] For a.e.~$t\in\R$, we have $\eta(\cdot, t)\in D({\mathscr{E}}^{\,\rm HK})$ and $\dot{\eta}(\cdot,t)\in L^2(T^*\!M)$. Moreover, $\int_{\R}{\mathscr{E}}^{\,\rm HK}_1(\eta(\cdot,t),\eta(\cdot,t))\d t<\infty$ and $\int_{\R}\|\dot{\eta}(\cdot,t)\|_{L^2(T^*\!M)}^2\d t<\infty$. 
\end{enumerate}
Here $\dot{\eta}(\cdot,t)$ denotes the $L^2(T^*\!M)$-valued 
distributional derivative of $\eta(\cdot,t)$ with respect to $t\in \R$
\end{lem}
\begin{proof}[{\bf Proof}] 
(1) $\Longrightarrow$ (2): Suppose $\eta\in D(\wh{\mathscr{E}}^{\,\rm HK})$. Let $\eta_n\in {\rm Reg}(T^*\!M)\odot C_c^{\infty}(\R)$ be an $\wh{\mathscr{E}}^{\,\rm HK}_1$-approximating sequence to $\eta$. By taking a subsequence if necessary we may assume
\begin{align}
\sum_{k=1}^{\infty}2^k\wh{\mathscr{E}}^{\,\rm HK}_1(\eta_{k+1}-\eta_k, \eta_{k+1}-\eta_k)<\infty.\label{eq:CauchyCriteia}
\end{align} 
Then, for each $T\in]0,+\infty[$, 
\begin{align*}
\int_0^T&\sum_{k=1}^{\infty}\mathscr{E}^{\,\rm HK}_1(\eta_{k+1}(\cdot,t)-\eta_k(\cdot,t), \eta_{k+1}(\cdot,t)-\eta_k(\cdot,t))^{\frac12}\d t\\
&\leq \sum_{k=1}^{\infty} \sqrt{T}2^{-\frac{k}{2}}\left(2^k\int_0^T{\mathscr{E}}^{\,\rm HK}_1(\eta_{k+1}(\cdot,t)-\eta_k(\cdot,t), \eta_{k+1}(\cdot,t)-\eta_k(\cdot,t))\d t\right)^{\frac12}\\
&=\sqrt{T}\left(\sum_{k=1}^{\infty}2^{-k} \right)^{\frac12}\left(\sum_{k=1}^{\infty}2^k\wh{\mathscr{E}}^{\,\rm HK}_1(\eta_{k+1}-\eta_k, \eta_{k+1}-\eta_k)\right)^{\frac12}<\infty.
\end{align*}
Thus, 
\begin{align*}
\sum_{k=1}^{\infty}\mathscr{E}^{\rm HK}_1(\eta_{k+1}(\cdot,t)-\eta_k(\cdot,t), \eta_{k+1}(\cdot,t)-\eta_k(\cdot,t))^{\frac12}<\infty\quad\text{ for\;\; a.e. }\quad t\in\R.
\end{align*}
This means that $\{\eta_n(\cdot,t)\}$ is an $\mathscr{E}^{\,\rm HK}_1$-Cauchy sequence for a.e.~$t\in\R$. On the other hand, 
by taking a further subsequence, 
$\{\eta_n(\cdot,t)\}$ converges to $\eta(\cdot,t)$ in $L^2(T^*\!M)$ for a.e.~$t\in\R$. Therefore, $\eta(\cdot,t)\in H^{1,2}(T^*\!M)=D(\mathscr{E}^{\,\rm HK})$ and $\eta_n(\cdot,t)\to\eta(\cdot,t)$ in $H^{1,2}(T^*\!M)$ for a.e.~$t\in\R$. Moreover, 
\begin{align*}
\int_{\R}\mathscr{E}^{\rm HK}_1(\eta(\cdot,t),\eta(\cdot,t))\d t&=\int_{\R}\lim_{n\to\infty}
\mathscr{E}^{\rm HK}_1(\eta_n(\cdot,t),\eta_n(\cdot,t))\d t\\
&\leq\varliminf_{n\to\infty}\int_{\R}\mathscr{E}^{\rm HK}_1(\eta_n(\cdot,t),\eta_n(\cdot,t))\d t\\
&=\lim_{n\to\infty}\wh{\mathscr{E}}^{\,\rm HK}_1(\eta_n,\eta_n)=\wh{\mathscr{E}}^{\,\rm HK}_1(\eta,\eta)<\infty.
\end{align*} 
From \eqref{eq:CauchyCriteia}, we can similarly deduce 
\begin{align*}
\sum_{k=1}^{\infty}\|\dot{\eta}_{k+1}(\cdot,t)-\dot{\eta}(\cdot,t)\|_{L^2(T^*\!M)}<\infty\quad\text{ for\;\; a.e. }\quad t\in\R 
\end{align*}
as above. 
Hence $\{\dot{\eta}_n(\cdot,t)\}$ is an $L^2(T^*\!M)$-Cauchy sequence  for a.e.~$t\in\R$. Let $v(\cdot,t)$ be its limit in 
$L^2(T^*\!M)$. By using an integration by parts, we see for $\theta\in {\rm Reg}(T^*\!M)$ and $\varphi\in C_c^{\infty}(\R)$
\begin{align*}
\langle v,\theta\otimes\varphi\rangle_{L^2(T^*\!M)\otimes L^2(\R)}&=\int_{\R}\lim_{n\to\infty}\langle \dot{\eta}_n(\cdot,t),\theta\rangle_{L^2(T^*\!M)}\varphi(t)\d t\\
&=-\lim_{n\to\infty}\int_{\R}\langle \eta_n(\cdot,t),\theta\rangle_{L^2(T^*\!M)}\dot{\varphi}(t)\d t\\
&=-\int_{\R}\langle \eta(\cdot,t),\theta\rangle_{L^2(T^*\!M)}\dot{\varphi}(t)\d t\\
&=\left\langle -\int_{\R}\eta(\cdot,t)\dot{\varphi}(t)\d t,\theta \right\rangle_{L^2(T^*\!M)}=\left\langle \int_{\R}\dot{\eta}(\cdot,t){\varphi}(t)\d t,\theta \right\rangle_{L^2(T^*\!M)}\\
&=\langle \dot{\eta},\theta\otimes\varphi\rangle_{L^2(T^*\!M)\otimes L^2(\R)},
\end{align*}
hence $v(\cdot,t)=\dot{\eta}(\cdot,t)$ for a.e.~$t\in\R$. Therefore 
\begin{align*}
\int_{\R}(\|\dot{\eta}(\cdot,t)\|_{L^2(T^*\!M)}^2+\|{\eta}(\cdot,t)\|_{L^2(T^*\!M)}^2)\d t&=\int_{\R}\lim_{n\to\infty}
(\|\dot{\eta}_n(\cdot,t)\|_{L^2(T^*\!M)}^2+\|{\eta}_n(\cdot,t)\|_{L^2(T^*\!M)}^2)\d t\\
&\leq\varliminf_{n\to\infty}\int_{\R}(\|\dot{\eta}_n(\cdot,t)\|_{L^2(T^*\!M)}^2+\|{\eta}_n(\cdot,t)\|_{L^2(T^*\!M)}^2)\d t\\
&\leq \lim_{n\to\infty}\wh{\mathscr{E}}^{\,\rm HK}(\eta_n,\eta_n)=\wh{\mathscr{E}}^{\,\rm HK}(\eta,\eta)<\infty.
\end{align*}
(2) $\Longrightarrow$ (1): Suppose the conditions in (2) for $\eta\in  L^2(T^*\!M)\otimes L^2(\R)$.
Let $(\wh{P}_t^{\,\rm HK})$ be the semigroup on $L^2(T^*\!M)\otimes L^2(\R)$ associated to 
the quadratic form $(\wh{\mathscr{E}}^{\,\rm HK}, D(\wh{\mathscr{E}}^{\,\rm HK}))$ on $L^2(T^*\!M)\otimes L^2(\R)$.  
Then we see 
\begin{align}
\wh{P}_t^{\,\rm HK}\eta(x,s)={P}_t^{\,\rm HK}\left(P_t^w\eta(\cdot,s) \right)(x)\quad \text{ for }\quad \wh\m\text{-a.e.~}(x,s),
\label{eq:SemigroupConincidence}
\end{align}
where 
\begin{align*}
P_t^w\eta(\cdot,s):=\frac{1}{\sqrt{4\pi t}}\int_{\R}e^{-\frac{(u-s)^2}{4\pi t}}\eta(\cdot,u)\d u
\end{align*}
is defined to be a Bochner integral on $L^2(T^*\!M)$. Similarly, we can define $P_t^w\d\eta(\cdot,s)$ and 
$P_t^w\d^*\eta(\cdot,s)$. Then 
\begin{align*}
\wh{\mathscr{E}}^{\,\rm HK}&(\wh{P}_t^{\,\rm HK}\eta,\wh{P}_t^{\,\rm HK}\eta)\\
&=\int_{\R}{\mathscr{E}}^{\rm HK}({P}_t^{\,\rm HK}P_t^w\eta(\cdot,s),{P}_t^{\,\rm HK}P_t^w\eta(\cdot,s))\d s+
\int_{\R}\left|P_t^{\rm HK}\frac{\partial}{\partial s}P_t^w\eta(\cdot,s) \right|_{L^2(T^*\!M)}^2\d s\\
&\leq 
\int_{\R}{\mathscr{E}}^{\rm HK}(P_t^w\eta(\cdot,s),P_t^w\eta(\cdot,s))\d s+
\int_{\R}\left|\frac{\partial}{\partial s}P_t^w\eta(\cdot,s) \right|_{L^2(T^*\!M)}^2\d s\\
&\leq 
\int_{\R}{\mathscr{E}}^{\rm HK}(\eta(\cdot,s),\eta(\cdot,s))\d s+
\int_{\R}\left|\frac{\partial}{\partial s}\eta(\cdot,s) \right|_{L^2(T^*\!M)}^2\d s<\infty.
\end{align*}
In the last inequality, we use $|P_t^w\d\eta|\leq P_t^w|\d \eta|$, $|P_t^w\d^*\eta|\leq P_t^w|\d^* \eta|$, and $ \frac{\partial}{\partial s}P_t^w\eta(\cdot,s)=P_t^w\frac{\partial}{\partial s}\eta(\cdot,s)$ for $
\wh{\m}$-a.e.~$(x,s)$. Thus we have $(\wh{P}_t^{\,\rm HK}\eta)_{t\geq0}$ is $\wh{\mathscr{E}}^{\,\rm HK}_1$-bounded. 
Therfore, $\eta\in  D(\wh{\mathscr{E}}^{\,\rm HK})$. 
\end{proof}

\begin{thm}\label{thm:domain}
Take $\omega=\omega(x,t)\in D(\wh{\mathscr{E}}^{\,\rm HK})$. Then $|\omega|\in D(\wh{\mathscr{E}})$ and 
\begin{align*}
\wh{\mathscr{E}}^{\,\wh\kappa}(|\omega|,|\omega|)\leq\wh{\mathscr{E}}^{\,\rm HK}(\omega,\omega).
\end{align*}
\end{thm}
\begin{proof}[{\bf Proof}] 
Applying the Fatou's lemma for Dirichlet form on $L^2(\R)$ associated to the one-dimensional Brownian motion $(B_t)_{t\geq0}$ 
(see \cite{Schmuland:Fatou} for Fatou's lemma for Dirichlet forms),
\begin{align}
\int_{\R}\left|\frac{\partial |\omega|}{\partial t} \right|^2\d t&\leq\varliminf_{n\to\infty}\int_{\R}\left|\frac{\partial(|\omega|^2+\frac{1}{n^2})^{\frac12}}{\partial t} \right|^2\d t\notag \\
&\leq \varliminf_{n\to\infty}\int_{\R}\frac{\langle \omega,\dot{\omega}\rangle ^2}{|\omega|^2+\frac{1}{n^2}}\d t
\leq \varliminf_{n\to\infty}\int_{\R}\frac{|\omega|^2|\dot{\omega}|^2}{|\omega|^2+\frac{1}{n^2}}\d t\notag\\
&\leq\int_{\R}|\dot{\omega}|^2\d t.\label{eq:HSUTime}
\end{align} 
On the other hand, Hess-Schrader-Uhlenbrock inequality \eqref{eq:HSU} tells us the following (see \cite{Kw:HessSchraderUhlenbrock})
\begin{align}
\mathscr{E}^{\kappa}(|\omega|(\cdot,t),|\omega|(\cdot,t))\leq\mathscr{E}^{\rm HK}(\omega(\cdot,t),\omega(\cdot,t)).\label{eq:HSUForm}
\end{align} 
Combining \eqref{eq:HSUTime} and \eqref{eq:HSUForm}, we have 
\begin{align*}
\wh{\mathscr{E}}^{\,\wh\kappa}(|\omega|,|\omega|)&=\int_{\R}\mathscr{E}^{\kappa}(|\omega|(\cdot,t),|\omega|(\cdot,t))\d t+\int_M\int_{\R}
\left|\frac{\partial|\omega|}{\partial t} \right|^2\d\m\d t\\
&\leq \int_{\R}\mathscr{E}^{\rm HK}(\omega(\cdot,t),\omega(\cdot,t))\d t+\int_M\int_{\R}
\left|\dot\omega\right|^2\m(\d x)\d t\d\m=\wh{\mathscr{E}}^{\rm HK}(\omega,\omega).
\end{align*}
\end{proof} 
\begin{cor}\label{cor:domain}
Let $\eta=\eta(x,t):=Q_{|t|}^{(\alpha),{\rm HK}}\theta(x)$ with $\theta\in D(\DD)\cap L^{\infty}(T^*\!M)$ and $\alpha>C_{\kappa}$. 
Then $\eta\in D(\wh{\mathscr{E}}^{\,\rm HK})$, hence 
$|\eta|\in D(\wh{\mathscr{E}}^{\,\wh\kappa})\cap L^{\infty}(\wh{M};\wh\m)=D(\wh{\mathscr{E}})\cap L^{\infty}(\wh{M};\wh\m)$. 
In particular, $|\eta|^2\in D(\wh{\mathscr{E}})\cap L^{\infty}(\wh{M};\wh\m)$. 
Moreover, for $p\in]0,+\infty[$ and $\ep>0$, $\eta_{\eps}:=(|\eta|^2+\eps^2)^{\frac{p}{2}}-\eps^p$ satisfies $\eta_{\eps}\in D(\wh{\mathscr{E}})\cap L^{\infty}(\wh{M};\wh\m)$.
\end{cor}
\begin{proof}[{\bf Proof}] 
From \eqref{eq:QtContraction}, we easily see 
\begin{align*}
\int_{\R}\|\eta(\cdot,t)\|_{L^2(T^*\!M)}^2\d t\leq C^2\|\theta\|_{L^2(T^*\!M)}^2\int_{\R}e^{-2\sqrt{\alpha-C_{\kappa}}|t|}\d t<\infty,
\end{align*}
hence $\eta\in L^2(T^*\!M)\otimes L^2(\R)$. 
For a fixed $t>0$, we see
\begin{align*}
\mathscr{E}^{\rm HK}_{\alpha}(\eta(\cdot,t),\eta(\cdot,t))&=((\alpha-\DD)\eta(\cdot,t),\eta(\cdot,t))_{L^2(T^*\!M)}\\
&=\|\sqrt{\alpha-\DD}\eta(\cdot,t) \|_{L^2(T^*\!M)}\\
&=\int_M|\dot{\eta}(\cdot,t)|^2\d\m,
\end{align*}
where we use $\dot{\eta}=\sqrt{\alpha-\DD}\eta$ in $L^2(T^*\!M)$. 
Applying \eqref{eq:QtContraction} again, 
 \begin{align*}
 \int_{\R}\|(\alpha-\DD)\eta(\cdot,t) \|_{L^2(T^*\!M)}^2\d t&=\int_{\R}\| Q_{|t|}^{(\alpha),{\rm HK}}(\alpha-\DD)\theta\|_{L^2(T^*\!M)}^2\d t\\
 &\leq \|(\alpha-\DD)\theta\|_{L^2(T^*\!M)}^2 C^2\int_{\R}e^{-2\sqrt{\alpha-C_{\kappa}}|t|}\d t<\infty. 
 \end{align*} 
 Therefore, we obtain 
 \begin{align*}
 \wh{\mathscr{E}}^{\,\rm HK}_{\alpha}(\eta,\eta)
 &=\int_{\R}{\mathscr{E}}^{\,\rm HK}_{\alpha}(\eta(\cdot,t),\eta(\cdot,t))\d t+\int_{\R}\|\dot{\eta}(\cdot,t)\|_{L^2(T^*\!M)}^2\d t\\
 &=2\int_{\R}{\mathscr{E}}^{\,\rm HK}_{\alpha}(\eta(\cdot,t),\eta(\cdot,t))\d t\\
 &\leq2\sqrt{\int_{\R}\|(\alpha-\DD)\eta(\cdot,t)\|_{L^2(T^*\!M)}^2\d t}\sqrt{\int_{\R}\|\eta(\cdot,t)\|_{L^2(T^*\!M)}^2\d t}<\infty.
 \end{align*}
 By Lemma~\ref{lem:CharacterizationDomain}, this inequality shows $\eta\in D(\wh{\mathscr{E}}^{\,\rm HK})$. 
 Theorem~\ref{thm:domain} also shows $|\eta|\in D(\wh{\mathscr{E}})\cap L^{\infty}(\wh{M};\wh{\m})$ . 
 The final assertion follows from the fact that $T(x):=\Psi(|x|)=(|x|+\eps^2)^{\frac{p}{2}}-\eps^p$ satisfies $T(0)=0$ and 
 $|T(x)-T(y)|\leq\frac{p}{2}\eps^{p-2}|x-y|$ for $x,y\in \R$ (see \cite[\text{$(\mathscr{E},4)''$}]{FOT}). 
\end{proof} 

By Corollary~\ref{cor:domain}, we can apply Fukushima's decomposition to 
$|\eta|^2$ for $\eta=Q_{|t|}^{(\alpha),{\rm HK}}\theta$ with 
$\theta\in  D(\DD)\cap L^{\infty}(T^*\!M)$ and $\alpha>C_{\kappa}$ with respect to $\wh{\bf X}$. 
That is, there exists 
a martingale additive functional of finite energy $M^{[|\eta|^2]}$ and a continuous additive functional of zero energy $N^{[v]}$ such that 
\begin{align}
\wt{|\eta|^2}(\widehat{X}_t)-\wt{|\eta|^2}(\widehat{X}_0)=\widehat{M}_t^{[|\eta|^2]}+\widehat{N}_t^{[|\eta|^2]}\quad t\geq0\quad \P_{\hat{x}}\text{-a.s.~for }\wh{\mathscr{E}}\text{-q.e.~}\hat{x},\label{eq:FukushimaDecomp} 
\end{align}
where $\wt{|\eta|^2}$ is an $\widehat{\mathscr{E}}$-quasi-continuous $\widehat{\m}$-version of $|\eta|^2$, where  
$\eta=Q_{|t|}^{(\alpha),{\rm HK}}\theta$ with 
$\theta\in  D(\DD)\cap L^{\infty}(T^*\!M)$.  
See \cite[Theorem~5.2.2]{FOT} for Fukushima's decomposition theorem. Thanks to \cite[Theorem~5.2.3]{FOT},  we know that 
\begin{align}
\langle \widehat{M}^{[|\eta|^2]}\rangle_t=\int_0^t \left\{\left|\nabla|\eta|^2\right|^2(\widehat{X}_s)+\left( \frac{\partial |\eta|^2}{\partial a}(\widehat{X}_s)\right)^2 \right\}\d s,\quad t\geq0.\label{eq:quadraticVariation}
\end{align}
See also \cite[Theorem~5.1.3 and Example~5.2.1]{FOT} for details.  
Similarly, for $\eta_{\eps}=\Psi(|\eta|^2)=(|\eta|^2+\eps^2)^{\frac{p}{2}}-\eps^p\in D(\wh{\mathscr{E}})$ with $p\in]0,+\infty[$ we can consider 
\begin{align}
\wt{\eta}_{\eps}(\widehat{X}_t)-\wt{\eta}_{\eps}(\widehat{X}_0)=\widehat{M}_t^{[\eta_{\eps}]}+\widehat{N}_t^{[\eta_{\eps}]}\quad t\geq0\quad \P_{\hat{x}}\text{-a.s.~for }\wh{\mathscr{E}}\text{-q.e.~}\hat{x}\label{eq:FukushimaDecomp*} 
\end{align}
with respect to $\wh{\bf X}$.

\section{Expressions for the CAFs $\widehat{N}_t^{[|\eta|^2]}$ and $\widehat{N}_t^{[\eta_{\eps}]}$ of zero energy}\label{sec:CAF}
\begin{lem}[Bochner formula]\label{lem:LaplacianEta}
Let $\eta\in D(\DD)\cap L^{\infty}{(T^*\!M)}$. 
Then $|\eta|^2\in 
D_f(\text{\boldmath$\Delta$})_+$ with 
\begin{align}
\frac12\text{\boldmath$\Delta$}|\eta|^2=\left[\langle \eta,\DD\eta\rangle+|\nabla \eta^{\sharp}|_{\rm HS}^2\right]\m+{\bf Ric}(\eta^{\sharp},\eta^{\sharp}). \label{eq:MeasureWeitzenboeck}
\end{align}
\end{lem}
\begin{proof}[{\bf Proof}] 
Since $\eta\in D(\DD)\cap L^{\infty}(T^*\!M)\subset H^{1,2}(T^*\!M)\cap L^{\infty}(T^*\!M)$,  $|\eta|\in D(\mathscr{E})\cap L^{\infty}(M;\m)$ is proved in \cite{Kw:HessSchraderUhlenbrock}, hence $|\eta|^2\in D(\mathscr{E})\cap L^{\infty}(M;\m)$. 
When $\eta\in {\rm Reg}(T^*\!M)$, by \cite[Lemma~8.2]{Braun:Tamed2021}, $|\eta|^2\in D(\text{\boldmath$\Delta$}^{2\kappa})$, then 
$|\eta|^2\in D(\text{\boldmath$\Delta$})$ by Lemma~\ref{lem:Inclusion}. In this case, we have \eqref{eq:MeasureWeitzenboeck} 
by \eqref{eq:RicciMeasure}. The expression  \eqref{eq:MeasureWeitzenboeck} yields $|\eta|^2\in D_f(\text{\boldmath$\Delta$}^{2\kappa})$, 
i.e., $\text{\boldmath$\Delta$}^{2\kappa}|\eta|^2$ is a finite signed smooth measure, because the right-hand-side of 
\eqref{eq:MeasureWeitzenboeck} is a signed finite smooth measure for any $\eta\in D(\DD)$ by Theorem~\ref{thm:Theorem8.9} and 
$|\eta|\in D(\mathscr{E})$ implies $\int_M\wt{|\eta|^2}\d|\kappa|<\infty$. Therefore, we have $|\eta|^2\in D_f(\text{\boldmath$\Delta$})$ 
by Lemma~\ref{lem:Inclusion} and \eqref{eq:MeasureWeitzenboeck}. This implies that for any $\phi\in D(\mathscr{E})$ and $\eta\in {\rm Reg}(T^*\!M)$  
\begin{align}
-\frac12\mathscr{E}^{2\kappa}(\phi,|\eta|^2)=\int_M\phi\left[\langle \eta,\DD\eta\rangle+|\nabla \eta^{\sharp}|_{\rm HS}^2 \right]\d\m+\int_M\tilde{\phi}\,\d\,{\bf Ric}^{\kappa}(\eta^{\sharp},\eta^{\sharp}),\label{eq:OriginalExp1}\\
-\frac12\mathscr{E}(\phi,|\eta|^2)=\int_M\phi\left[\langle \eta,\DD\eta\rangle+|\nabla \eta^{\sharp}|_{\rm HS}^2 \right]\d\m+\int_M\tilde{\phi}\,\d\,{\bf Ric}(\eta^{\sharp},\eta^{\sharp}).\label{eq:OriginalExp2}
\end{align}  

Finally we prove the assertion $|\eta|^2\in D_f(\text{\boldmath$\Delta$})$ with \eqref{eq:MeasureWeitzenboeck} for the case 
$\eta=Q_t^{(\alpha),{\rm HK}}\theta$, $\theta\in D(\DD)\cap L^{\infty}(T^*\!M)$.  
Let $\{\eta_n\}\subset {\rm Reg}(T^*\!M)(\subset L^{\infty}(T^*\!M))$ be an $H^{1,2}(T^*\!M)$-approximating sequence to 
$\eta$ satisfying $\sup_{n\in\N}\|\eta_n\|_{L^{\infty}(T^*\!M)}<\infty$. Fix $\phi\in {\rm Test}(M)$. 
It is easy to see that $\{|\eta_n|^2\}(\subset D(\mathscr{E}))$ $L^2$-strongly converges to $|\eta|^2\in D(\mathscr{E})$.
We can prove that $\{|\eta_n|^2\}(\subset D(\mathscr{E}))$ is $\mathscr{E}_1$-weakly (and $\mathscr{E}$-weakly) convergent to $|\eta|^2\in D(\mathscr{E})$. Indeed, for $\phi\in {\rm Test}(M)$, we can easily get that $\{\mathscr{E}_1(|\eta_n|^2,\phi)\}$ and $\{\mathscr{E}(|\eta_n|^2,\phi)\}$ are Cauchy sequences 
for such $\eta_n\in{\rm Reg}(T^*\!M)$. 
Since ${\rm Test}(M)$ is $\mathscr{E}_1^{1/2}$-dense in $D(\mathscr{E})$, they are 
Cauchy sequences and converges to $\{\mathscr{E}_1(|\eta|^2,\phi)\}$ and $\{\mathscr{E}(|\eta|^2,\phi)\}$  
for any $\phi\in D(\mathscr{E})$, respectively. Thus, we have the expressions \eqref{eq:OriginalExp1} and \eqref{eq:OriginalExp2} for $\phi\in D(\mathscr{E})$. Therefore, 
\eqref{eq:OriginalExp2}  for $\phi\in D(\mathscr{E})$ implies $|\eta|^2\in D_f(\text{\boldmath$\Delta$})$ and \eqref{eq:MeasureWeitzenboeck}. 
\end{proof} 
\begin{cor}\label{cor:LaplacianEta}
Let $\eta=Q_t^{(\alpha),{\rm HK}}\theta$ with $\theta\in D(\DD)\cap L^{\infty}(T^*\!M)$. Then, for $\alpha\geq C_{\kappa}$ 
\begin{align}
\frac{\;1\;}{2}\left(\frac{\partial^2}{\partial t^2}+\text{\boldmath$\Delta$}\right)|\eta|^2&=\left[\alpha|\eta|^2+
\left|\overline{\nabla}\eta \right|^2 \right]\m+{\bf Ric}(\eta^{\sharp},\eta^{\sharp}),\label{eq:Identity1}\\
\frac{\;1\;}{2}\left(\frac{\partial^2}{\partial t^2}+\Delta_{\ll}\right)|\eta|^2&=\left[\alpha|\eta|^2+
\left|\overline{\nabla}\eta \right|^2 \right]+{\mathfrak{Ric}}(\eta^{\sharp},\eta^{\sharp})\quad\m\text{-a.e.}\label{eq:Identity2}
\end{align}
Here the left hand side of \eqref{eq:Identity1} is understood as 
\begin{align}
\left( \frac{\partial^2}{\partial t^2}+\text{\boldmath$\Delta$}\right)|\eta|^2:=\frac{\partial^2|\eta|^2}{\partial t^2}\m+
\text{\boldmath$\Delta$}|\eta|^2.\label{eq:Formula4}
\end{align}
\end{cor}
\begin{proof}[{\bf Proof}] 
We calculate
\begin{align}
\frac{\,1\,}{2}\frac{\partial^2}{\partial t^2}|\eta|^2=|\dot\eta|^2+\langle \eta,\dot{\eta}\rangle =|\dot{\eta}|^2+\langle \eta,(\alpha-\DD)\eta\rangle, \label{eq:Formula3}
\end{align}
because $\eta=Q_t^{(\alpha),{\rm HK}}\theta$ is the solution of 
\begin{align*}
\frac{\partial^2}{\partial t^2}\eta(x,t)=(\alpha-\DD)\eta(x,t)
\end{align*}
in $L^2(M;\m)$ for $\alpha\geq C_{\kappa}$. 
Adding \eqref{eq:Formula3}$\times\m$ to \eqref{eq:MeasureWeitzenboeck}, we obtain \eqref{eq:Identity1} and \eqref{eq:Identity2}. 
\end{proof} 

\begin{thm}\label{thm:CAFExpression}
Let $\eta=Q_t^{(\alpha),{\rm HK}}\theta$ with $\theta\in D(\DD)\cap L^{\infty}(T^*\!M)$ and $\alpha> C_{\kappa}$.
The CAF of zero energy part $\widehat{N}_t^{[|\eta|^2]}$ appeared in the Fukushima's decomposition \eqref{eq:FukushimaDecomp} 
is given by 
\begin{align}
\widehat{N}_t^{[|\eta|^2]}=2\int_0^t(\alpha|\eta|^2+|\overline{\nabla}\eta|^2)(\wh{X}_s)\d s+2\wh{A}_t^{\,{\bf Ric}
(\eta^{\sharp},\eta^{\sharp})\otimes m}.
\label{eq:CAFExpression}
\end{align}
Here $\wh{A}_t^{\,{\bf Ric}
(\eta^{\sharp},\eta^{\sharp})\otimes m}$ is the PCAF associated with the smooth measure 
${\bf Ric}
(\eta^{\sharp},\eta^{\sharp})\otimes m$ with respect to $\wh{\bf X}$. 
Moreover, for $\eta_{\eps}=(|\eta|^2+\eps^2)^{\frac{p}{2}}-\eps^p$ with $p\in]0,+\infty[$
\begin{align}
\widehat{N}_t^{[\eta_{\eps}]}&=p\int_0^t\left(|\eta|^2(\wh{X}_s)+\eps^2 \right)^{\frac{p}{2}-1}\left(\alpha|\eta|^2(\wh{X}_s)+|\overline{\nabla}\eta|^2(\wh{X}_s)\right)\d s\notag\\
&\hspace{1cm}+p\int_0^t\left(|\eta|^2(\wh{X}_s)+\eps^2 \right)^{\frac{p}{2}-1}\d \wh{A}_s^{\,{\bf Ric}(\eta^{\sharp},\eta^{\sharp})\otimes m}\label{eq:CAFExpression*}\\
&\hspace{2cm}+\frac{p}{4}(p-2)\int_0^t\left(|\eta|^2(\wh{X}_s)+\eps^2 \right)^{\frac{p}{2}-2}
\left(\left|\nabla|\eta|^2 \right|^2(\wh{X}_s)+\left|\frac{\partial |\eta|^2}{\partial a}(\wh{X}_s) \right|^2\right)\d s.\notag
\end{align}
\end{thm}
\begin{remark}
{\rm  
Thanks to \eqref{eq:Identity1} and \eqref{eq:Formula*} below, 
\eqref{eq:CAFExpression} and \eqref{eq:CAFExpression*} possesses the following formal expressions:
\begin{align*}
\widehat{N}_t^{[|\eta|^2]}&=\wh{A}_t^{\left(\frac{\partial^2}{\partial t^2}+\text{\boldmath$\Delta$}\right)|\eta|^2}=
\text{\lq\lq}\int_0^t\left(\frac{\partial^2}{\partial t^2}+\text{\boldmath$\Delta$}\right)|\eta|^2(X_s)\d s
\text{\rq\rq},
\\
\widehat{N}_t^{[\eta_{\eps}]}&=\wh{A}_t^{\left(\frac{\partial^2}{\partial t^2}+\text{\boldmath$\Delta$}\right)\eta_{\eps}}=
\text{\lq\lq}\int_0^t\left(\frac{\partial^2}{\partial t^2}+\text{\boldmath$\Delta$}\right)\eta_{\eps}(X_s)\d s\text{\rq\rq}.
\end{align*}
$N_t^{[|\eta|^2]}$ may not be a PCAF (equivalently $|\eta|^2(\wh{X}_t)$ 
may not be a submartingale) unless $\kappa^-=0$. 
}
\end{remark}
\begin{proof}[\bf Theorem~\ref{thm:CAFExpression}]
Take $\phi\in D(\mathscr{E})$ and $\varphi\in C_c^{\infty}(\R)$. By way of expression \eqref{eq:OriginalExp1} for 
$\phi\in D(\mathscr{E})$ and $\eta\in D(\DD)\cap L^{\infty}(T^*\!M)$, we have 
\begin{align*}
-\wh{\mathscr{E}}^{\,2\wh\kappa}(|\eta|^2,\phi\otimes\varphi)&=-\int_{\R}\mathscr{E}^{2\kappa}(|\eta|^2,\phi)\varphi(t)\d t-\int_{\R}\left\langle \frac{\partial|\eta|^2}{\partial t},\phi\right\rangle_{L^2(T^*\!M)}\dot{\varphi}(t)\d t\\
&=\int_{\R}\varphi(t)\int_M\tilde{\phi}\,\d \text{\boldmath$\Delta$}^{2\kappa}|\eta|^2\d t+\int_{\R}\varphi(t)\int_M\phi 
\frac{\partial^2|\eta|^2}{\partial t^2}\d\m\d t\\
&=\int_{\R}\varphi(t)\int_M\tilde{\phi}\,\d \left(\frac{\partial^2}{\partial t^2}+\text{\boldmath$\Delta$}^{2\kappa}\right)|\eta|^2\d t\\
&=\int_{\R}\varphi\left\langle \left(\frac{\partial^2}{\partial t^2}+\text{\boldmath$\Delta$}^{2\kappa}\right)|\eta|^2,\phi \right\rangle_{L^2(T^*\!M)}\d t.
\end{align*}
By \eqref{eq:Identity1}, we have
\begin{align*}
-\frac12\wh{\mathscr{E}}^{\,2\wh\kappa}(|\eta|^2,\phi\otimes\varphi)
&=\int_{\R}\varphi\left\langle (\alpha|\eta|^2+|\overline{\nabla}\eta|^2)\m+{\bf Ric}(\eta^{\sharp},\eta^{\sharp})-\wt{|\eta|^2}\kappa,\phi \right\rangle\d t\\
&=\left\langle  (\alpha|\eta|^2+|\overline{\nabla}\eta|^2)\wh\m+{\bf Ric}^{\kappa}(\eta^{\sharp},\eta^{\sharp})\otimes m,\phi\otimes\varphi\right\rangle
\end{align*}
and 
\begin{align*}
-\frac12\wh{\mathscr{E}}(|\eta|^2,\phi\otimes\varphi)
&=\int_{\R}\varphi\left\langle (\alpha|\eta|^2+|\overline{\nabla}\eta|^2)\m+{\bf Ric}(\eta^{\sharp},\eta^{\sharp} \right\rangle\d t\\
&=\left\langle  (\alpha|\eta|^2+|\overline{\nabla}\eta|^2)\wh\m+{\bf Ric}(\eta^{\sharp},\eta^{\sharp})\otimes m,\phi\otimes\varphi\right\rangle.
\end{align*}
Applying \cite[Theorem~5.4.2]{FOT}, we obtain
\begin{align*}
\widehat{N}_t^{[|\eta|^2]}=2\left[\int_0^t(\alpha|\eta|^2+|\overline{\nabla}\eta|^2)(\wh{X}_s)\d s+\wh{A}_t^{\,{\bf Ric}(\eta^{\sharp}\eta^{\sharp})\otimes m} \right]=\wh{A}_t^{
\left(\frac{\partial^2}{\partial t^2}+\text{\boldmath$\Delta$}^{2\kappa}\right)|\eta|^2}.
\end{align*}
Finally, \eqref{eq:CAFExpression*} can be obtained by way of the usual It\^o's formula for semimartingale \eqref{eq:FukushimaDecomp} with 
\eqref{eq:CAFExpression}.  
\end{proof} 

\begin{lem}\label{lem:QuadraticVariation}
Suppose $\alpha>C_{\kappa}$. 
Let $\eta=Q_t^{(\alpha),{\rm HK}}\theta$ with $\theta\in D(\DD)\cap L^{\infty}(T^*\!M)$  
and set $\eta_{\eps}:=(|\eta|^2+\eps^2)^{\frac{p}{2}}-\eps^p$ with $\ep>0$ and $p\in]0,+\infty[$. Then 
$\E_{(x,a)}[\widehat{M}_{\tau}^{[\eta_{\eps}]}]=0$ for $\wh\m$-a.e.~$(x,a)\in M\times]0,+\infty[$. 
\end{lem}
\begin{proof}[{\bf Proof}] 
It suffices to prove that 
\begin{align*}
\E_{(x,a)}[\langle \widehat{M}^{[\eta_{\eps}]}\rangle_{\tau}]=\E_{(x,a)}\left[\int_0^{\tau}\left(
|\nabla\eta_{\eps}|^2(\wh{X}_s)+\left|\frac{\partial \eta_{\eps}}{\partial t} \right|^2(\wh{X}_s)\right)\d s
\right]<\infty
\end{align*}
for $\wh\m$-a.e.~$(x,a)$. This clearly holds 
in view of the transience of $\wh{\bf X}_G$ for $G:=M\times]0,+\infty[$ and 
$|\nabla\eta_{\eps}|^2+\left|\frac{\partial \eta_{\eps}}{\partial t} \right|^2\in L^1(G;\wh\m)$ due to $\eta_{\eps}\in D(\wh{\mathscr{E}})$. 
\end{proof} 
\begin{lem}\label{lem:TimeDepLaplacian}
Let $\eta=\eta(x,t)=Q_t^{(\alpha),{\rm HK}}\theta$ with $\theta\in D(\DD)$. 
Suppose $p\in[1,2]$. Then 
\begin{align*}
\left(|\eta|^2+\eps^2 \right)^{1-\frac{p}{2}}\left( \frac{\partial^2}{\partial t^2}+\Delta_{\ll}\right)\left(|\eta|^2+\eps^2\right)^{\frac{p}{2}}\geq 
p(\alpha+k)|\eta|^2+p(p-1)|\overline{\nabla}\eta|^2\quad\m\text{-a.e.}
\end{align*}
Here $k$ is the Radon-Nikodym density of $\kappa_{\ll}$ with respect to $\m${\rm:} $k:=\frac{\d\kappa_{\ll}}{\d\m}$. 
\end{lem}
\begin{proof}[{\bf Proof}] 
By a simple calculation, we have 
\begin{align}
\frac{\partial^2}{\partial t^2}\left(|\eta|^2+\eps^2 \right)^{\frac{p}{2}}=p(p-2)\left(|\eta|^2+\eps^2 \right)^{\frac{p}{2}-2}
\langle \dot{\eta},\eta\rangle^2+\frac{\;p\;}{2}\left(|\eta|^2+\eps^2\right)^{\frac{p}{2}-1}
\frac{\partial^2}{\partial t^2}|\eta|^2.\label{eq:Formula1}
\end{align}
On the other hand, from Corollary~\ref{cor:Composition} with $|\eta|^2\in D_f(\text{\boldmath$\Delta$})_+$ (Lemma~\ref{lem:LaplacianEta}), we have 
\begin{align}
{\text{\boldmath$\Delta$}}(|\eta|^2+\eps^2)^{\frac{p}{2}}=\frac{p}{2}\left(\frac{p}{2}-1 \right)\left(|\eta|^2+\eps^2\right)^{\frac{p}{2}-2}|\nabla |\eta|^2|^2\m+\frac{p}{2}(\wt{|\eta|^2}+\eps^2)^{\frac{p}{2}-1}{\text{\boldmath$\Delta$}}|\eta|^2.\label{eq:Formula2}
\end{align}
Adding \eqref{eq:Formula1} to \eqref{eq:Formula2} and dividing the both sides by $(|\eta|^2+\eps^2)^{\frac{p}{2}-1}$, we obtain 
\begin{align}
\left(|\eta|^2+\eps^2 \right)^{1-\frac{p}{2}}&
\left( \frac{\partial^2}{\partial t^2}+
\text{\boldmath$\Delta$}\right)\left(|\eta|^2+\eps^2\right)^{\frac{p}{2}}\notag\\
&\hspace{-0.5cm}=\frac{\;p\;}{2}\left( \frac{\partial^2}{\partial t^2}+\text{\boldmath$\Delta$}\right)|\eta|^2
+\frac{p}{2}\left(\frac{p}{2}-1 \right)
\left(|\eta|^2+\eps^2 \right)^{-1}
\left(4\langle \dot{\eta},\eta\rangle^2+\left|\nabla|\eta|^2\right|^2 \right)\m\notag\\
&\hspace{-0.5cm}=\frac{\;p\;}{2}\left( \frac{\partial^2}{\partial t^2}+\text{\boldmath$\Delta$}\right)|\eta|^2
+\frac{p}{2}\left(\frac{p}{2}-1 \right)
\left(|\eta|^2+\eps^2 \right)^{-1}
\left(\left|\frac{\partial}{\partial t}|\eta|^2\right|^2+\left|\nabla|\eta|^2\right|^2 \right)\m\notag\\
&\hspace{-0.5cm}=\frac{\;p\;}{2}\left( \frac{\partial^2}{\partial t^2}+\text{\boldmath$\Delta$}\right)|\eta|^2
+\frac{p}{2}\left(\frac{p}{2}-1 \right)
\left(|\eta|^2+\eps^2 \right)^{-1}
\left|\overline{\nabla}|\eta|^2\right|^2\m\notag\\
&\hspace{-0.5cm}=p\left[(\alpha|\eta|^2+|\overline{\nabla}\eta|^2)\m+{\bf Ric}(\eta^{\sharp},\eta^{\sharp})\right]+
\frac{p}{4}(p-2)(|\eta|^2+\eps^2)^{-1}\left|\overline{\nabla}|\eta|^2\right|^2\m.\label{eq:Formula*}
\end{align} 
Hence
\begin{align*}
\left(|\eta|^2+\eps^2 \right)^{1-\frac{p}{2}}&
\left( \frac{\partial^2}{\partial t^2}+
\Delta_{\ll}\right)\left(|\eta|^2+\eps^2\right)^{\frac{p}{2}}\\
&\hspace{-0.5cm}=\frac{\;p\;}{2}\left( \frac{\partial^2}{\partial t^2}+\Delta_{\ll}\right)|\eta|^2
+\frac{p}{2}\left(\frac{p}{2}-1 \right)
\left(|\eta|^2+\eps^2 \right)^{-1}
\left|\overline{\nabla}|\eta|^2\right|^2\quad\m\text{-a.e.}
\end{align*}
By using Kato's inequality (Lemma~\ref{lem:KatoIneq}), 
\begin{align}
\left|\nabla|\eta|^2 \right|^2&=\left|2|\eta|\nabla|\eta| \right|^2
=4|\eta|^2\cdot\left|\nabla|\eta| \right|^2\leq 4|\eta|^2\cdot\left|\nabla\eta^{\sharp}\right|_{\rm HS}^2.\label{eq:KatoModified}
\end{align}
Then 
\begin{align*}
\left|\overline{\nabla}|\eta|^2 \right|^2&=\left|\frac{\partial^2}{\partial t^2}|\eta|^2 \right|^2+\left|\nabla|\eta|^2 \right|^2\\
&\hspace{-0.3cm}\stackrel{\eqref{eq:KatoModified}}{\leq} 4\langle \eta,\dot{\eta}\rangle+4|\eta|^2\left|\nabla\eta^{\sharp}\right|_{\rm HS}^2\\
&\leq 4|\eta|^2\left(|\dot\eta|^2+ \left|\nabla\eta^{\sharp}\right|_{\rm HS}^2\right)
=4|\eta|^2\left|\overline{\nabla}\eta \right|^2\\
&\leq 4(|\eta|^2+\eps^2)\left|\overline{\nabla}\eta \right|^2.
\end{align*}
Noting $p-2\leq0$, we obtain 
\begin{align*}
\left(|\eta|^2+\eps^2 \right)^{1-\frac{p}{2}}&
\left( \frac{\partial^2}{\partial t^2}+
\Delta_{\ll}\right)\left(|\eta|^2+\eps^2\right)^{\frac{p}{2}}\\
&\hspace{-0.5cm}\geq \frac{\;p\;}{2}\left( \frac{\partial^2}{\partial t^2}+\Delta_{\ll}\right)|\eta|^2
+p\left(p-2\right)
\left|\overline{\nabla}\eta\right|^2\\
&\geq p\left(\alpha|\eta|^2+|\overline{\nabla}\eta|^2+{\mathfrak{Ric}}(\eta^{\sharp},\eta^{\sharp}) \right)+p\left(p-2\right)
\left|\overline{\nabla}\eta\right|^2\\
&=p(\alpha|\eta|^2+{\mathfrak{Ric}}(\eta^{\sharp},\eta^{\sharp})) +p(p-1)\left|\overline{\nabla}\eta\right|^2\\
&\hspace{-0.2cm}\stackrel{\eqref{eq:RicciTensor*}}{\geq} p(\alpha+k)|\eta|^2+p(p-1)\left|\overline{\nabla}\eta\right|^2
\quad\m\text{-a.e.}
\end{align*}
\end{proof} 
\begin{cor}\label{cor:TimeDepLaplacian}
Let $\eta=\eta(x,t)=Q_t^{(\alpha),{\rm HK}}\theta$ with $\theta\in D(\DD)$. 
Suppose $p\in[1,2]$. Then 
\begin{align*}
\left|\overline{\nabla}\eta(x,t)\right|^2&\leq\frac{1}{p(p-1)}\sup_{s>0}|\eta(x,s)|^{2-p}\varliminf_{\eps\to0}
\left( \frac{\partial^2}{\partial t^2}+
\Delta_{\ll}\right)\left(|\eta|^2+\eps^2\right)^{\frac{p}{2}}(x,t)\\
&\hspace{2cm}+\frac{1}{p-1}\sup_{s>0}|\eta(x,s)|^{2-p}|\eta(x,t)|^p(\alpha+k)^-(x)\quad\m\text{-a.e.}~x\in M.
\end{align*}
\end{cor}

Denote by $\mathscr{P}(M)$, the family of all Borel probability measures on $M$, and  
by $\mathscr{B}^*(M)$, the family of all universally measurable sets, that is, $\mathscr{B}^*(M):=\bigcap_{\nu\in\mathscr{P}(M)}\overline{\mathscr{B}(M)}^{\nu}$, where $\overline{\mathscr{B}(M)}^{\nu}$ is the $\nu$-completion of $\mathscr{B}(M)$ for $\nu\in\mathscr{P}(M)$. $\mathscr{B}^*(M)$ also denotes the family of 
universally measurable real valued functions. Moreover, $\mathscr{B}_b^*(M)$ (resp.~$\mathscr{B}_+^*(M)$) denotes the family of bounded (resp.~non-negative) universally measurable functions.  
For $f\in \mathscr{B}_b^*(M)$, or $f\in \mathscr{B}_+^*(M)$, we set
\begin{align}
q_t^{(\alpha),\kappa}f(x):=\E_x\left[\int_0^{\infty}e^{-\alpha s-A_s^{\kappa}}f(X_s)\lambda_t(\d s) \right]
\end{align}
and write $q_t^{(\alpha)}f(x)$ instead of $q_t^{(\alpha),0}f(x)$. Then $q_t^{(\alpha),\kappa}f\in \mathscr{B}_b^*(M)$ (resp.~$q_t^{(\alpha),\kappa}f\in \mathscr{B}_+^*(M)$) under $f\in \mathscr{B}_b^*(M)$ (resp.~$f\in \mathscr{B}_+^*(M)$).  
It is easy to see $q_t^{(\alpha),\kappa}1(x)\leq Ce^{-\sqrt{\alpha-C_{\kappa}}t}$ ($\alpha\geq C_{\kappa}$) and 
$q_t^{(\alpha)}1(x)\leq e^{-\sqrt{\alpha}t}$ ($\alpha\geq0$). 
Take $f\in L^2(X;\m)\cap \mathscr{B}_b^*(M)$ or $f\in L^2(M;\m)\cap \mathscr{B}_+^*(M)$.  
According to the equation 
\begin{align}
(q_t^{(\alpha),\kappa}f,g)_{\m}&=\int_0^{\infty}e^{-\alpha s}(p_s^{\kappa}f,g)_{\m}\lambda_t(\d s)\notag\\
&=\int_0^{\infty}e^{-\alpha s}(P_s^{\kappa}f,g)_{\m}\lambda_t(\d s)\label{eq:Version}\\
&=(Q_t^{(\alpha),\kappa}f,g)_{\m}\quad\text{ for any }\quad g\in L^2(M;\m),\notag
\end{align}
$q_t^{(\alpha),\kappa}f$ is an $\m$-version of $Q_t^{(\alpha),\kappa}f$.

\bigskip

To prove Theorem~\ref{thm:LittlewoodPaley1Form} below, we need the following lemma on the ${\mathscr{E}}$-quasi-continuity of $x\mapsto q_{t}^{\alpha,\kappa}(x)$: 
\begin{lem}\label{lem:QuasiContProduct}
For $u\in L^2(M;\m)\cap \mathscr{B}^*(M)$, $t>0$ and $\alpha>C_{\kappa}$, we set 
\begin{align}
Q_{t}^{(\alpha),\kappa}u(x):&=\int_0^{\infty}e^{-\alpha s}P_s^{\kappa}u(x)\lambda_t(\d s),\label{eq:SpaceTimeFuncFunc}\\
q_{t}^{(\alpha),\kappa}u(x):&=\int_0^{\infty}e^{-\alpha s}p_s^{\kappa}u(x)\lambda_t(\d s).\label{eq:SpaceTimeFunc}
\end{align}
Then $Q_t^{(\alpha),\kappa}u\in D(\mathscr{E})$ and 
$q_t^{(\alpha),\kappa}u$ is an 
${\mathscr{E}}$-quasi-continuous $\m$-version of $Q_t^{(\alpha),\kappa}u$. 
\end{lem}
\begin{proof}[{\bf Proof}]
The first assertion is clear from the following estimate: for $\alpha_0:=\alpha\|U_{\alpha}\kappa^-\|_{\infty}$ 
in \eqref{eq:Equivalence*}, 
\begin{align}
\|Q_t^{(\alpha),\kappa}u\|_{L^2(M;\m)}&\leq\int_0^{\infty}e^{-\alpha s}\|P_s^{\kappa}u\|_{L^2(M;\m)}\lambda_t(\d s)\notag \\
&\leq C(\kappa)e^{-\sqrt{\alpha-C_{\kappa}}t}\|u\|_{L^2(M;\m)}<\infty\label{eq;Domain0}
\end{align}
and
\begin{align}
\mathscr{E}_{\alpha_0}^{\kappa}(Q_t^{(\alpha),\kappa}u,Q_t^{(\alpha),\kappa}u)^{\frac12}&\leq 
\int_0^{\infty}e^{-\alpha s}\mathscr{E}_{\alpha_0}^{\kappa}(P_s^{\kappa}u, P_s^{\kappa}u)^{\frac12}\lambda_t(\d s)\notag\\
&\leq \left(\int_0^{\infty}e^{-\alpha s}\frac{1}{\sqrt{2s}}\lambda_t(\d s)\right)\|u\|_{L^2(M;\m)}<\infty.\label{eq;Domain1}
\end{align}
Fix $t>0$. Suppose $u\in L^2(M;\m)\cap C_b(M)$.  
Note that $p_s^{\kappa}u$ is an $\mathscr{E}$-quasi continuous $\m$-version of 
$P_s^{\kappa}u$ by \cite[Theorem~7.3]{Kw:stochII}. 
Then it is easy to see that for the following $D(\mathscr{E})$-valued step function $s\mapsto 
p_s^{n,\kappa}u$, $p_s^{n,\kappa}u$  
is $\mathscr{E}$-quasi continuous and $\{p_s^{n,\kappa}u\}_{n\in\mathbb{N}}$ $\mathscr{E}_1$-converges to $P_s^{\kappa}u$ for each $s>0$: 
\begin{align*}
p_s^{n,\kappa}u(x):=\sum_{k=0}^{n2^n-1}p_{\frac{k+1}{2^n}}^{\kappa}u(x)\1_{[\frac{k}{2^n},\frac{k+1}{2^n}[}(s)+p_n^{\kappa}u(x)\1_{[n,+\infty[}(s). 
\end{align*}
We can see 
\begin{align}
\|p_s^{n,\kappa}u\|_{L^2(M;\m)}&=\sum_{k=0}^{n2^n-1}\|p_{\frac{k+1}{2^n}}^{\kappa}u\|_{L^2(M;\m)}\1_{[\frac{k}{2^n},\frac{k+1}{2^n}[}(s)+\|p_n^{\kappa}u\|_{L^2(M;\m)}\1_{[n,+\infty[}(s)\notag\\
&=\sum_{k=0}^{n2^n-1}\|P_{\frac{k+1}{2^n}}^{\kappa}u\|_{L^2(M;\m)}\1_{[\frac{k}{2^n},\frac{k+1}{2^n}[}(s)+\|P_n^{\kappa}u\|_{L^2(M;\m)}\1_{[n,+\infty[}(s)\notag\\
&\leq C(\kappa)e^{C_{\kappa}\left(s+\frac{1}{2^n}\right)}\|u\|_{L^2(M;\m)}
\leq C(\kappa)e^{C_{\kappa}\left(s+1\right)}\|u\|_{L^2(M;\m)}.
\label{eq:Contraction}
\end{align}
Then, 
\begin{align*}
q_t^{(\alpha),n,\kappa}u(x):&=\int_0^{\infty}e^{-\alpha s}p_s^{n,\kappa}u(x)\lambda_t(\d s)\\&=
\sum_{k=0}^{n2^n-1}p_{\frac{k+1}{2^n}}^{\kappa}u(x)\int_{\frac{k}{2^n}}^{\frac{k+1}{2^n}}e^{-\alpha s}\lambda_t(\d s)+p_n^{\kappa}u(x)\int_n^{\infty}e^{-\alpha s}\lambda_t(\d s)
\end{align*}
satisfies $q_t^{(\alpha),n,\kappa}\in D(\mathscr{E})$ and it 
$\mathscr{E}_1$-converges to 
$Q_t^{(\alpha),\kappa}u$. Indeed, for $\alpha_0:=\alpha\|U_{\alpha}\kappa^-\|_{\infty}$ and $C>0$ 
 specified in \eqref{eq:Equivalence*},
\begin{align}
\|q_t^{(\alpha),n,\kappa}u-Q_t^{(\alpha),\kappa}u\|_{L^2(M;\m)}&\leq 
\int_0^{\infty}e^{-\alpha s}\|p_s^{n,\kappa}u-P_s^{\kappa}u\|_{L^2(M;\m)}\lambda_t(\d s)\notag\\
&\hspace{5cm}\to 0\quad\text{ as }\quad n\to\infty\label{eq:Dominated0}
\end{align}
and
\begin{align}
\mathscr{E}_{\alpha_0}^{\kappa}&(q_t^{(\alpha),n,\kappa}u-Q_t^{(\alpha),\kappa}u,q_t^{(\alpha),n,\kappa}u-Q_t^{(\alpha),\kappa}u)^{\frac12}\notag\\&\leq 
\int_0^{\infty}e^{-\alpha s}\mathscr{E}_{\alpha_0}^{\kappa}(p_s^{n,\kappa}u-P_s^{\kappa}u, p_s^{n,\kappa}u-P_s^{\kappa}u)^{\frac12}\lambda_t(\d s)\notag\\
&\leq \sqrt{C}\int_0^{\infty}e^{-\alpha s}\mathscr{E}_1(p_s^{n,\kappa}u-P_s^{\kappa}u, p_s^{n,\kappa}u-P_s^{\kappa}u)^{\frac12}\lambda_t(\d s)\notag\\
&\hspace{5cm}\to 0\quad\text{ as }\quad n\to\infty,\label{eq:Dominated1}
\end{align}
where we use the Lebesgue's dominated convergence theorem 
under the following observation: by \eqref{eq:Contraction}
\begin{align*}
\|p_s^{n,\kappa}u-P_s^{\kappa}u\|_{L^2(M;\m)}&\leq 2C(\kappa)e^{C_{\kappa}(s+1)},
\end{align*}
and 
for any $s>0$ and $n\in\mathbb{N}$, $s\geq n$ or $s\in[\frac{k}{2^n},\frac{k+1}{2^n}[$ for some $k\in\{0,\cdots, n2^n-1\}$. 
If $s\geq n$, then by \eqref{eq:Equivalence*}
\begin{align*}
\mathscr{E}(p_s^{n,\kappa}u-P_s^{\kappa}u, p_s^{n,\kappa}u-P_s^{\kappa}u)^{\frac12}&\leq 
(1-\|U_{\alpha}\kappa^-\|_{\infty})^{-\frac12}\mathscr{E}_{\alpha_0}^{\kappa}(p_n^{\kappa}u-P_s^{\kappa}u, p_n^{n,\kappa}u-P_s^{\kappa}u)^{\frac12}\\
&\leq (1-\|U_{\alpha}\kappa^-\|_{\infty})^{-\frac12}\left(\frac{1}{\sqrt{2n}}+\frac{1}{\sqrt{2s}}
 \right)\|u\|_{L^2(M;\m)}\\
&\leq (1-\|U_{\alpha}\kappa^-\|_{\infty})^{-\frac12}\frac{\sqrt{2}}{\sqrt{s}}\|u\|_{L^2(M;\m)}.
\end{align*}
If $s\in[\frac{k}{2^n},\frac{k+1}{2^n}[$ for some $k\in\{0,\cdots, n2^n-1\}$, then by \eqref{eq:Equivalence*} 
\begin{align*}
\mathscr{E}(p_s^{n,\kappa}u-P_s^{\kappa}u, p_s^{n,\kappa}u-P_s^{\kappa}u)^{\frac12}&\leq 
(1-\|U_{\alpha}\kappa^-\|_{\infty})^{-\frac12}
\mathscr{E}_{\alpha_0}^{\kappa}(p_{\frac{k+1}{2^n}}^{\kappa}u-P_s^{\kappa}u, p_{\frac{k+1}{2^n}}^{n,\kappa}u-P_s^{\kappa}u)^{\frac12}\\
&\leq (1-\|U_{\alpha}\kappa^-\|_{\infty})^{-\frac12}\left(\frac{1}{\sqrt{2(k+1)/2^n}}+\frac{1}{\sqrt{2s}}
 \right)\|u\|_{L^2(M;\m)}\\
&\leq (1-\|U_{\alpha}\kappa^-\|_{\infty})^{-\frac12}\frac{\sqrt{2}}{\sqrt{s}}\|u\|_{L^2(M;\m)}.
\end{align*}
Thus we can justify \eqref{eq:Dominated0} and \eqref{eq:Dominated1}. 
Since $u\in C_b(M)$ and $2\kappa^-\in S_{E\!K}({\bf X})$, 
$s\mapsto p_s^{\kappa}u(x)$ is continuous on $[0,+\infty[$, consequently, $p_s^{n,\kappa}u(x)$ converges to $p_s^{\kappa}u(x)$ as $n\to\infty$ for each $x\in M$. This means that 
$q_t^{(\alpha),n,\kappa}u(x)=\int_0^{\infty}e^{-\alpha s}p_s^{n,\kappa}(x)\lambda_t(\d s)$ converges to $q_t^{(\alpha),\kappa}u(x)=
\int_0^{\infty}e^{-\alpha s}p_s^{\kappa}u(x)(x)\lambda_t(\d s)$ as $n\to\infty$ for each $x\in M$. 
Therefore, $q_t^{(\alpha),\kappa}u$ is an $\mathscr{E}$-quasi continuous $\m$-version of $Q_t^{(\alpha),\kappa}u$ by \cite[Theorem~2.1.4(ii)]{FOT}. 

Next we prove the assertion for general $u\in L^2(M;\m)\cap \mathscr{B}(M)$. 
Let us put $H:=\{u\in L^2(M;\m)\cap \mathscr{B}_+(M) \mid q_t^{(\alpha),\kappa}u\text{ is an }\mathscr{E}\text{-quasi continuous }\m\text{-version of }Q_t^{(\alpha),\kappa}u\}$. We have just proved 
$H\supset L^2(M;\m)\cap C_b(M)_+$, and $H$ satisfies (H.3) of \cite[Lemma~4.2.3]{FOT}. 
Indeed, for any open set $G$, we choose an increasing sequence of $\m$-finite open sets $\{G_n\}$. 
Since $(M,\tau)$ can be metrizable in view of the quasi-regularity of $(\mathscr{E},D(\mathscr{E})$ 
(in fact, $M$ can be replaced with a countable union of compact $\mathscr{E}$-nest), $u_n:=n{\sf d}(x,G_n^c)\land1 \uparrow \1_G$. 
(H.1) of \cite[Lemma~4.2.3]{FOT} is trivially satisfied. To check (H.2) of \cite[Lemma~4.2.3]{FOT}, let $u_n\in H$ increase to $u\in L^2(M;\m)$. Then $Q_t^{(\alpha),\kappa}u_n$ $\mathscr{E}_1$-converges to 
$Q_t^{(\alpha),\kappa}u$. Indeed, for $\alpha_0:=C_{\kappa}$, we see
\begin{align*}
\mathscr{E}_{\alpha_0}^{\kappa}(Q_t^{(\alpha),\kappa}u_n-Q_t^{(\alpha),\kappa}u,&Q_t^{(\alpha),\kappa}u_n-Q_t^{(\alpha),\kappa}u)^{\frac12}\\&\leq \int_0^{\infty}e^{-\alpha s}\mathscr{E}_{\alpha_0}^{\kappa}(P_s^{\kappa}(u_n-u), P_s^{\kappa}(u_n-u))^{\frac12}\lambda_t(\d s)\\
&\leq
\int_0^{\infty}e^{-\alpha s}\frac{1}{\sqrt{2s}}\|u_n-u\|_{L^2(M;\m)}\lambda_t(\d s)\\
&=\|u_n-u\|_{L^2(M;\m)}\int_0^{\infty}e^{-\alpha s}\frac{\lambda_t(\d s)}{\sqrt{2s}}\\
&\hspace{5cm}\to 0\quad\text{ as }\quad n\to\infty
\end{align*}
and 
\begin{align*}
\|Q_t^{(\alpha),\kappa}u_n-Q_t^{(\alpha),\kappa}u\|_{L^2(M;\m)}&\leq
\int_0^{\infty}e^{-\alpha s}\|P_s^{\kappa}(u_n-u)\|_{L^2(M;\m)}\lambda_t(\d s)\\
&\leq C(\kappa)e^{-\sqrt{\alpha-C_{\kappa}}t}\|u_n-u\|_{L^2(M;\m)}\to 0\quad\text{ as }\quad n\to\infty.
\end{align*}
Since $q_t^{(\alpha),\kappa}u_n$ is an $\mathscr{E}$-quasi continuous $\m$-version of $Q_t^{(\alpha),\kappa}u_n$ and converges to $q_t^{(\alpha),\kappa}u$ pointwise, we get $u\in H$ by \cite[Theorem~2.1.4(ii)]{FOT}. 
Then \cite[Lemma~4.2.3]{FOT} implies that the assertion holds for $u\in L^2(M;\m)\cap\mathscr{B}(M)$. 
The extension to the universally measurable $u$ is clear. 
\end{proof}

\begin{thm}\label{thm:LittlewoodPaley1Form}
Suppose $p\in]1,2]$, $\kappa^-\in S_K({\bf X})$ and $\alpha>C_{\kappa}$. Then 
the $G$-function  
$G_{\theta}$ can be extended to be in $L^p(T^*\!M)$ 
for $\theta\in L^p(T^*\!M)$ and the following inequality holds for 
$\theta\in L^p(T^*\!M)$
\begin{align}
\|G_{\theta}\|_{L^p(M;\m)}&\lesssim\|\theta\|_{L^p(T^*\!M)}.\label{eq:LittlewoodPaleyStein1}
\end{align}
\end{thm}
\begin{proof}[{\bf Proof}] 
First we assume 
$\theta\in D(\DD)
\cap L^p(T^*\!M)$. By setting $\eta(x,t):=Q_t^{(\alpha),{\rm HK}}\theta(x)$, we see from Corollary~\ref{cor:TimeDepLaplacian}
\begin{align*}
G_{\theta}(x)^2&=\int_0^{\infty}t\left|\overline{\nabla}\eta(x,t)\right|^2\d t
\\
&=
\frac{1}{p(p-1)}\sup_{s>0}|\eta(x,s)|^{2-p}\int_0^{\infty}t \varliminf_{\eps\to0}
\left( \frac{\partial^2}{\partial t^2}+
\Delta_{\ll}\right)\left(|\eta|^2+\eps^2\right)^{\frac{p}{2}}(x,t)\d t,\\
&\hspace{2cm}+\frac{1}{p-1}\sup_{s>0}|\eta(x,s)|^{2-p}(\alpha+k)^-(x)\int_0^{\infty}t |\eta(x,t)|^p\d t.
\end{align*}
Using the H\"older inequality with exponent $\frac{2}{2-p}$ and $\frac{2}{p}$, we have 
\begin{align*}
\|G_{\theta}\|_{L^p(M;\m)}^p&\leq \left(\frac{1}{p(p-1)}\right)^{\frac{p}{2}}\int_M\sup_{s>0}|\eta(x,s)|^{\frac{p(2-p)}{2}}\\
&\hspace{1cm}\times
\left(\int_0^{\infty}t\left[\varliminf_{\eps\to0}\left(\frac{\partial^2}{\partial t^2}+\Delta_{\ll} \right)(|\eta|^2+\eps^2)^{\frac{p}{2}}+p|\eta|^p(\alpha+k)^-\right]\d t \right)^{\frac{p}{2}}\d\m\\
&\leq \left(\frac{1}{p(p-1)}\right)^{\frac{p}{2}}\|\sup_{s>0}|\eta(\cdot,s)|\|_{L^p(M;\m)}^{\frac{p(2-p)}{2}}\\
&\hspace{1cm}\times
\left(\int_M\int_0^{\infty}t\left[ \varliminf_{\eps\to0}
\left(\frac{\partial^2}{\partial t^2}+\Delta_{\ll} \right)(|\eta|^2+\eps^2)^{\frac{p}{2}}+p|\eta|^p(\alpha+k)^-\right]\d t\d\m
 \right)^{\frac{p}{2}}.
\end{align*}
From Proposition~\ref{prop:MaxErgo}, we have
\begin{align*}
\|G_{\theta}\|_{L^p(M;\m)}^p\leq C_p\|\theta\|_{L^p(T^*\!M)}^{\frac{p(2-p)}{2}}
\left(\int_M\int_0^{\infty}t\left[ \varliminf_{\eps\to0}
\left(\frac{\partial^2}{\partial t^2}+\Delta_{\ll} \right)(|\eta|^2+\eps^2)^{\frac{p}{2}}+p|\eta|^p(\alpha+k)^-\right]\d t\d\m
 \right)^{\frac{p}{2}}.
\end{align*}
Then, it suffices to show that 
\begin{align}
\int_M\int_0^{\infty}t\left[ \varliminf_{\eps\to0}
\left(\frac{\partial^2}{\partial t^2}+\Delta_{\ll} \right)(|\eta|^2+\eps^2)^{\frac{p}{2}}+p|\eta|^p(\alpha+k)^-\right]\d t\d\m
\lesssim\|\theta\|_{L^p(T^*\!M)}^p.\label{eq:FainalEst}
\end{align}
The left-hand side of \eqref{eq:FainalEst} is estimated as follows: 
\begin{align*}
\int_M\int_0^{\infty}&t\left[ \varliminf_{\eps\to0}
\left(\frac{\partial^2}{\partial t^2}+\Delta_{\ll} \right)(|\eta|^2+\eps^2)^{\frac{p}{2}}+p|\eta|^p(\alpha+k)^-\right]\d t\d\m
\\&\leq\varliminf_{\eps\to0}\lim_{a\to\infty}
\int_M\int_0^{\infty}(t\land a)\left[
\left(\frac{\partial^2}{\partial t^2}+\Delta_{\ll} \right)(|\eta|^2+\eps^2)^{\frac{p}{2}}+p|\eta|^p(\alpha+k)^-
 \right]\d t\d\m\\
 &=\varliminf_{\eps\to0}\lim_{a\to\infty}\E_{\m\otimes\delta_a}\left[
 \int_0^{\tau}\left[\left(\frac{\partial^2}{\partial t^2}+\Delta_{\ll} \right)(|\eta|^2+\eps^2)^{\frac{p}{2}}(\wh{X}_s)+p|\eta|^p(\alpha+k)^-(\wh{X}_s) \right]\d s \right].
\end{align*}
Here we use Proposition~\ref{prop:ShigekawaProp6.1} (\cite[Proposition~6.1]{Shigekawa1}). 
On the other hand, the signed finite smooth measure
\begin{align*}
\left(\frac{\partial^2}{\partial t^2}+\Delta_{\ll} \right)(|\eta|^2+\eps^2)^{\frac{p}{2}}\m+p|\eta|^p(\alpha+k)^-\m
\end{align*}
is estimated above by 
\begin{align*}
\left(\frac{\partial^2}{\partial t^2}+\text{\boldmath$\Delta$}\right)(|\eta|^2+\eps^2)^{\frac{p}{2}}+
p\wt{|\eta|^p}\kappa^-.
\end{align*}
Indeed, 
\begin{align*}
\left(\frac{\partial^2}{\partial t^2}+\Delta_{\ll} \right)&(|\eta|^2+\eps^2)^{\frac{p}{2}}\m+p|\eta|^p(\alpha+k)^-\m\\
&\hspace{-1cm}=\left(\frac{\partial^2}{\partial t^2}+\text{\boldmath$\Delta$}\right)(|\eta|^2+\eps^2)^{\frac{p}{2}}-
\left(\frac{\partial^2}{\partial t^2}+\text{\boldmath$\Delta$}\right)_{\perp}(|\eta|^2+\eps^2)^{\frac{p}{2}}+p|\eta|^p(\alpha+k)^-\m\\
&\hspace{-1cm}=\left(\frac{\partial^2}{\partial t^2}+\text{\boldmath$\Delta$}\right)(|\eta|^2+\eps^2)^{\frac{p}{2}}-
p(\wt{|\eta|^2}+\eps^2)^{\frac{p}{2}-1}{\bf Ric}_{\perp}(\eta^{\sharp},\eta^{\sharp})+p|\eta|^p(\alpha+k)^-\m\\
&\hspace{-1cm}=\left(\frac{\partial^2}{\partial t^2}+\text{\boldmath$\Delta$}\right)(|\eta|^2+\eps^2)^{\frac{p}{2}}-
p(\wt{|\eta|^2}+\eps^2)^{\frac{p}{2}-1}{\bf Ric}_{\perp}^{\kappa}(\eta^{\sharp},\eta^{\sharp})\\
&\hspace{4cm}+p(\wt{|\eta|^2}+\eps^2)^{\frac{p}{2}-1}\wt{|\eta|^2}\kappa_{\perp}
+p|\eta|^p\kappa_{\ll}^-\\
&\hspace{-1cm}\leq\left(\frac{\partial^2}{\partial t^2}+\text{\boldmath$\Delta$}\right)(|\eta|^2+\eps^2)^{\frac{p}{2}}
+p(\wt{|\eta|^2}+\eps^2)^{\frac{p}{2}-1}\wt{|\eta|^2}\kappa^-_{\perp}+ p|\eta|^p\kappa_{\ll}^-\\
&\hspace{-1cm}\leq\left(\frac{\partial^2}{\partial t^2}+\text{\boldmath$\Delta$}\right)(|\eta|^2+\eps^2)^{\frac{p}{2}}
+p\wt{|\eta|^p}\kappa^-.
\end{align*}
Here we use $p\leq2$, 
${\bf Ric}_{\perp}^{\kappa}(\eta^{\sharp},\eta^{\sharp})\geq0$ and ${\bf Ric}_{\perp}(\eta^{\sharp},\eta^{\sharp})={\bf Ric}^{\kappa}_{\perp}(\eta^{\sharp},\eta^{\sharp})+\wt{|\eta|^2}\kappa_{\perp}$. Then the left-hand side of \eqref{eq:FainalEst} is estimated above by 
\begin{align*}
\varliminf_{\eps\to0}&\lim_{a\to\infty}\left(\E_{\m\otimes\delta_a}\left[\wh{A}_{\tau}^{\left(\frac{\partial^2}{\partial t^2}+\text{\boldmath$\Delta$}\right)\eta_{\eps}}\right]+p\,\E_{\m\otimes\delta_a}\left[\int_0^{\tau}\wt{|\eta|^p}(\wh{X}_s)\d 
\wh{A}_s^{\,\wh\kappa^-} \right]\right)\\
&=
\varliminf_{\eps\to0}\lim_{a\to\infty}\left(\E_{\m\otimes\delta_a}\left[\widehat{N}_{\tau}^{[\eta_{\eps}]}
\right]+p\,\E_{\m\otimes\delta_a}\left[\int_0^{\tau}\wt{|\eta|^p}(\wh{X}_s)\d 
\wh{A}_s^{\,\wh\kappa^-} \right]\right)\\
&=
\varliminf_{\eps\to0}\lim_{a\to\infty}\left(\E_{\m\otimes\delta_a}\left[\wt{{\eta}_{\eps}}(\wh{X}_{\tau})-\wt{{\eta}_{\eps}}(\wh{X}_0)
\right]+p\,\E_{\m\otimes\delta_a}\left[\int_0^{\tau}\wt{|\eta|^p}(\wh{X}_s)\d 
\wh{A}_s^{\,\wh\kappa^-} \right]\right)\quad (\text{Lemma~\ref{lem:QuadraticVariation}})\\
&\leq
\varliminf_{\eps\to0}\lim_{a\to\infty}\left(\E_{\m\otimes\delta_a}\left[(\wt{|\eta|^2}(\wh{X}_{\tau})+\eps^2)^{\frac{p}{2}}-\eps^p\right]+p\,\E_{\m\otimes\delta_a}\left[\int_0^{\tau}\wt{|\eta|^p}(\wh{X}_s)\d 
\wh{A}_s^{\,\wh\kappa^-} \right]\right)\quad (\wt{\eta_{\eps}}(\wh{X}_0)\geq0)\\
&\leq \lim_{a\to\infty}\left(\E_{\m\otimes\delta_a}
\left[\wt{|\eta|^p}(\wh{X}_{\tau})\right]+
p\,\E_{\m\otimes\delta_a}\left[\int_0^{\tau}\wt{|\eta|^p}(\wh{X}_s)\d 
\wh{A}_s^{\,\wh\kappa^-} \right]\right)\quad ((a^2+b^2)^{\frac{p}{2}}-b^p\leq a^p)\\
&\leq \|\theta\|_{L^p(T^*\!M)}^p+ \lim_{a\to\infty}
p\,\E_{\m\otimes\delta_a}\left[\int_0^{\tau}\wt{|\eta|^p}(\wh{X}_s)\d 
\wh{A}_s^{\,\wh\kappa^-}\right]\\
&\hspace{-0.2cm}\stackrel{\eqref{eq:ShigekawaProp3.10}}{=}\|\theta\|_{L^p(T^*\!M)}^p+ \lim_{a\to\infty}
p\int_M\int_0^{\infty}(a\land t)\wt{|\eta|^p}(x,t)\d t\kappa^-(\d x)\\
&=\|\theta\|_{L^p(T^*\!M)}^p+p\int_M\int_0^{\infty}t\wt{|\eta|^p}(x,t)\d t\kappa^-(\d x).
\end{align*}
Here we use $\E_{\m\otimes\delta_a}[\wt{|\eta|^2}(\wh{X}_{\tau})]=\E_{\m\otimes\delta_a}[\wt{|\eta|^2}({X}_{\tau},0)]=\int_M\wt{|\eta|^p}(x,0)\d\m=\int_M|\theta|^p\d\m$ in the last inequality, and 
implicitly use a discrete limit for $a\to\infty$ avoiding Lebesgue measure zero set appeared in Lemma~\ref{lem:QuadraticVariation}. 
We can not use Lenglart-Lepingle-Pratelli inequality, because $\eta_{\eps}(\wh{X}_t)$ may not be a submartingale.  
Since the Bochner integral $\eta(x,t)=\int_0^{\infty}e^{-\alpha s}P_s^{\rm HK}\theta(x)\lambda_t(\d s)$ is an 
$L^2(T^*\!M)$-limit of 
\begin{align*}
\sum_{k=0}^{n2^n-1}P_{\frac{k}{2^n}}^{\rm HK}\theta(x)\int_{\frac{k}{2^n}}^{\frac{k+1}{2^n}}e^{-\alpha s}\lambda_t(\d s)+P_n^{\rm HK}\theta(x)\int_n^{\infty}e^{-\alpha s}\lambda_t(\d s), 
\end{align*}
its pointwise norm converges to $|\eta(x,t)|$ in $\m$-measure with respect to $x\in M$ as $n\to\infty$. Thus we see
\begin{align*}
|\eta(x,t)|&=\left|\int_0^{\infty}e^{-\alpha s}P_s^{\rm HK}\theta(x)\lambda_t(\d s) \right|\\
&\leq \int_0^{\infty}e^{-\alpha s}|P_s^{\rm HK}\theta(x)|\lambda_t(\d s)\quad\m\text{-a.e.}\\
&\hspace{-0.2cm}\stackrel{\eqref{eq:HSU}}{\leq}
\int_0^{\infty}e^{-\alpha s}P_s^{\kappa}|\theta|(x)\lambda_t(\d s)=Q_t^{\alpha,\kappa}|\theta|(x).
\end{align*} 
From this, 
we have $\wt{|\eta|}(x,t)\leq q_t^{\alpha}|\theta|(x)$ $\mathscr{E}$-q.e.~$x\in M$ for a.e.~$t>0$ by Lemma~\ref{lem:QuasiContProduct},. 
Note here that for a.e.~$t>0$ the restriction $\wt{|\eta|}(\cdot, t)$ 
of $\wh{\mathscr{E}}$-quasi continuous function 
$\wt{|\eta|}$ is $\mathscr{E}$-quasi continuous by Lemma~\ref{lem:Capacity}. Moreover, we take $|\theta|$ as a Borel $\m$-version. 
Then the second term of the right hand side above is 
\begin{align*}
p\int_M\int_0^{\infty}t\,\wt{|\eta|^p}(x,t)\d t\,\kappa^-(\d x)&\leq p\int_M\int_0^{\infty}t\,
q_t^{\alpha}|\theta|(x)^p
\d t\,\kappa^-(\d x)\\
&\leq p\int_M\int_0^{\infty}t\int_0^{\infty}e^{-\alpha ps}p_s^{\kappa}|\theta|(x)^p\lambda_t(\d s)\d t\,\kappa^-(\d x)\\
&= p\int_M\int_0^{\infty}\left\{\int_0^{\infty}t
\left(\frac{t}{2\sqrt{\pi}}e^{-\frac{t^2}{4s}}s^{-\frac32} \right)
\d t \right\}\\
&\hspace{3cm}\times
\left(\E_x\left[e^{-\alpha s-A_s^{\kappa}}|\theta|(X_s) \right] \right)^p\d s\,\kappa^-(\d x)\\
&\hspace{-0.2cm}\stackrel{\eqref{eq:Expectation}}{\leq} p\int_M\int_0^{\infty}\E_x\left[e^{-p(\alpha s+A_s^{\kappa})}|\theta|^p(X_s) \right]\d s\,\kappa^-(\d x)\\
&\leq p\int_M\E_x\left[\int_0^{\infty}e^{-p\alpha s+pA_s^{\kappa^-}}|\theta|^p(X_s)\d s\right]\kappa^-(\d x)\\
&=p\int_M|\theta|^p(x)\E_x\left[\int_0^{\infty}e^{-p\alpha s+pA_s^{\kappa^-}}\d A_s^{\kappa^-} \right]\m(\d x),
\end{align*}
where 
we use the generalization of Revuz correspondence (see \cite[(3.11)]{AM:AF}) and 
\begin{align}
\int_0^{\infty}t
\left(\frac{t}{2\sqrt{\pi}}e^{-\frac{t^2}{4s}}s^{-\frac32} \right)
\d t =1.\label{eq:Expectation}
\end{align}
By using the probability law  $\P_x^{(p\alpha)}$ for the killed process ${\bf X}^{(p\alpha)}$ characterized by $\E_x^{(p\alpha)}[f(X_t)]=\E_x[e^{-p\alpha t}f(X_t)]$, $f\in\mathscr{B}_+(M)$, we have
\begin{align*}
p\,\E_x\left[\int_0^{\infty}e^{-p\alpha s+pA_s^{\kappa^-}}\d A_s^{\kappa^-} \right]&=\E_x\left[\int_0^{\infty}e^{-p\alpha s}\d e^{pA_s^{\kappa^-}}\right]
=\E_x^{(p\alpha)}\left[e^{pA_{\infty}^{\kappa^-}} \right]\\
&\leq \frac{1}{1-p\left\|\E_{\cdot}^{(p\alpha)}\left[A_{\infty}^{\kappa^-} \right]\right\|_{L^{\infty}(M;\m)}}\\
&\leq \frac{1}{1-p\left\|\E_{\cdot}\left[\int_0^{\infty}e^{-p\alpha s}\d A_s^{\kappa^-} \right]\right\|_{L^{\infty}(M;\m)}}<\infty,\quad \m\text{-a.e.~}x\in M.
\end{align*} 
Then, we can conclude that the left-hand side of \eqref{eq:FainalEst} is estimated by 
\begin{align*}
\left(1+\frac{1}{1-p\left\|\E_{\cdot}\left[\int_0^{\infty}e^{-p\alpha s}\d A_s^{\kappa^-} \right]\right\|_{L^{\infty}(M;\m)}} \right)\|\theta\|_{L^p(T^*\!M)}^p.
\end{align*}
Note here that under $\alpha>C_{\kappa}$, we have $p\alpha>C_{p\kappa}$, hence from Lemma~\ref{lem:ExtendeddKato1}
\begin{align*}
\left\|\E_{\cdot}\left[\int_0^{\infty}e^{-p\alpha s}\d A_s^{\kappa^-} \right]\right\|_{L^{\infty}(M;\m)}<1/p.
\end{align*}
Any $\theta\in L^p(T^*\!M)\cap L^2(T^*\!M)$ can be 
 approximated by a sequence $\{P_{1/n}^{\rm HK}\theta\}$ 
 of $D(\DD)\cap L^p(T^*\!M)$ in $L^p(T^*\!M)$-norm and in $L^2(T^*\!M)$-norm by 
 Corollary~\ref{cor:HessShcraderUhlenbrock}.
Then one can conclude the statement of 
Theorem~\ref{thm:LittlewoodPaley1Form} holds
for general $\theta \in  L^p(T^*\!M)\cap L^2(T^*\!M)$ in view of the triangle inequality
$|G_{\theta_1}
-G_{\theta_2}
| \leq G_{\theta_1-\theta_2}
$. 
Moreover, since any $\theta\in   L^p(T^*\!M)$ can be approximated by a sequence of $L^p(T^*\!M)\cap L^2(T^*\!M)$ in $L^p(T^*\!M)$, we can conclude the assertion. 
\end{proof} 

Now we define the $H$-function by 
\begin{align*}
H_{\theta}^{\rightarrow}(x):&=\left\{\int_0^{\infty}t Q_t^{(\alpha),\kappa}\left(g_{\theta}^{\rightarrow}(\cdot,t)^2 \right)(x)\d t \right\}^{\frac12},\\
H_{\theta}^{\uparrow}(x):&=\left\{\int_0^{\infty}t Q_t^{(\alpha),\kappa}\left(g_{\theta}^{\uparrow}(\cdot,t)^2 \right)(x)\d t \right\}^{\frac12},\\
H_{\theta}(x):&=\left\{\int_0^{\infty}t Q_t^{(\alpha),\kappa}\left(g_{\theta}(\cdot,t)^2 \right)(x)\d t \right\}^{\frac12}.
\end{align*}
\begin{lem}\label{lem:GlessH}
We have $G_{\theta}^{\rightarrow}(x)\leq 2\sqrt{C(\kappa)}H_{\theta}^{\rightarrow}(x)$. 
\end{lem}
\begin{proof}[{\bf Proof}] 
We calculate 
\begin{align*}
g_{\theta}^{\rightarrow}(x,2t)^2&=\left|\left.\frac{\partial }{\partial s}Q_s^{(\alpha),{\rm HK}}\theta(x)\right|_{s=2t} \right|^2=\left|\sqrt{\alpha-\DD}Q_{2t}^{(\alpha),{\rm HK}}\theta(x) \right|^2\\
&=\left|Q_t^{(\alpha),{\rm HK}}\sqrt{\alpha-\DD}Q_t^{(\alpha),{\rm HK}}\theta(x) \right|^2\\
&\leq\left(Q_t^{(\alpha),\kappa}\left|\sqrt{\alpha-\DD}Q_t^{(\alpha),{\rm HK}}\theta \right|(x) \right)^2\\
&\leq Q_t^{(\alpha),\kappa}1\cdot Q_t^{(\alpha),\kappa}\left|\frac{\partial }{\partial t}Q_t^{(\alpha),{\rm HK}}\theta \right|^2(x).
\end{align*}
Then 
\begin{align*}
G_{\theta}^{\rightarrow}(x)&=\left\{\int_0^{\infty}t g_{\theta}^{\rightarrow}(x,t)^2\d t \right\}^{\frac12}=\left\{4\int_0^{\infty}t g_{\theta}^{\rightarrow}(x,2t)^2\d t \right\}^{\frac12}\\
&\leq 2\left\{\int_0^{\infty}t Q_t^{(\alpha),\kappa}1\cdot Q_t^{(\alpha),\kappa}g_{\theta}^{\rightarrow}(\cdot,t)^2
(x)\d t \right\}^{\frac12}\\
&\leq 2\sqrt{C(\kappa)}\left\{\int_0^{\infty}t Q_t^{(\alpha),\kappa}g_{\theta}^{\rightarrow}(\cdot,t)^2(x)\d t \right\}^{\frac12}\\
&=2\sqrt{C(\kappa)}H_{\theta}^{\rightarrow}(x).
\end{align*}
\end{proof} 

\begin{thm}\label{thm:LittlewoodPaley2Form}
Suppose $p\in]2,+\infty[$, $\kappa^-_{\ll}=-R\m$ with $R\geq0$, $\kappa_{\perp}^-=0$ and $\alpha>C_{\kappa}$. Then 
the $G$-function  
$G_{\theta}^{\rightarrow}$ can be extended to be in $L^p(T^*\!M)$ 
for $\theta\in L^p(T^*\!M)$ and the following inequality holds for 
$\theta\in L^p(T^*\!M)$
\begin{align}
\|G_{\theta}^{\rightarrow}\|_{L^p(M;\m)}&\lesssim\|\theta\|_{L^p(T^*\!M)}.\label{eq:LittlewoodPaleyStein2}
\end{align}
\end{thm}
\begin{proof}[{\bf Proof}] 
By Lemma~\ref{lem:GlessH}, it suffices to prove $\|H_{\theta}\|_{L^p(M;\m)}\lesssim\|\theta\|_{L^p(T^*\!M)}$. 
\begin{align}
\|H_{\theta}\|_{L^p(M;\m)}&=\left\|\left\{\int_0^{\infty}t Q_t^{(\alpha),\kappa}\left(g_{\theta}(\cdot,t)^2 \right)\d t \right\}^{\frac{p}{2}} \right\|_{L^1(M;\m)}\notag\\
&=
\int_M\left\{\int_0^{\infty}t Q_t^{(\alpha),\kappa}\left(g_{\theta}(\cdot,t)^2 \right)\d t \right\}^{\frac{p}{2}}\d \m\label{eq:Htheta}\\
&=\lim_{n\to\infty}
\int_M\left\{\int_0^{\infty}(t\land a_n) Q_t^{(\alpha),\kappa}\left(g_{\theta}(\cdot,t)^2 \right)\d t \right\}^{\frac{p}{2}}\d\m\notag\\
&=\lim_{n\to\infty}\E_{\m\otimes\delta_{a_n}}\left[\left\{\int_0^{\infty}(t\land a_n) Q_t^{(\alpha),\kappa}\left(g_{\theta}(\cdot,t)^2(X_{\tau}) \right)\d t \right\}^{\frac{p}{2}} \right].\notag
\end{align}
By using \cite[(3.34)]{EJK}, the right hand side of \eqref{eq:Htheta} is estimated above by
\begin{align*}
\varliminf_{n\to\infty}&\E_{\m\otimes\delta_{a_n}}\left[\E_{\m\otimes\delta_{a_n}}\left[
\left. \int_0^{\tau}g_{\theta}^2(\wh{X}_s)\d s\,\right|\, X_{\tau}
\right]^{\frac{p}{2}} \right]\\
&=
\varliminf_{n\to\infty}\E_{\m\otimes\delta_{a_n}}\left[\E_{\m\otimes\delta_{a_n}}\left[
\left. \left(\int_0^{\tau}g_{\theta}^2(\wh{X}_s)\d s\right)^{\frac{p}{2}}\,\right|\, X_{\tau}
\right] \right]\\
&=\varliminf_{n\to\infty}\E_{\m\otimes\delta_{a_n}}\left[\left(\int_0^{\tau}g_{\theta}^2(\wh{X}_s)\d s\right)^{\frac{p}{2}} \right].
\end{align*}
On the other hand, under the conditions $\kappa^-_{\ll}=-R\m$ with $R\geq0$ and $\kappa^-_{\perp}=0$, the semimartingale 
\eqref{eq:FukushimaDecomp} is actually a submartingale, because
\begin{align*}
\wh{A}_t^{\,\left(\frac{\partial^2}{\partial a^2}+\text{\boldmath$\Delta$} \right)|\eta|^2}&=
N_t^{|\eta|^2}=2\int_0^t(\alpha|\eta|^2+|\overline{\nabla}\eta|^2)(\wh{X}_s)\d s
+2\wh{A}_t^{\,{\bf \bf Ric}(\eta^{\sharp},\eta^{\sharp})\otimes m}\\
&\geq 2\int_0^t((\alpha-R)|\eta|^2+|\overline{\nabla}\eta|^2)(\wh{X}_s)\d s\\
&\geq 2\int_0^t|\overline{\nabla}\eta|^2(\wh{X}_s)\d s=2\int_0^tg_{\theta}^2(\wh{X}_s)\d s
\end{align*} 
under $\alpha>C_{\kappa}=R$. Then one can apply Lenglart-Lepingle-Pratelli inequality (see \cite[(3.39)]{EJK}) and Doob's inequality so that 
\begin{align*}
\E_{\m\otimes\delta_{a_n}}\left[\left(\int_0^{\tau}g_{\theta}^2(\wh{X}_s)\d s\right)^{\frac{p}{2}} \right]&\leq 
2^{-\frac{p}{2}}\E_{\m\otimes\delta_{a_n}}\left[\left(\wh{A}_t^{\,\left(\frac{\partial^2}{\partial a^2}+\text{\boldmath$\Delta$} \right)|\eta|^2}\right)^{\frac{p}{2}} \right]\\
&\leq \left(\frac{p}{2}\right)^{\frac{p}{2}}\E_{\m\otimes\delta_{a_n}}\left[\left||\eta|^2(\wh{X}_{\tau}) -|\eta|^2(\wh{X}_0)\right|^{\frac{p}{2}} \right]\\
&=\left(\frac{p}{2}\right)^{\frac{p}{2}}\E_{\m\otimes\delta_{a_n}}\left[\left||Q_0^{(\alpha),{\rm  HK}}(X_{\tau})|^2 -|Q_{a_n}^{(\alpha),{\rm  HK}}(X_0)|^2\right|^{\frac{p}{2}} \right]\\
&\leq \left(\frac{p}{2}\right)^{\frac{p}{2}}\E_{\m\otimes\delta_{a_n}}\left[|Q_0^{(\alpha),{\rm  HK}}(X_{\tau})|^p +|Q_{a_n}^{(\alpha),{\rm  HK}}(X_0)|^p \right]\\
&\leq \left(\frac{p}{2}\right)^{\frac{p}{2}}\left(\|\theta\|_{L^p(T^*\!M)}+\| Q_{a_n}^{(\alpha),{\rm HK}}\theta\|_{L^p(M;\m)}^p \right)\\
&\leq \left(\frac{p}{2}\right)^{\frac{p}{2}}\left(\|\theta\|_{L^p(T^*\!M)}+Ce^{-\sqrt{\alpha-C_{\kappa}}a_n}\|\theta\|_{L^p(M;\m)}^p \right)\\
&\lesssim \|\theta\|_{L^p(T^*\!M)}.
\end{align*}
\end{proof} 
\section{Proof of Theorem~\ref{thm:main3}}\label{sec:Proof} 
For $f\in L^2(M;\m)\cap L^p(M;\m)$ and $\alpha>0$, we introduce the $G$-functions for $f$ as follows: 
\begin{align*}
{g_{\stackrel{\rightarrow}{f}}}(x,t)&:=\left|\frac{\partial}{\partial t}(Q_t^{(\alpha)}f)(x) \right|,\qquad\quad\quad {G_{\stackrel{\rightarrow}{f}}}(x):=\left(\int_0^{\infty}t{g_{\stackrel{\rightarrow}{f}}}(x,t)^2\d t  \right)^{1/2},\\
{g_f^{\uparrow}}
(x,t)&:=\left|\d (Q_t^{(\alpha)}f)(x)\right|,\qquad\qquad\quad {G_f^{\uparrow}}(x):=\left(\int_0^{\infty}t{g_f^{\uparrow}}(x,t)^2\d t  \right)^{1/2},\\
g_f(x,t)&:=\sqrt{{g_{\stackrel{\rightarrow}{f}}}(x,t)^2+
{g_f^{\uparrow}}
(x,t)^2}, \qquad\quad G_f(x):=\left(\int_0^{\infty}tg_f(x,t)^2\d t \right)^{1/2}.
\end{align*}
It is proved in \cite[Theorem~1.2]{EJK} that $G$-functions $g_f^{\rightarrow}(\cdot,t)$, $g_f^{\uparrow}(\cdot,t)$, $g_f(\cdot,t)$, 
$G_f^{\rightarrow}$, $G_f^{\uparrow}$ and $G_f$ can be extended for $f\in L^p(M;\m)$ and 
$\|G_f\|_{L^p(M;\m)}\lesssim\|f\|_{L^p(M;\m)}$, in particular, $\|G_f^{\uparrow}\|_{L^p(M;\m)}\lesssim\|f\|_{L^p(M;\m)}$.

\begin{lem}\label{lem:RieszBdd}
Take $p\in]1,+\infty[$ and its conjugate exponent $q:=p/(p-1)\in ]1,+\infty[$. For $f\in L^p(M;\m)$, we define the 
Riesz transform $\mathscr{R}_{\alpha}(\Delta)f$ by 
\begin{align*}
\mathscr{R}_{\alpha}(\Delta)f:=\d(\alpha-\Delta_p)^{-\frac12}f.
\end{align*}
Then we have 
\begin{align}
\left|{}_{L^p(T^*\!M)}\left\langle \mathscr{R}_{\alpha}(\Delta)f,\theta\right\rangle_{L^q(T^*\!M)} \right|\leq 4\int_MG_f^{\uparrow}(x)G_{\theta}^{\rightarrow}(x)\m(\d x),\quad\theta\in L^q(T^*\!M).\label{eq:FundamentalIneq}
\end{align}
\end{lem}
\begin{proof}[{\bf Proof}] 
First we prove the assertion for $p=q=2$. In this case, we have the following identity: For $\theta,\eta\in L^2(T^*\!M)$,
\begin{align}
4\int_0^{\infty}t\,\left\langle \frac{\partial }{\partial t}Q_t^{(\alpha),{\rm HK}}\theta, \frac{\partial }{\partial t}Q_t^{(\alpha),{\rm HK}}\eta \right\rangle_{L^2(T^*\!M)}\d t=\langle \theta,\eta\rangle_{L^2(T^*\!M)}. \label{eq:Identity*}
\end{align}
Indeed, the left hand side of \eqref{eq:Identity*} is 
\begin{align*}
4\int_0^{\infty}t\left\langle \sqrt{\alpha-\DD}\right. &Q_t^{(\alpha),{\rm HK}}\theta,\left.\sqrt{\alpha-\DD}Q_t^{(\alpha),{\rm HK}}\eta\right\rangle_{L^2(T^*\!M)}\d t\\
&=4\int_0^{\infty}t\,\langle (\alpha-\DD)Q_{2t}^{(\alpha),{\rm HK}}\theta, \eta\rangle_{L^2(T^*\!M)}\d t\\
&=\left\langle \int_0^{\infty}t\,(\alpha-\DD)Q_t^{(\alpha),{\rm HK}}\theta \d t,\eta \right\rangle_{L^2(T^*\!M)}\\
&=\langle \theta,\eta\rangle_{L^2(T^*\!M)},
\end{align*} 
where we use 
\begin{align*}
\int_0^{\infty}t\,(\alpha-\DD)Q_t^{(\alpha),{\rm HK}}\theta \,\d t&=\int_0^{\infty}e^{-\alpha s}
(\alpha-\DD)P_s^{\rm HK}\theta \left(\int_0^{\infty}t\left(\frac{t}{2\sqrt{\pi}}e^{-\frac{t^2}{4s}}s^{-\frac32} \right)\d t\right)\d s\\
&\hspace{-0.3cm}\stackrel{\eqref{eq:Expectation}}{=}\int_0^{\infty}e^{-\alpha s}
(\alpha-\DD)P_s^{\rm HK}\theta\,\d s\\
&=-\int_0^{\infty}\frac{\partial}{\partial s}\left(e^{-\alpha s}P_s^{\rm HK}\theta \right)\d s\\
&=-\left[e^{-\alpha s}P_s^{\rm HK}\theta \right]_0^{\infty}=\theta. 
\end{align*}
On the other hand, the intertwining property \eqref{eq:intertwining1} yields that 
$Q_t^{(\alpha),{\rm HK}}(\d g)=\d Q_t^{(\alpha)}g$ holds for $g=(\alpha-\Delta)^{-\frac12}f\in D(\mathscr{E})$. 
Differentiating the both sides with respect to $t$, we have 
\begin{align*}
(\alpha-\DD)^{\frac12}Q_t^{(\alpha),{\rm HK}}(\d ((\alpha-\Delta)^{-\frac12}f))&
=\d(\alpha-\Delta)^{\frac12}Q_t^{(\alpha)}(\alpha-\Delta)^{-\frac12}f=\d Q_t^{(\alpha)}f.
\end{align*} 
Thus
\begin{align}
-\frac{\partial}{\partial t}Q_t^{(\alpha),{\rm HK}}\left(\d(\alpha-\Delta)^{-\frac12}f \right)&=
(\alpha-\DD)^{\frac12}Q_t^{(\alpha),{\rm HK}}\left(\d(\alpha-\Delta)^{-\frac12}f \right)=\d Q_t^{(\alpha)}f.\label{eq:Identity**}
\end{align}
Set $\eta:=\mathscr{R}_{\alpha}(\Delta)f=\d(\alpha-\Delta)^{-\frac12}f$ for $f\in L^2(M;\m)$. From 
\eqref{eq:Identity*} and \eqref{eq:Identity**}, 
\begin{align*}
\left|\langle \mathscr{R}_{\alpha}(\Delta)f,\theta\rangle_{L^2(T^*\!M)}\right|&=4\left|\int_0^{\infty}t\,\left\langle \d Q_t^{(\alpha)}f, \frac{\partial}{\partial t}
Q_t^{(\alpha),{\rm HK}}\theta \right\rangle_{L^2(T^*\!M)}\d t\right|\\
&\leq 4\int_0^{\infty}t\,\int_M
\left|\d Q_t^{(\alpha)}f(x) \right|\cdot
\left|\frac{\partial}{\partial t}Q_t^{(\alpha),{\rm HK}}\theta(x)\right|\m(\d x)\d t\\
&\leq 4\int_M G_f^{\uparrow}(x)G_{\theta}^{\rightarrow}(x)\m(\d x).
\end{align*}
Then we have \eqref{eq:FundamentalIneq} for $f\in L^2(M;\m)$ and $\theta\in L^2(T^*\!M)$ with $p=q=2$. 

Finally, we prove \eqref{eq:FundamentalIneq} for general $f\in L^p(M;\m)$ and $\theta\in L^q(T^*\!M)$ with $p\in]1,+\infty[$ and $q:=p/(p-1)\in ]1,+\infty[$. Take $f\in L^p(M;\m)$ and $\theta\in L^q(T^*\!M)$. Let $\{f_n\}\subset L^p(M;\m)\cap L^2(M;\m)$ 
(resp.~$\{\theta_n\}\subset L^q(T^*\!M)$) be an approximating sequence to $f\in L^p(M;\m)$ (resp.~$\theta\in L^q(T^*\!M)$) in 
$L^p(M;\m)$ (resp.~in $L^q(T^*\!M)$).  Then, for any $g\in L^p(M;\m)\cap L^2(M;\m)$
\begin{align*}
\left|{}_{L^p(T^*\!M)}\left\langle \mathscr{R}_{\alpha}(\Delta)g,\theta_n-\theta_m \right\rangle_{L^q(T^*\!M)} \right|&\leq 4\int_M G_{g}^{\uparrow}(x)G_{\theta_n-\theta_m}^{\rightarrow}(x)\m(\d x)\\
&\leq 4\|G_{g}^{\uparrow}\|_{L^p(M;\m)}\|G_{\theta_n-\theta_m}^{\rightarrow}\|_{L^q(M;\m)}\\
&\lesssim \|g\|_{L^p(M;\m)}\|\theta_n-\theta_m\|_{L^q(T^*\!M)}\to0\quad\text{ as }\quad n, m\to\infty.
\end{align*}
We set for $\theta\in L^q(T^*\!M)$
\begin{align*}
{}_{L^p(T^*\!M)}\left\langle \mathscr{R}_{\alpha}(\Delta)g,\theta \right\rangle_{L^q(T^*\!M)}:=\lim_{n\to\infty}{}_{L^p(T^*\!M)}
\left\langle \mathscr{R}_{\alpha}(\Delta)g,\theta_n\right\rangle_{L^q(T^*\!M)}.
\end{align*}
Then, one can deduce
\begin{align*}
\left|{}_{L^p(T^*\!M)}\left\langle \mathscr{R}_{\alpha}(\Delta)g,\theta \right\rangle_{L^q(T^*\!M)} \right|&\leq 
4\int_M G_{g}^{\uparrow}(x)G_{\theta}^{\rightarrow}(x)\m(\d x)\\
&\leq 4\|G_{g}^{\uparrow}\|_{L^p(M;\m)}\|G_{\theta}^{\rightarrow}\|_{L^q(M;\m)}\\
&\lesssim \|g\|_{L^p(M;\m)}\|\theta\|_{L^q(T^*\!M)},
\end{align*}
where we use $|G_{\theta_1}^{\rightarrow}-G_{\theta_2}^{\rightarrow}|\leq G_{\theta_1-\theta_2}^{\rightarrow}$ for $\theta_1,\theta_2\in L^q(T^*\!M)$. 
Similarly, for $\theta\in L^q(T^*\!M)$
\begin{align*}
\left|{}_{L^p(T^*\!M)}\left\langle \mathscr{R}_{\alpha}(\Delta)(f_n-f_m),\theta \right\rangle_{L^q(T^*\!M)} \right|&\leq 4\int_M G_{f_n-f_m}^{\uparrow}(x)G_{\theta}^{\rightarrow}(x)\m(\d x)\\
&\leq 4\|G_{f_n-f_m}^{\uparrow}\|_{L^p(M;\m)}\|G_{\theta}^{\rightarrow}\|_{L^q(M;\m)}\\
&\lesssim \|f_n-f_m\|_{L^p(M;\m)}\|\theta\|_{L^q(T^*\!M)}\to0\quad\text{ as }\quad n,m\to\infty.
\end{align*}
Hence, we have \eqref{eq:FundamentalIneq} for $f\in L^p(M;\m)$ and $\theta\in L^q(T^*\!M)$ with $p\in]1,+\infty[$ and $q=p/(p-1)\in ]1,+\infty[$. 
\end{proof} 
\begin{proof}[\bf Proof of Theorem~\ref{thm:main3}]
By Theorems~\ref{thm:LittlewoodPaley1Form} and \ref{thm:LittlewoodPaley2Form}, under the present conditions, we always have 
\begin{align*}
\|G_{\theta}^{\rightarrow}\|_{L^p(M;\m)}\lesssim\|\theta\|_{L^p(T^*\!M)}\quad \text{ for }\quad \theta\in L^p(T^*\!M),\quad p\in ]1,+\infty[. 
\end{align*}
By Lemma~\ref{lem:RieszBdd}, for $f\in L^p(M;\m)$ and $\theta\in L^q(T^*\!M)$ with $p\in]1,+\infty[$ and $q:=p/(p-1)\in ]1,+\infty[$, 
\begin{align*}
\left|{}_{L^p(T^*\!M)}\left\langle \mathscr{R}_{\alpha}(\Delta)f,\theta\right\rangle_{L^q(T^*\!M)} \right|\lesssim\|G_f^{\uparrow}\|_{L^p(M;\m)}\|G_{\theta}^{\rightarrow}\|_{L^q(M;\m)}\lesssim\|f\|_{L^p(M;\m)}\|\theta\|_{L^q(T^*\!M)}.
\end{align*} 
This implies 
\begin{align*}
\|{R}_{\alpha}(\Delta)f\|_{L^p(M;\m)}=\|\mathscr{R}_{\alpha}(\Delta)f\|_{L^p(T^*\!M)}\lesssim\|f\|_{L^p(M;\m)}\quad\text{ for }\quad p\in]1,+\infty[\quad\text{ and }\quad \alpha>C_{\kappa}.
\end{align*}
\end{proof} 
\section{Examples}\label{sec:examples}
It is well-known that (abstract) Wiener space $(B,H,\mu)$ satisfies Littlewood-Paley-Stein inequality for $L^p(B,K;\mu)$-functions with a real separable Hilbert space $K$ (see Shigekawa~\cite[Chapter~3]{ShigekawaText}). 
Here $L^p(B,K;\mu)$ is a $K$-valued $L^p$-integrable functions.  
Though our Theorem~\ref{thm:main3} is not new for $(B,H,\mu)$, we observe the conditions:   
Let $(\mathscr{E}^{\rm OU},D(\mathscr{E}^{\rm OU}))$ be the Dirichlet form on 
$L^2(B;\mu)$ associated to the Ornstein-Uhlenbeck process ${\bf X}^{\rm OU}$ and $(T_t^{\rm OU})_{t\geq0}$ its associated semigroup on $L^2(B;\mu)$. 
Let $D_H$ be the $H$-derivative, that is, $\langle D_HF(z), h\rangle_H=\lim_{\eps\to0}\frac{F(z+\eps h)-F(z)}{\eps}$, for a cylindrical function $F$. It is known that  
${\sf BE}_2(1,\infty)$-condition holds for $(B,\mathscr{E}^{\rm OU},\mu)$ (see \cite[13.2]{AES}), that is, 
$(B,\mathscr{E}^{\rm OU},\mu)$ is tamed by $\mu\in S_K({\bf X}^{\rm OU})$. 
In this case, the measure-valued Ricci lower bound $\kappa=\mu$ satisfies $\kappa_{\ll}=\mu$ and $\kappa_{\perp}=0$. 
It is easy to see that $L^p(B,K;\mu)$ forms an $L^p$-normed $L^{\infty}$-module over $(B,\mu)$. Indeed, 
for $f,g\in L^{\infty}(B;\mu)$ and $v\in L^p(B,K;\mu)$, we see $f\cdot v\in L^p(B,K;\mu)$ 
satisfying $(fg)\cdot v=f\cdot(g\cdot v)$ and $\1_B\cdot v=v$, and 
$|\cdot|_K: L^p(B,K;\mu)\to L^p(B;\mu)$ is actually a point-wise norm, i.e.  
$|f\cdot v|_K=|f|\cdot|v|_K$ $\mu$-a.e. and $\|v\|_{L^p(B,K;\mu)}=\||v|_K\|_{L^p(B;\mu)}$ hold. 
The conclusions in Theorems~\ref{thm:LittlewoodPaley1Form} and \ref{thm:LittlewoodPaley2Form}, hence Theorem~\ref{thm:main3} hold for $(B,H,\mu)$ and $\alpha>C_{\kappa}=0$ in terms of $L^p(B,H,\mu)$. Thus 
we have the boundedness of 
Riesz operator $R_{\alpha}(\Delta^{\rm OU})$ on $L^p(B;\mu)$. 

\bigskip

Moreover, in Kawabi~\cite{Kawabi}, 
the content of Theorem~\ref{thm:main3} is also announced in the framework of diffusion semigroup of ${\rm P}(\phi)_1$-quantum fields in infinite volume. 
More concretely, let $E$ be a Hilbert space defined by $E:=L^2(\R,\R^d;e^{-2\lambda\chi(x)}\d x)$ with a fixed $\chi\in C^{\infty}(\R)$ satisfying $\chi(x)=|x|$ for $|x|\geq1$ and another Hilbert space $H:=L^2(\R,\R^d;\d x)$.   
They consider a Dirichlet form $(\mathscr{E}, D(\mathscr{E}))$ on $L^2(E;\mu)$ associated with the diffusion process ${\bf X}$ on an infinite volume 
 path space $C(\R,\R^d)$ with ($U$-)Gibbs measures $\mu$ associated with the (formal) Hamiltonian
 \begin{align*}
 \mathcal{H}(w):=\frac12\int_{\R}|w'(x)|^2_{\R^d}\d x+\int_{\R}U(w(x))\d x, 
 \end{align*}  
 where $U: \R^{d} \to \R$ is an interactions potential satisfying ${\nabla^2}U\geq -K_{1}$
with some $K_{1}\in\R$  
 (see \cite[4.1]{KawabiMiyokawa}). 
 Then the $L^2$-semigroup $(P_t)_{t\geq0}$ associated to $(\mathscr{E}, D(\mathscr{E}))$ satisfies the following gradient estimate
\begin{align*}
|D(P_tf)(w)|_H\leq e^{K_1t}P_t(|Df|_H)(w)\quad \text{ for }\quad \mu\text{-a.e.~}w\in E,
\end{align*}
which is equivalent to ${\sf BE}_2(-K_1,\infty)$-condition. 
Here $Df$ is a closed extension of the Fr\'echet derivative $Df:E\to H$ for cylindrical function $f$.    
Hence $(E,\mathscr{E},\mu)$ is tamed by $\kappa:=-K_1\mu$ with $K_1\mu\in S_K({\bf X})$. 
In this case, the measure-valued Ricci lower bound $\kappa$ satisfies $\kappa_{\ll}=-K_1\mu$ and 
$\kappa_{\perp}=0$. Denote by $L^p(E,H;\mu)$ an $H$-valued $L^p$-integrable functions. 
 $L^p(E,H;\mu)$ is an $L^p$-normed $L^{\infty}$-module as remarked above. 
Then, Theorem~\ref{thm:main3} in terms of $L^p(E,H;\mu)$
holds for $(E,\mathscr{E},\mu)$ under any $\alpha>C_{\kappa}=K_1\lor 0$ and $p\in]1,+\infty[$. 

\bigskip

Denote by $(M,g,\m)$, a complete weighted Riemannian manifold without boundary, where $\m:=e^{-\phi}\mathfrak{v}$ is the 
weighted measure with a potential function $\phi\in C^2(M)$ and volume measure $\mathfrak{v}$. Let $(\mathscr{E},D(\mathscr{E}))$ be the Dirichlet form on $L^2(M;\m)$ defined by 
the closure of $(\mathscr{E},C_c^{\infty}(M))$: 
\begin{align*}
\mathscr{E}(f,g)=\int_M\langle \nabla f,\nabla g\rangle\d\m, \quad f,g\in  C_c^{\infty}(M).
\end{align*}
The $L^2(M;\m)$-generator $L$ associated with $(\mathscr{E},D(\mathscr{E}))$ is given by 
$Lf=\Delta f-\langle \nabla f,\nabla\phi\rangle$ for $f\in C_c^{\infty}(M)$. 
Denote by ${\bf X}$ the $\m$-symmetric diffusion process associated with $(\mathscr{E},D(\mathscr{E}))$. 
Let ${\rm Ric}_\phi^{\infty}$ be Bakry-\'Emery Ricci curvature defined by 
\begin{align*}
{\rm Ric}_\phi^{\infty}(u,v):={\rm Ric}_g(u,v)+{\rm Hess}\,\phi(u,v),\quad u,v\in  TM.
\end{align*}
We assume ${\rm Ric}_\phi^{\infty}\geq kg$ for a measurable function $k$ on $M$. In this case, measure-valued Ricci lower bound $\kappa$ satisfies 
$\kappa_{\ll}=k\m$ and $\kappa_{\perp}=0$. Theorem~\ref{thm:LittlewoodPaley1Form} in the present setting is proved by X.-D.~Li~\cite[Theorem~4.3]{Xdli:RieszTrans}. Based on this, X.-D.~Li~\cite[Theorem~2.1]{Xdli:RieszTrans} obtained the boundedness of Riesz transform $\mathscr{R}_{\alpha}(L)$ on $L^p(M;\m)$ for $p\in[2,+\infty[$ under $k\in S_K({\bf X})$ (or under a weaker assumption).  Theorem~\ref{thm:LittlewoodPaley2Form} in the present setting is proved by Bakry~\cite{Bakry1}. 
Thus we have the boundedness of Riesz transform $\mathscr{R}_{\alpha}(L)$ on $L^p(M;\m)$ for $p\in]1,2[$ and $\alpha>C_{\kappa}:=(-\inf_Mk)\lor0$ under 
the lower boundedness of $k$.

\bigskip

In the rest, we expose new examples. 
\begin{example}[{RCD spaces}]
{\rm A metric measure space $(M,{\sf  d},\m)$ is a complete separable metric space $(M,{\sf d})$ with a $\sigma$-finite Borel measure with $\m(B)<\infty$ for any bounded Borel set $B$. We assume $\m$ has full topological support, i.e., $\m(G)>0$ for non-empty open set $G$. 

Any metric open ball is denoted by $B_r(x):=\{y\in M\mid {\sf d}(x,y)<r\}$ for $r>0$ and $x\in M$. 
A subset $B$ of $M$ is said to be bounded if it is included in a metric open ball.    
Denote by $C([0,1],M)$ the space of continuous curve defined on the unit interval $[0,1]$ equipped the distance ${\sf d}_{\infty}(\gamma,\eta):=\sup_{t\in[0,1]}{\sf d}(\gamma_t,\eta_t)$ for every $\gamma,\eta\in C([0,1],M)$. This turn $C([0,1],M)$ into complete separable metric space.   
Next we consider the set of $2$-absolutely continuous curves, denoted by $AC^q([0,1],M)$, is the subset of $\gamma\in C([0,1],M)$ so that there exists  $g\in L^q(0,1)$ satisfying 
\begin{align*}
{\sf d}(\gamma_t,\gamma_s)\leq\int_s^tg(r)\d r,\quad s<t\text{ in }[0,1].
\end{align*}
Recall that for any $\gamma\in AC^2([0,1],M)$, there exists a minimal a.e.~function $g\in L^2(0,1)$ satisfying the above, called {\it metric speed} denoted by $|\dot\gamma_t|$, which is defined as 
$|\dot\gamma_t|:=\lim_{h\downarrow0}{\sf d}(\gamma_{t+h}, \gamma_t)/h$ for $\gamma\in AC^2([0,1],M)$, 
$|\dot\gamma_t|:=+\infty$ otherwise. 
We define the kinetic energy functional $C([0,1],M)\ni\gamma\mapsto {\sf Ke}_t(\gamma):=\int_0^1|\dot{\gamma}_t|^2\d t$, if $\gamma\in AC^2([0,1],M)$, ${\sf Ke}_t(\gamma):=+\infty$ otherwise. 
\begin{defn}[$2$-test plan]\label{def:$q$-test}
{\rm Let $(M,{\sf d},\m)$ be a metric measure space. 
A measure {\boldmath$\pi$}$\in\mathscr{P}(C([0,1],M))$ is said to be a $2$-test plan, provided 
\begin{enumerate}
\item[(i)]\label{item:qtest1} there exists $C>0$ so that $({\sf e}_t)_{\sharp}${\boldmath$\pi$}$\leq C\m$ for every $t\in[0,1]$;
\item[(ii)]\label{item:qtest2} we have $\int_{C([0,1],M)}{\sf Ke}_2(\gamma)${\boldmath$\pi$}$(\d\gamma)<\infty$. 
\end{enumerate} 
}
\end{defn}

\begin{defn}[Sobolev space $W^{\hspace{0.03cm}1,2}(M)$]\label{def:W1pSobolev}
{\rm
A Borel function $f\in L^2(M;\m)$ belongs to $W^{1,2}(M)$, 
provided there exists a $G\in L^2(M;\m)$,  
called {\it $2$-weak upper gradient} of $f$ so that 
\begin{align}
\int_{C([0,1],M)}|f(\gamma_1)-f(\gamma_0)|\text{\boldmath$\pi$}(\d\gamma)\leq \int_{C([0,1],M)}\int_0^1
G(\gamma_t)|\dot{\gamma}_t|\d t\text{\boldmath$\pi$}(\d\gamma),\quad \text{$\forall${\boldmath$\pi$} 
$2$-test plan}. \label{eq:SpSobolev}
\end{align} 
The assignment $(t,\gamma)\mapsto G(\gamma_t)|\dot{\gamma}_t|$ is Borel measurable (see \cite[Remark~2.1.9]{GPLecture}) and the right hand side of 
\eqref{eq:SpSobolev} is finite for $G\in L^2(M;\m)_+$ (see \cite[(2.5)]{GigliNobili}). These shows not only the finiteness of the right hand side of \eqref{eq:SpSobolev} but also the continuity of the assignment 
$L^2(M;\m)\ni G\mapsto  \int_{C([0,1],M)}\int_0^1 G(\gamma_t)|\dot{\gamma}_t|\d t\text{\boldmath$\pi$}(\d\gamma)$. This, combined with the closedness of the convex combination of the $2$-weak upper gradient, shows that the set of $2$-weak upper gradient of a given Borel function $f$ is a closed convex subset of $L^2(M;\m)$. The minimal $p$-weak upper gradient, denoted by $|Df|_2$ is then the element of 
minimal $L^p$-norm in this class. Also, by making use of the lattice property of the set of $2$-weak 
upper gradient, such a minimality is also in the $\m$-a.e. sense (see \cite[Proposition~2.17 and Theorem~2.18]{Ch:metmeas}). 

Then, $W^{1,2}(M)$ forms a Banach space equipped with the following norm:
equipped with the norm 
\begin{align*}
\|f\|_{W^{1,2}(M)}:=\left(\|f\|^2_{L^2(M;\m)}+\| |Df|_2\|^2_{L^2(M;\m)} \right)^{\frac{1}{2}} ,\quad f\in  W^{1,2}(M). 
\end{align*}
 
}
\end{defn}

It is in general false 
that  $(W^{1,2}(M),\|\cdot\|_{W^{1,2}(M)})$ is a Hilbert space. 
When this occurs, we say that $(M,{\sf d},\m)$ is {\it infinitesimally Hilbertian} (see 
\cite{Gigli:OntheDifferentialStr}). Equivalently, we call $(X,{\sf d},\m)$  infinitesimally Hilbertian 
provided the following {\it parallelogram identity} holds:
\begin{align}
2|Df|_2^2+2|Dg|_2^2=|D(f+g)|_2^2+|D(f-g)|_2^2, \quad \m\text{-a.e.} \quad \forall f,g\in W^{1,2}(M).\label{eq:parallelogram}
\end{align}

For simplicity, when $p=2$, we omit the suffix $2$ from $|Df|_2$ for $f\in W^{1,2}(M)$, i.e. we write $|Df|$ instead of $|Df|_2$ for $f\in W^{1,2}(M)$.  
Under \eqref{eq:parallelogram}, we can give a bilinear form $\langle D\cdot ,D\cdot\rangle:W^{1,2}(M)\times W^{1,2}(M)\to L^1(M;\m)$ which is defined by 
\begin{align*}
\langle Df,Dg\rangle:=\frac14|D(f+g)|^2-\frac14|D(f-g)|^2,\quad f,g\in W^{1,2}(M).
\end{align*}
Moreover, under the infinitesimally Hilbertian condition, 
the bilinear form $(\mathscr{E},D(\mathscr{E}))$ defined by 
\begin{align*}
D(\mathscr{E}):=W^{1,2}(M),\quad \mathscr{E}(f,g):=\frac12\int_M\langle Df,Dg\rangle\d \m
\end{align*}
is a strongly local Dirichlet form on $L^2(M;\m)$. 
Denote by $(P_t)_{t\geq0}$ the $\m$-symmetric semigroup on $L^2(M;\m)$ associated with $(\mathscr{E},D(\mathscr{E}))$. 
Under \eqref{eq:parallelogram}, let $(\Delta, D(\Delta))$ be the 
$L^2$-generator associated with $(\mathscr{E},D(\mathscr{E}))$ similarly defined as in \eqref{eq:generatorL2} before. 
\begin{defn}[RCD-spaces]
{\rm A metric measure space $(X,{\sf d},\m)$ is said to be an {\it {\sf RCD}$(K,\infty)$-space} if 
it satisfies 
the following conditions: 
\begin{enumerate}
\item[\rm(1)]
$(M, {\sf d}, \m)$ is infinitesimally Hilbertian. 

\item[\rm(2)]
There exist $x_0 \in M$ and constants $c, C > 0$ such that 
$\m(B_r(x_0)) \le C \e^{c r^2}$. 

\item[\rm(3)]
If $f \in W^{1,2}(M)$ satisfies 
$| D f | \le 1$ $\m$-a.e., then $f$ has a $1$-Lipschitz representative. 

\item[\rm(4)]
For any $f \in D ( \Delta )$ 
with $\Delta f \in W^{1,2}(M)$ 
and $g \in D ( \Delta ) \cap L^\infty (M; \m)$ 
with $g \ge 0$ and $\Delta g \in L^\infty (M; \m)$, 
\begin{align*}
\frac12 \int_M | D f |^2 \Delta g \, \d \m 
- \int_M \langle D f, D \Delta f \rangle g \, \d \m 
\ge 
K \int_M | D f |^2 g \, \d \m 
\end{align*}
\end{enumerate}
Let $N\in[1,+\infty[$. 
A metric measure space $(M,{\sf d},\m)$ is said to be an {\it {\sf RCD}${}^*(K,N)$-space} if 
it is an {\sf RCD}$(K,\infty)$-space and 
for any $f \in D ( \Delta )$ 
with $\Delta f \in W^{1,2}(M)$ 
and $g \in D( \Delta ) \cap L^\infty (M; \m)$ 
with $g \ge 0$ and $\Delta g \in L^\infty (M; \m)$, 
\begin{align*}
\frac12 \int_M | D f |^2 \Delta g \, \d \m 
- \int_M \langle D f, D \Delta f \rangle g \, \d \m 
\ge 
K \int_M | D f |^2 g \, \d \m 
+ \frac{1}{N} \int_M ( \Delta f )^2 g \, \d \m. 
\end{align*}

}
\end{defn}
\begin{remark}
{\rm 
\begin{enumerate}
\item[\rm(1)] It is shown in \cite{Cav-Mil,ZLi} that for 
$N\in[1,+\infty[$ {\sf RCD}$(K,N)$-space is equivalent to 
{\sf RCD}${}^*(K,N)$-space, where {\sf RCD}${}^*(K,N)$-space  (resp.~{\sf RCD}${}^*(K,N)$-space) is defined to be a {\sf CD}${}^*(K,N)$-space 
(resp.~{\sf CD}${}^*(K,N)$-space) having infinitesimal Hilbertian condition. 
Originally, for $N\in[1,+\infty]$, the notion of {\sf CD}$(K,N)$-spaces was defined by Lott-Villani~\cite{LV2} and Sturm~\cite{StI,StII}. Later, Bacher-Sturm~\cite{Bacher-Sturm} gave the notion of 
{\sf CD}${}^*(K,N)$-space as a variant for $N\in[1,+\infty[$. Finally, Erbar-Kuwada-Sturm~\cite{EKS} invented the notion of 
{\sf CD}${}^e(K,N)$-space, so-called the metric measure space satisfying entropic curvature dimension condition under $N\in[1,+\infty[$ and prove that {\sf RCD}${}^*(K,N)$-space coincides with {\sf RCD}${}^e(K,N)$-space, i.e., 
{\sf CD}${}^e(K,N)$-space satisfying infinitesimal Hilbertian condition. 
The notion {\sf RCD}$(K,\infty)$-space, i.e., {\sf CD}$(K,\infty)$-space satisfying 
infinitesimal Hilbertian condition, was firstly considered in \cite{AGS_Riem,AGS_BakryEmery}, and 
the {\sf RCD}$(K,N)$-space under $N\in[1,+\infty[$ was also given by \cite{EKS,Gigli:OntheDifferentialStr}.
Moreover, under $N\in[1,+\infty[$, {\sf RCD}$(K,N)$-space (or {\sf RCD}${}^*(K,N)$-space) is a locally compact separable metric space, consequently, $\m$ becomes a Radon measure.  
\item[\rm(2)] If $(M,{\sf d},\m)$ is an {\sf RCD}$(K,N)$-space under $N\in[1,+\infty[$, 
it enjoys the Bishop-Gromov inequality: 
Let $\kappa:=K/(N-1)$ if $N>1$ and $\kappa:=0$ if $N=1$. We set $\omega_N:=\frac{\pi^{N/2}}{\int_0^{\infty}t^{N/2}e^{-t}\d t}$ (volume of unit bal in $\R^N$ provided $N\in\mathbb{N}$) and 
$V_{\kappa}(r):=\omega_N\int_0^r\s_{\kappa}^{N-1}(t)\d t$.  
 Then 
\begin{align}
\frac{\m(B_R(x))}{V_{\kappa}(R)}\leq \frac{\m(B_r(x))}{V_{\kappa}(r)},\qquad x\in M, \qquad 0<r<R. \label{eq:BishopGromov}
\end{align}
Here  
$\s_{\kappa}(s)$ is the solution to Jacobi equation $\s_{\kappa}''(s)+\kappa\s_{\kappa}(s)=0$ with 
$\s_{\kappa}(0)=0$, $\s_{\kappa}'(0)=1$. More concretely, $\s_{\kappa}(s)$ is given by
\begin{align*}
\s_{\kappa}(s):=\left\{\begin{array}{cc}\frac{\sin \sqrt{\kappa}s}{\sqrt{\kappa}} & \kappa>0, \\ s & \kappa=0, \\ \frac{\sinh \sqrt{-\kappa}s}{\sqrt{-\kappa}} & \kappa<0.\end{array}\right.
\end{align*}
\item[\rm(3)] If $(M,{\sf d},\m)$ is an $\mathsf{RCD}(K,\infty)$-space, it satisfies the following Bakry-\'Emery estimate:
\begin{align}
|DP_tf|\leq e^{-Kt}P_t|Df|\quad \m\text{-a.e.}\quad for \quad f\in W^{1,2}(M),\label{eq:BakryEmery}
\end{align}
in particular, $P_tf\in W^{1,2}(M)$
(see \cite[Corollary~3.5]{Sav14}, \cite[Proposition~3.1]{GigliHan}).
\item[\rm(4)] 
If $(M,{\sf d},\m)$ is an $\mathsf{RCD}(K,N)$-space with $N\in[1,+\infty[$, then,  
for any $f\in {\rm Lip}(M)$ 
${\rm lip}(f)=|Df|$ $\m$-a.e. holds, 
because  $\mathsf{RCD}(K,N)$ with $N\in[1,+\infty[$ admits the local volume doubling and a weak local $(1,2)$-Poincar\'e inequality (see \cite[\S5 and \S6]{Ch:metmeas}). 
\end{enumerate}
}
\end{remark}
By definition, for ${\sf RCD}(K,N)$-space $(M,{\sf d},\m)$,  $(M,\mathscr{E},\m)$ 
satisfies ${\sf BE}_2(K,N)$-condition and $K^{\pm}\m$ is always a Kato class smooth measure, 
hence it is a tamed Dirichlet space by $K\m$. 
In this case, the measure-valued Ricci lower bound $\kappa=K\m$ satisfies $\kappa_{\ll}=K\m$ and $\kappa_{\perp}=0$. 
So Theorem~\ref{thm:main3} holds for $(M,{\sf d},\m)$ and $\alpha>C_{\kappa}=(-K)\lor0$.

}
\end{example}

\begin{example}[Riemannian manifolds with boundary]
{\rm \quad Let $(M,g)$ be a smooth Riemannian manifold with boundary $\partial M$. 
Denote by $\mathfrak{v}:={\rm vol}_g$ the Riemannian volume measure induced by $g$, and by 
$\mathfrak{s}$ the surface measure on $\partial M$
(see \cite[\S1.2]{Braun:Tamed2021}). 
If $\partial M\ne \emptyset$, then $\partial M$ is a smooth co-dimension $1$ submanifold of $M$ and it becomes Riemannian 
when endowed with the pulback metric 
\begin{align*}
\langle \cdot,\cdot\rangle_j:=j^*\langle \cdot,\cdot\rangle,\qquad \langle u,v:=g(u,v)\quad\text{ for }\quad u,v\in TM.
\end{align*}  
under the natural inclusion $j:\partial M\to M$. The map $j$ induces a natural inclusion $d_j:T\partial M\to TM|_{\partial M}$ which is not surjective. In particular, the vector bundles $T\partial M$ and $TM|_{\partial M}$ do not coincide. 

Let $\m$ be a Borel measure on $M$ which is locally equivalent to $\mathfrak{v}$. 
Let $D(\mathscr{E}):=W^{1,2}(M^{\circ})$ be the Sobolev space with respect to 
$\m$ defined in the usual sense on $M^{\circ}:=M\setminus\partial M$. Define $\mathscr{E}:W^{1,2}(M^{\circ})\to [0,+\infty[$ 
by 
\begin{align*}
\mathscr{E}(f):=\int_{M^{\circ}}|\nabla f|^2\d \m
\end{align*}  
and the quantity $\mathscr{E}(f,g)$, $f,g\in W^{1,2}(M^{\circ})$, by polarization. Then $(\mathscr{E}, D(\mathscr{E}))$ becomes a strongly local regular Dirichlet form on $L^2(M;\m)$, since $C_c^{\infty}(M)$ is a dense set of $D(\mathscr{E})$. 
Let $k:M^{\circ}\to\R$ and $\ell:\partial M\to\R$ be continuous functions providing lower bounds on the Ricci curvature and the second fundamental form of $\partial M$, respectively. 
Suppose that $M$ is compact and $\m=\mathfrak{v}$. Then $(M,\mathscr{E},\m)$ is tamed by 
\begin{align*}
\kappa:=k\mathfrak{v}+\ell\mathfrak{s},
\end{align*}
because $\mathfrak{v},\mathfrak{s}\in S_K({\bf X})$ (see \cite[Theorem~2.36]{ERST} and \cite[Theorem~5.1]{Hsu:2001}). Then one can apply Theorem~\ref{thm:main3} to $(M,\mathscr{E},\m)$. 
Remark that Theorems~\ref{thm:LittlewoodPaley1Form} and \ref{thm:LittlewoodPaley2Form} 
are proved by 
Shigekawa~\cite[Propositions~6.2 and 6.4]{Shigekawa1} in the framework of compact smooth Riemannian manifold with boundary. 

More generally, if $M$ is regularly exhaustible, i.e., there exists an increasing sequence $(M_n)_{n\in\N}$ of domains $M_n\subset M^{\circ}$ with smooth boundary $\partial M_n$ such that $g$ is smooth on $M_n$ and the following properties hold:
\begin{enumerate}
\item[(1)] The closed sets $(\overline{M}_n)_{n\in\N}$ constitute an $\mathscr{E}$-nest for $(\mathscr{E},W^{1,2}(M))$.  
\item[(2)] For all compact sets $K\subset M^{\circ}$ there exists $N\in\N$ such tht $K\subset M_n$ for all $n\geq N$.
\item[(3)] There are lower bounds $\ell_n:\partial M_n\to\R$ for the curvature of $\partial M_n$ with $\ell_n=\ell$ on $\partial M\cap \partial M_n$ such that the distributions $\kappa_n=k\mathfrak{v}_n+\ell_n\mathfrak{s}_n$ are uniformly $2$-moderate in 
the sense that 
\begin{align}
\sup_{n\in\N}\sup_{t\in[0,1]}\sup_{x\in M_n}\E^{(n)}\left[e^{-2A_t^{\kappa_n}} \right]<\infty,\label{eq:2moderate}
\end{align}
where $\mathfrak{v}_n$ is the volume measure of $M_n$ and $\mathfrak{s}_n$ is the surface measure of $\partial M_n$.
\end{enumerate}
Suppose $\m=\mathfrak{v}$. 
Then $(M,\mathscr{E},\m)$ is tamed by $\kappa=k\mathfrak{v}+\ell\mathfrak{s}$ 
(see \cite[Theorem~4.5]{ERST}). In this case, $\kappa_{\ll}=k\mathfrak{v}$ and $\kappa_{\perp}=\ell\mathfrak{s}$, hence Theorem~\ref{thm:main3} holds for $(M,\mathscr{E},\m)$ provided 
$\kappa^+\in S_D({\bf X})$, $\kappa^-\in S_{K}({\bf X})$ for $p\in[2,+\infty[$, 
or $k$ is bounded below and $\ell\geq0$ for $p\in]1,2[$. 

Let $Y$ be the domain defined by 
\begin{align*}
Y:=\{(x,y,z)\in\R^3\mid z>\phi(\sqrt{x^2+y^2})\},
\end{align*}
where $\phi:[0,+\infty[\to[0,+\infty[$ is $C^2$ on $]0,+\infty[$ with $\phi(r):=r-r^{2-\alpha}$, $\alpha\in]0,1[$ for 
$r\in[0,1]$, $\phi$ is a constant for $r\geq2$ and $\phi''(r)\leq0$ for $r\in[0,+\infty[$. 
Let $\m_Y$ be the $3$-dimensional Lebesgue measure restricted to $Y$ and 
$\sigma_{\partial Y}$ the $2$-dimensional Hausdorff measure on $\partial Y$. Denote by 
$\mathscr{E}_Y$ the Dirichlet form on $L^2(Y;\m)$ with Neumann boundary conditions. 
The smallest eigenvalue of the second fundamental form of $\partial Y$ can be given by 
\begin{align*}
\ell(r,\phi)=\frac{\phi''(r)}{(1+|\phi'(r)|^2)^{3/2}} (\leq0),
\end{align*}
for $r\leq 1$ and $\ell=0$ for $r\geq2$. It is proved in \cite[Theorem~4.6]{ERST} that 
the Dirichlet space $(Y,\mathscr{E}_Y,\m_Y)$ is tamed by 
\begin{align*}
\kappa=\ell\sigma_{\partial Y}.
\end{align*}
In this case, the measure-valued Ricci lower bound $\kappa$ satisfies 
$|\kappa|=|\ell|\sigma_{\partial Y}\in S_K({\bf X})$ (see \cite[Lemma~2.34, Theorem~2.36, Proof of Theorem~4.6]{ERST}), $\kappa_{\ll}=0$ and $\kappa_{\perp}=\ell\sigma_{\partial Y}$. Then 
Theorem~\ref{thm:main3} holds for $(Y,\mathscr{E}_Y,\m_Y)$ and $\alpha>C_{\kappa}$ if $p\in[2,+\infty[$, or $p\in]1,2]$ under $\ell\geq0$. 
}
\end{example}
\begin{example}[Configuration space over metric measure spaces]
{\rm Let $(M,g)$ be a complete smooth Riemannian manifold without boundary.   
The configuration space $\Upsilon$ over $M$ is the space of all locally finite point measures, that is, 
\begin{align*}
\Upsilon:=\{\gamma\in\mathcal{M}(M)\mid \gamma(K)\in\N\cup\{0\}\quad \text{ for all compact sets}\quad K\subset M\}. 
\end{align*}
In the seminal paper Albeverio-Kondrachev-R\"ockner~\cite{AKR} identified a natural geometry on $\Upsilon$ by lifting the geometry of $M$ to $\Upsilon$. In particular, there exists a natural gradient $\nabla^{\Upsilon}$, divergence ${\rm div}^{\Upsilon}$ and 
Laplace operator $\Delta^{\Upsilon}$ on $\Upsilon$. It is shown in \cite{AKR} that the Poisson point measure $\pi$ on 
$\Upsilon$ is the unique (up to intensity) measure on $\Upsilon$ under which the gradient and divergence become 
dual operator in $L^2(\Upsilon;\pi)$. Hence, the Poisson measure $\pi$ is the natural volume measure on $\Upsilon$ 
and $\Upsilon$ can be seen as an infinite dimensional Riemannian manifold. The canonical Dirichlet form 
\begin{align*}
\mathscr{E}(F)=\int_{\Upsilon}|\nabla^{\Upsilon}F|_{\gamma}^2\pi(\d\gamma)
\end{align*}
 constructed in \cite[Theorem~6.1]{AKR} is quasi-regular and strongly local
 and it induces the heat semigroup $T_t^{\Upsilon}$ and a Brownian motion ${\bf X}^{\Upsilon}$ on $\Upsilon$ which can be identified with the independent infinite particle process. If ${\rm Ric}_g\geq K$ on $M$ with $K\in \R$, then $(\Upsilon,\mathscr{E}^{\Upsilon},\pi)$ is tamed by 
 $\kappa:=K\pi$ with $|\kappa|\in S_K({\bf X}^{\Upsilon})$ (see \cite[Theorem~4.7]{EKS} and \cite[Theorem~3.6]{ERST}). 
 Then one can apply Theorem~\ref{thm:main3} for $(\Upsilon,\mathscr{E}^{\Upsilon},\pi)$. 
 
 More generally, in Dello Schiavo-Suzuki~\cite{DelloSuzuki:ConfigurationI}, configuration space $\Upsilon$ over proper complete and separable metric space $(M,{\sf d})$ is considered. The configuration space $\Upsilon$ is endowed with the \emph{vage topology} $\tau_V$, induced by duality with continuous compactly supported functions on $M$, and with a reference Borel probability measure $\mu$ satisfying \cite[Assumption~2.17]{DelloSuzuki:ConfigurationI}, commonly understood as the law of a proper point process on $M$. In \cite{DelloSuzuki:ConfigurationI}, 
 they constructed the strongly local Dirichlet form $\mathscr{E}^{\Upsilon}$ defined to be the $L^2(\Upsilon;\mu)$-closure of a certain pre-Dirichlet form on a class of certain cylinder functions and prove its quasi-regularity for a wide class of measures $\mu$ and base spaces (see \cite[Proposition~3.9 and Theorem~3.45]{DelloSuzuki:ConfigurationI}). Moreover, in 
 Dello Schiavo-Suzuki~\cite{DelloSuzuki:ConfigurationII}, for any fixed $K\in \R$ they prove that a Dirichlet form $(\mathscr{E},D(\mathscr{E}))$ with its 
 carr\'e-du-champ $\Gamma$ satisfies ${\sf BE}_2(K,\infty)$ if and only if the Dirichlet form $(\mathscr{E}^{\Upsilon},D(\mathscr{E}^{\Upsilon}))$ on $L^2(\Upsilon;\mu)$ with its carr\'e-du-champ $\Gamma^{\Upsilon}$ satisfies ${\sf BE}_2(K,\infty)$.
Hence, if $(M,\mathscr{E},\m)$ is tamed by $K\m$ with $|K|\m\in S_K({\bf X})$, then  $(\Upsilon,\mathscr{E}^{\Upsilon},\mu)$ is tamed by $\kappa:=K\mu$ with $|\kappa|\in S_K({\bf X}^{\Upsilon})$. In this case, the measure-valued Ricci lower bound $\kappa$ satisfies $\kappa_{\ll}=K\mu$ and $\kappa_{\perp}=0$. 
Then Theorem~\ref{thm:main3} holds for $(\Upsilon,\mathscr{E}^{\Upsilon},\mu)$ and any $p\in]1,+\infty[$ 
under the suitable class of measures $\mu$ defined in \cite[Assumption~2.17]{DelloSuzuki:ConfigurationI}.  
}
\end{example}
\bigskip

\noindent
{\bf Acknowledgment.} 
The authors would like to thank Professor  Luca Tamanini for telling us 
his related on-going work on the boundedness of Riesz transforms in the 
framework of RCD$(K,\infty)$-spaces during the preparation of this paper. 
\bigskip

\noindent
{\bf Conflict of interest.} The authors have no conflicts of interest to declare that are relevant to the content of this article.

\bigskip
\noindent
{\bf Data Availability Statement.} Data sharing is not applicable to this article as no datasets were generated or analyzed during the current study.

\providecommand{\bysame}{\leavevmode\hbox to3em{\hrulefill}\thinspace}
\providecommand{\MR}{\relax\ifhmode\unskip\space\fi MR }
\providecommand{\MRhref}[2]{%
  \href{http://www.ams.org/mathscinet-getitem?mr=#1}{#2}
}
\providecommand{\href}[2]{#2}

\end{document}